\documentclass[10pt]{amsart}

\usepackage[latin2]{inputenc}
\usepackage{amsmath}
\usepackage{graphicx}
\usepackage{amssymb}
\usepackage{esint}
\usepackage{color}
\usepackage{amsthm}
\usepackage{epsfig}
\usepackage{enumitem}
\usepackage{mathtools}
\usepackage{esvect}
\usepackage{tikz}
\usepackage{mathrsfs}
\usetikzlibrary{calc,intersections,through,backgrounds,patterns,arrows.meta, math,through}
\usepackage[english]{babel}
\usepackage[linktocpage=true,colorlinks=true,linkcolor=Blue,citecolor=Green]{hyperref}

\newtheorem{theorem}{Theorem}
\newtheorem{proposition}[theorem]{Proposition}
\newtheorem{lemma}[theorem]{Lemma}

\newtheorem*{theorem*}{Theorem}

\theoremstyle{definition}
\newtheorem{definition}{Definition}
\newtheorem{remark}[theorem]{Remark}

\def\Xint#1{\mathchoice
{\XXint\displaystyle\textstyle{#1}}%
{\XXint\textstyle\scriptstyle{#1}}%
{\XXint\scriptstyle\scriptscriptstyle{#1}}%
{\XXint\scriptscriptstyle\scriptscriptstyle{#1}}%
\!\int}
\def\XXint#1#2#3{{\setbox0=\hbox{$#1{#2#3}{\int}$ }
\vcenter{\hbox{$#2#3$ }}\kern-.6\wd0}}

\def\dashint{\Xint-}

\definecolor{Yellow}{rgb}{0.95,0.9,0.0} 
\definecolor{Red}{rgb}{0.8,0.1,0.1}
\definecolor{Green}{rgb}{0.1,0.65,0.2}
\definecolor{Blue}{rgb}{0.1,0.1,0.8}
\definecolor{Purple}{rgb}{0.7,0.1,0.7}
\definecolor{Grey}{rgb}{0.6,0.6,0.6}

\definecolor{YELLOW}{rgb}{0.95,0.9,0.0} 
\definecolor{RED}{rgb}{0.8,0.1,0.1}
\definecolor{GREEN}{rgb}{0.25,0.65,0.1}
\definecolor{BLUE}{rgb}{0.1,0.1,0.8}
\definecolor{PURPLE}{rgb}{0.7,0.1,0.7}

\newcommand{\supp}{\operatorname{supp}} 
 
\newcommand{\sign}{\operatorname{sign}}

\newcommand{\dist}{\operatorname{dist}} 
\DeclareMathOperator*{\esssup}{ess\,sup}

\newcommand{\Rd}[1][d]{{\mathbb{R}^{#1}}}
\allowdisplaybreaks[4]

\newcommand{\dH}[1][d-1]{d\mathcal{H}^{#1}}

\newcommand{\Tgood}{\mathscr{T}_{\mathrm{good}}}
\newcommand{\Tbad}{\mathscr{T}_{\mathrm{bad}}}

\newcommand{\p}{\partial}
\newcommand{\A}{\mathscr{A}}

\renewcommand{\vec}[1]{{\operatorname{#1}}}
\newcommand{\n}{\vec{n}}
\newcommand{\ta}{\vec{t}}
\newcommand{\MC}{\vec{H}}
\newcommand{\V}{\vec{V}}
\newcommand {\jump}[1] {[\![ #1 ]\!]}

\newcommand{\pvec}{\vec{p}}

\begin{document}

\title[Weak-strong uniqueness of the Mullins--Sekerka equation]
{A weak-strong uniqueness principle for the Mullins--Sekerka equation}

\author{Julian Fischer}
\address{Institute of Science and Technology Austria (IST Austria), Am~Campus~1, 
3400 Klosterneuburg, Austria}
\email{julian.fischer@ist.ac.at}

\author{Sebastian Hensel}
\address{Hausdorff Center for Mathematics, Universit{\"a}t Bonn, Endenicher Allee 62, 53115 Bonn, Germany}
\email{sebastian.hensel@hcm.uni-bonn.de}

\author{Tim Laux}
\address{Faculty of Mathematics, University of Regensburg, Universit\"{a}tsstrasse 31, 93053 Regensburg, Germany}
\email{tim.laux@ur.de}

\author{Theresa M.\ Simon}
\address{Westf{\"a}lische Wilhelms-Universit{\"a}t M{\"u}nster, Orl{\'e}ansring 10, 48149 M{\"u}nster, Germany}
\email{theresa.simon@uni-muenster.de}

\begin{abstract}
We establish a weak-strong uniqueness principle for the two-phase Mullins--Sekerka equation 
in ambient dimension $d = 2$ and~$3$: 
As long as a classical solution to the evolution problem exists, any weak De~Giorgi type varifold solution
(see for this notion the recent work of Stinson and the second author, 
Arch.\ Ration.\ Mech.\ Anal.\ \textbf{248}, 8, 2024)
must coincide with it. In particular, in the absence of geometric singularities such weak 
solutions do not introduce a mechanism for (unphysical) non-uniqueness. We also derive a stability estimate 
with respect to changes in the data. Our method is based on the notion of relative entropies for interface 
evolution problems, a reduction argument to a perturbative graph setting,
and a stability analysis in this perturbative regime relying crucially on the gradient flow structure of
the Mullins--Sekerka equation.
\end{abstract}


\maketitle
\tableofcontents

\section{Introduction}

\subsection{Context}
This work addresses the uniqueness and stability properties of the Mullins--Sekerka 
equation, a free boundary problem modeling volume-preserving phase separation 
and coarsening processes for various physical situations. The model was
first introduced by Mullins and Sekerka in~\cite{Mullins1963} and can be derived from thermodynamic 
principles~\cite{Gurtin}. It also arises as the sharp-interface limit of the Cahn--Hilliard
equation (\cite{AlikakosBatesChen}, \cite{Chen_weak} and~\cite{Le}).

The Mullins--Sekerka equation can be viewed as a geometric third-order partial
differential equation. This viewpoint is useful for constructing classical 
solutions (\cite{Xinfu1996} and~\cite{Escher1997}),
studying long-term asymptotics via a center manifold~\cite{Escher1998}, and it also directly predicts
the typical scaling behavior of the induced coarsening process.
Naturally, as a
higher-order evolution equation, the Mullins--Sekerka equation does not admit a
comparison principle, which makes uniqueness questions subtle. Additionally, certain 
symmetric singularities can lead to physical non-uniqueness. However, it is
conjectured that before the onset of singularities, the evolution is unique.
We prove this uniqueness property in a large class of weak solutions by controlling a
suitable relative entropy functional. As our relative entropy controls the distance
of the weak and strong solutions, we additionally derive the stability of solutions
with respect to perturbations of their initial data.

Our viewpoint is variational, interpreting the Mullins--Sekerka equation as a
gradient flow. Albeit a well-appreciated fact in the community and obvious for
smooth solutions, the rigorous definition and an existence proof of weak solutions based on the gradient
flow structure has only been given recently by Stinson and the second author~\cite{Hensel2022}.
In the present work, we show that it is this very structure that guarantees the
weak-strong uniqueness of the flow. It is natural that the gradient flow structure
of the Mullins--Sekerka equation plays a crucial role in its uniqueness properties
as it also lies at the heart of many other properties of the equation, such as the
long-term asymptotics (\cite{JulinMoriniPonsiglioneSpadaro} and~\cite{otto2023convergence}) 
or its (formal) recovery as a sharp interface limit of
the Cahn--Hilliard equation~\cite{Le}.

Next to~\cite{Hensel2022}, there are several other weak solution concepts for the Mullins--Sekerka equation. 
The pioneering work~\cite{Luckhaus1995} of Luckhaus and Sturzenhecker constructs weak
solutions via an implicit time discretization. They rely on an assumption on the
convergence of energies that was later be removed by R\"{o}ger~\cite{Roeger2005}, using a result of
Sch\"{a}tzle~\cite{Schaetzle2001}. We point out that our results also apply to the well-known weak
solutions constructed by Luckhaus and Sturzenhecker satisfying a sharp energy-dissipation 
inequality.

\subsection{Weak-strong uniqueness principles for interface evolution}
In general, there are two theoretical frameworks for proving weak-strong uniqueness 
for interface evolution equations. Evans and Spruck~\cite{EvansSpruck} and independently
Chen, Giga and Goto~\cite{ChenGigaGoto} provided a notion of viscosity solution for mean curvature
flow. This powerful framework relies on the comparison principle and is therefore
only applicable to certain second-order geometric evolution equations.

More recently, the authors have developed a novel framework based on the dissipative 
nature of many interface evolution problems. The key idea is to monitor
a suitable relative entropy functional that controls the distance between a weak
and a strong solution. For the two-phase Navier--Stokes equation, this was carried
out in~\cite{FischerHensel}, and for multiphase mean curvature flow in~\cite{Fischer2020a}. 
Also for other geometric 
motions that do not admit a geometric comparison principle, this method has
been shown to be applicable, cf.\ (\cite{kroemer2022uniqueness} 
and \cite{Laux}), and it even gives new results for classical
problems which admit a comparison principle (\cite{HenselLaux} and~\cite{LauxStinsonUllrich}).

In the present work, we face several novel difficulties due to the nonlocal structure
of the Mullins--Sekerka equation. To overcome them, we first derive a preliminary
version of our stability estimate of the relative entropy. When estimating the right-hand 
side, a key challenge is to define an appropriate chemical potential to compare
the dissipation of the weak solution to the one of the strong solution. To this end, we
distinguish two cases. If at a given time, the weak solution cannot be parametrized
as a graph over the strong solution, we show that the dissipation of the weak
solution dominates all other terms. On the other hand, at times when the weak
solution can be written as a graph over the strong solution, we are in a perturbative
setting. We construct an auxiliary chemical potential which is harmonic in the two
phases of the weak solution but attains (an extension of) the mean curvature of the
strong solution as its boundary condition on the weak interface. Finally, we use the
stability of the Dirichlet-to-Neumann operator to derive a stability estimate for the
relative entropy also in the perturbative case.

\subsection{Description of the model}
Let $d \in \{2,3\}$, let $T \in (0,\infty)$ be a finite
time horizon, let $\Omega \subset \Rd$
be a bounded domain with $C^2$-boundary~$\p\Omega$, 
and let $A := \bigcup_{t\in [0,T)} A(t) {\times} \{t\}$
be the space-time track of an evolving family of subsets $A(t) \subset \Omega$,
all of which are sets of finite perimeter in~$\Rd$ such that it holds $\p A(t) \cap \Omega
= \p^* A(t) \cap \Omega$ for all $t \in [0,T)$.  
We then consider the Mullins--Sekerka problem
\begin{subequations}
\begin{align}
\label{eq:MS1}
-\Delta u &= 0
&& \text{in } \bigcup_{t \in (0,T)} (\Omega \setminus \p A(t)) \times \{t\},
\\
\label{eq:MS2}
-\jump{(\n_{\partial A} \cdot \nabla)u}\n_{\partial A} &= \V_{\partial A}
&& \text{on } \bigcup_{t \in (0,T)} (\p A(t) \cap \Omega) \times \{t\},
\\
\label{eq:MS3}
u \n_{\partial A} &= \MC_{\partial A}
&& \text{on } \bigcup_{t \in (0,T)} (\p A(t) \cap \Omega) \times \{t\},
\\
\label{eq:MS4}
(\n_{\p \Omega} \cdot \nabla)u &= 0
&& \text{on } \bigcup_{t \in (0,T)}  \p\Omega \times \{t\}.
\end{align}
\end{subequations}
Here, for all $t \in [0,T)$ we denote by $\n_{\partial A}(\cdot,t)$ 
the unit normal vector field along the interface $\p A(t) \cap \Omega$
pointing inside~$A(t)$, by $\MC_{\partial A}(\cdot,t)$ the mean curvature
vector field along $\p A(t) \cap \Omega$, by $\V_{\partial A}(\cdot,t)$
the normal velocity of $\p A(t) \cap \Omega$,
and by $\n_{\p\Omega}$ the unit normal vector field along
the domain boundary~$\p\Omega$ pointing inside~$\Omega$. Furthermore, 
for all $t \in [0,T)$ the quantity $\jump{(\n_{\partial A} \cdot \nabla)u}(\cdot,t)$
denotes the jump of $(\n_{\partial A} \cdot \nabla)u$ across $\p A(t) \cap \Omega$,
with the orientation chosen such that for all sufficiently 
regular vector fields $f\colon\overline{\Omega} \to \Rd$ and
scalar fields $\eta\colon\overline{\Omega} \to \Rd[]$ it holds
(dropping for notational simplicity the time variable)
\begin{equation}
\label{eq:IBP}
\begin{aligned}
- \int_{\Omega \setminus \p A} (\nabla \cdot f) \eta \,dx
&= \int_{\Omega} f \cdot \nabla \eta \,dx
+ \int_{\p A \cap \Omega} \jump{\n_{\partial A} \cdot f} \eta \,\dH
\\&~~~
+ \int_{\p\Omega} (\n_{\p\Omega} \cdot f) \eta \,\dH.
\end{aligned}
\end{equation}

The associated energy is given by the interface area functional
\begin{align*}
E[A(t)] := \mathcal{H}^{d-1}(\partial A(t) \cap \Omega).
\end{align*}
Assuming that the evolving interface intersects the domain
boundary only orthogonally\footnote{At the level of a strong
solution, we will always assume in the present work that the evolving interface
is compactly supported within the domain. We will therefore be able to neglect
almost all problems in connection with potential contact point dynamics.},
we may compute at least on a formal level
(dropping again for notational simplicity the time variable)
\begin{align*}
\frac{d}{dt} E[A] &= - \int_{\p A \cap \Omega} \MC_{\partial A} \cdot \V_{\partial A} \,\dH
\\&\hspace*{-0.7ex}
\stackrel{\eqref{eq:MS3}}{=}
- \int_{\p A \cap \Omega} u (\n_{\partial A} \cdot \V_{\partial A}) \,\dH
\\&\hspace*{-0.7ex}
\stackrel{\eqref{eq:MS2}}{=}
\int_{\p A \cap \Omega} u \jump{(\n_{\partial A} \cdot \nabla u)} \,\dH
\\&\hspace*{-4.2ex}
\stackrel{\eqref{eq:IBP},\eqref{eq:MS1},\eqref{eq:MS4}}{=}
- \int_{\Omega} |\nabla u|^2 \,dx,
\end{align*}
which reveals the formal gradient flow structure of the equation.
Next to this energy dissipation relation, we also have formally
\begin{align}\label{eq:time_derivative_A}
\begin{split}
\frac{d}{dt} \int_{A} \zeta \,dx &= 
\int_{A} \p_t \zeta \,dx - \int_{\p A \cap \Omega} \zeta (\n_{\partial A} \cdot \V_{\partial A}) \,\dH
\\&\hspace*{-0.7ex}
\stackrel{\eqref{eq:MS2}}{=}
\int_{A} \p_t \zeta \,dx + \int_{\p A \cap \Omega} \zeta \jump{(\n_{\partial A} \cdot \nabla u)} \,\dH
\\&\hspace*{-4.2ex}
\stackrel{\eqref{eq:IBP},\eqref{eq:MS1},\eqref{eq:MS4}}{=}
\int_{A} \p_t \zeta \,dx - \int_{\Omega} \nabla u \cdot \nabla \zeta \,dx
\end{split}
\end{align}
for all sufficiently regular test functions $\zeta \colon \overline{\Omega} \times [0,T) \to \Rd[]$.
In particular, it follows that the volume of the evolving phase is preserved under the flow:
\begin{align*}
\int_{A(t)} 1 \,dx = \int_{A(0)} 1 \,dx \quad \text{for all } t \in [0,T).
\end{align*}

Definition~\ref{DefinitionWeakSolution} below provides a suitable weak
formulation of the Mullins--Sekerka problem~\eqref{eq:MS1}--\eqref{eq:MS4}
in the framework of sets of finite perimeter and varifolds. Strong solutions of~\eqref{eq:MS1}--\eqref{eq:MS4}
will be modelled on the basis of a smoothly evolving phase $\A = \bigcup_{t \in [0,T^*)}
\A(t) \times \{t\}$ with time horizon $T^* \in (0,T)$, where for each $t \in [0,T^*)$ the 
open set $\A(t) \subset \Omega$ is assumed to consist of finitely many
simply connected components each of which is compactly supported within~$\Omega$.

\section{Main result}
Our main result establishes the uniqueness of classical solutions in a large class
of weak solutions. As our method is a relative-entropy inequality, we additionally
obtain stability for perturbations of the initial data.

\begin{theorem}[Weak-strong uniqueness and quantitative stability for the Mullins--Sekerka equation]
\label{theo:mainResult}
Let $d \in \{2,3\}$.
Consider $A = \bigcup_{t \in (0,T)} A(t) {\times} \{t\}$
and $\mu=(\mu_t)_{t \in (0,T)}$ such that $(A,\mu)$ is a De~Giorgi 
type varifold solution for the Mullins-Sekerka 
problem~\emph{\eqref{eq:MS1}--\eqref{eq:MS4}} in the sense 
of Definition~\ref{DefinitionWeakSolution} with time horizon~$T\in (0,\infty)$
and initial data~$A(0)$. Furthermore, let a smoothly evolving phase $\A = \bigcup_{t \in [0,T_*)}
\A(t) \times \{t\}$ with time horizon $T_* \in (0,T)$ be given, the motion law of which is
subject to Mullins--Sekerka flow~\emph{\eqref{eq:MS1}--\eqref{eq:MS4}}. We finally
assume that for all $t \in [0,T_*)$ the open set $\A(t) \subset \Omega$ is 
compactly supported within~$\Omega$.

There exists an error functional $(0,T_*) \ni t \mapsto E[A,\mu|\A](t) \in L^\infty_{loc}((0,T_*);[0,\infty))$ 
between the two solutions such that
\begin{equation}
\begin{aligned}
\label{eq:baseline_coercivity}
&E[A,\mu|\A](t) = 0 
\\&
\Longleftrightarrow \mathcal{L}^d\big(A(t)\Delta \A(t)\big) = 0,\,
\mu_t = \mathcal{H}^{d-1} \llcorner \big(\p^*A(t) \cap \Omega) 
\otimes (\delta_{\n_{\p^*A}(x,t)})_{x \in (\p^*A(t) \cap \Omega)},
\end{aligned}
\end{equation}
and for almost every $T' \in (0,T_*)$ there exists a constant $C(T') > 0$ such that
\begin{align}
\label{eq:stability_estimate}
E[A,\mu|\A](T') \leq E[A,\mu|\A](0) + 
C(T') \int_{0}^{T'} E[A,\mu|\A](t) \,dt.
\end{align}
In particular,
\begin{equation}
\begin{aligned}
\label{eq:uniqueness}
&E[A,\mu|\A](0) = 0 
\\&
\Longrightarrow
\mathcal{L}^d\big(A(t)\Delta \A(t)\big) = 0,\,
\mu_t = \mathcal{H}^{d-1} \llcorner \big(\p^*A(t) \cap \Omega) 
\otimes (\delta_{\n_{\p^*A}(x,t)})_{x \in (\p^*A(t) \cap \Omega)} 
\\&~~~~~\text{ for a.e.\ } t \in (0,T_*).
\end{aligned}
\end{equation}
\end{theorem}

The class of weak solutions in which our weak-strong uniqueness statement holds 
is described in detail in the following definition.
Those solutions were constructed recently by Stinson and the second author~\cite[Theorem~1 and Lemma~3]{Hensel2022}.
These De~Giorgi type varifold solutions are a suitable version of 
curves of maximal slope for general gradient flows, cf.~\cite{Ambrosio2001},  
in the context of the Mullins--Sekerka equation.
That means, they characterize the flow by an optimal energy dissipation 
inequality, see~\eqref{WeakFormDissipation} below.
It is worth mentioning that any weak solution in the sense of 
Luckhaus and Sturzenhecker~\cite{Luckhaus1995} that satisfies the sharp 
energy dissipation inequality is such a De~Giorgi type varifold 
solution (with density $\rho=1$ and equal potentials $u=w$), which means 
that our main result shows in particular the weak-strong uniqueness in this smaller class as well.

\begin{definition}[De~Giorgi type varifold solutions of the Mullins--Sekerka equation]
\label{DefinitionWeakSolution}
Let $d \in \{2,3\}$, let $T \in (0,\infty)$ be a finite time horizon,
let $\Omega \subset \Rd$ be a bounded domain with 
$C^2$-boundary~$\p\Omega$, let $A(0) \subset \Omega$ be an open subset of~$\Omega$
such that $A(0)$ has finite perimeter in~$\Rd$,
and define the associated
canonical oriented varifold $\mu_0 := \mathcal{H}^{d-1} \llcorner \big(\p^*A(0) \cap \Omega) 
\otimes (\delta_{\n_{\p^*A(0)}})_{x \in (\p^*A(0) \cap \Omega)}
\in \mathrm{M}(\overline{\Omega}{\times}\mathbb{S}^{d-1})$.

Consider a family $(A(t))_{t \in (0,T)}$ of open subsets of $\Omega$,
define~$\chi_A$ as the characteristic function associated with
the space-time set $A := \bigcup_{t \in (0,T)} A(t) {\times} \{t\}$,
and consider a family $\mu=(\mu_t)_{t \in (0,T)}$ of oriented 
varifolds $\mu_t \in \mathrm{M}(\overline{\Omega}{\times}\mathbb{S}^{d-1})$,
$t \in (0,T)$,
such that $\mu_t$ is $(d{-}1)$-integer rectifiable for a.e.\ $t \in (0,T)$.\footnote{An 
oriented varifold~$\mu$ is called integer rectifiable if its canonically associated general 
varifold~$\widehat{\mu}$, cf.\ for an example~\eqref{eq:genVarifold}, is integer rectifiable.
\label{foot:intRectOrientedVar}}
We call~$(A,\mu)$ a \emph{De~Giorgi type varifold solution of the Mullins--Sekerka
problem}~\eqref{eq:MS1}--\eqref{eq:MS4} \emph{with time horizon~$T$ and initial data~$A(0)$} if
$(\chi_A,\mu)$ is an admissible pair in the precise sense 
of~\cite[Definition~3]{Hensel2022} (with respect to $\alpha=\frac{\pi}{2})$ and
\begin{subequations}
\begin{itemize}[leftmargin=0.7cm]
\item[i)] (Finite surface area and preserved mass) It holds
\begin{align}
\label{eq:regEvolvingPhase}
\chi_A \in L^\infty(0,T;BV(\Rd;\{0,1\}))&, \,
\int_{A(t)} 1 \,dx = \int_{A(0)} 1 \,dx \quad\text{for a.e.\ } t \in (0,T);
\end{align}
\item[ii)] (Kinetic potential) There exists  
$u \in L^2(0,T;H^1(\Omega))$ with $\dashint_{\Omega} u(\cdot,t) \,dx = 0$
for a.e.\ $t \in (0,T)$ such that
\begin{equation}
\label{WeakFormMullinsSekerka}
\begin{aligned}
&\int_{A(T')} \zeta(\cdot,T') \,dx
- \int_{A(0)} \zeta(\cdot,0) \,dx
\\& 
= \int_{0}^{T'} \int_{A(t)} \partial_t \zeta \,dx dt
- \int_{0}^{T'} \int_{\Omega} \nabla u \cdot \nabla \zeta \,dx dt
\end{aligned}
\end{equation} 
for a.e.\ $T' \in (0,T)$ and all $\zeta \in C^\infty_{cpt}(\overline{\Omega}{\times}[0,T))$;
\item[iii)] (Curvature potential for mean curvature) there exist $w \in L^2(0,T;H^1(\Omega))$
and $C_w = C_w(d,\Omega,T,A(0)) > 0$ such that it holds
\begin{align}
\label{eq:L2ControlChemPotential}
\|w(\cdot,t)\|_{H^1(\Omega)} &\leq C_w (1 + \|(\nabla w)(\cdot,t)\|_{L^2(\Omega)}),
\\
\label{eq:harmonicity2}
\Delta w(\cdot,t) &= 0 \text{ in } \Omega \setminus\overline{\p^*A(t)},
\quad (\n_{\p\Omega}\cdot\nabla) w(\cdot,t) = 0 \text{ on } \p\Omega,
\end{align}
for a.e.\ $t \in (0,T)$, 
and such that~$w$ satisfies the isotropic Gibbs--Thomson law~\eqref{eq:MS3} in the sense that
\begin{align}
\label{WeakFormGibbsThomson}
\int_{\overline{\Omega} {\times} \mathbb{S}^{d-1}}
					 (\mathrm{Id} {-} \pvec\otimes\pvec) : \nabla B(x) \,d\mu_t(x,\pvec)
= \int_{A(t)} \nabla \cdot \big(w(\cdot,t) B\big) \,dx
\end{align}
for a.e.\ $t \in (0,T)$ and all $B \in C^1(\overline{\Omega};\mathbb{R}^d)$
such that $B \cdot \n_{\p\Omega} = 0$ along $\p\Omega$;
\item[iv)] (Energy dissipation inequality) and finally the energy dissipation property in the form of
\begin{align}
\label{WeakFormDissipation}
E[\mu_{T'}]
+ \int_{0}^{T'} \int_{\Omega} \frac{1}{2}|\nabla u|^2 \,dx dt
+ \int_{0}^{T'} \int_{\Omega} \frac{1}{2}|\nabla w|^2 \,dx dt
\leq E[\mu_0]
\end{align}
is satisfied for a.e.\ $T' \in (0,T)$,
where the energy functional is defined by
\begin{align}
E[\mu_t] := |\mu_t|_{\mathbb{S}^{d-1}}(\overline{\Omega})
\end{align}
for all $t \in [0,T)$, where $|\mu_t|_{\mathbb{S}^{d-1}}
\in \mathrm{M}(\overline{\Omega})$ denotes the
mass measure of the oriented varifold $\mu_t$.
\end{itemize}
\end{subequations}
\end{definition}

\begin{remark}
That $\mu_t$ is indeed integer rectifiable for a.e.\ $t\in(0,T)$ in the precise sense of
Footnote~\ref{foot:intRectOrientedVar} --- in contrast to the slightly weaker
statement of \cite[Definition~3, item~i)]{Hensel2022}, i.e.,
that the mass measure of~$\mu_t$ is $(d{-}1)$-integer rectifiable for a.e.\ $t\in(0,T)$ --- is
in fact proven in \cite[Proof of Theorem~1, Step~6]{Hensel2022}.
\end{remark}

\section{Overview of the strategy}

For the rest of the paper, we consider 
$d \in \{2,3\}$ 
and fix both a De~Giorgi type varifold solution $(A,\mu)$  
with initial data $A(0)$ and time horizon $T \in (0,\infty)$ and a smoothly evolving solution
$\A = \bigcup_{t \in [0,T_*)} \A(t) \times \{t\}$ with $T_* \in (0,T)$
for the Mullins-Sekerka problem~\eqref{eq:MS1}--\eqref{eq:MS4}.

We require
the following structural assumptions on~$\A$. For every $t \in (0,T_*)$,
the phase $\A(t)$ consists of finitely many simply connected components
(the number of which is constant in time), all of which are open sets with 
smooth boundary and being compactly supported within $\Omega$. 
We also assume that there exists $\ell \in C^1_{loc}([0,T_*);(0,1))$
such that for every $t \in (0,T_*)$ the set
$B_{\ell(t)}(\p\A(t)) = \{x \in \Rd\colon \dist(x,\p\A(t)) < \ell(t)\}$ 
is a regular tubular neighborhood for the interface $\p\A(t) = \p\A(t) \cap \Omega$ satisfying
\begin{align}
\label{eq:no_contact_points}
B_{\ell(t)}(\p\A(t)) \subset \{x \in \Omega\colon \dist(x,\p\Omega) > \ell(t)\}.
\end{align}
As usual, the nearest point projection onto $\p\A(t)$, denoted by~$P_{\p\A}(\cdot,t)$,
together with the associated signed distance function, denoted by~$s_{\p\A}(\cdot,t)$, yield
a smooth change of variables 
\begin{align}
B_{\ell(t)}(\p\A(t)) \ni x \mapsto
(P_{\p\A}(\cdot,t),s_{\p\A}(\cdot,t)) \in \p\A(t) {\times} (-\ell(t),\ell(t)),
\end{align}
where the signed distance function is oriented according to 
\begin{align}
\nabla s_{\p\A} = \n_{\p\A}.
\end{align} 
Analogously, we assume that $\ell(t)$ is an admissible tubular neighborhood width
for $\p\Omega$ for all $t \in (0,T_*)$.

In order to introduce a suitable error functional between the two solutions~$A$
and~$\A$, we will later choose extensions of the strong normal
\begin{align}
\label{eq:min_assumption1}
\xi \in W^{1,\infty}_{loc}((0,T_*);W^{1,\infty}(\overline{\Omega};\Rd))
\cap L^{\infty}_{loc}((0,T_*);W^{2,\infty}(\overline{\Omega};\Rd))
\end{align}
and the signed distance function, a weight 
(which one should think of as a truncated version of the signed distance function)
\begin{align}
\label{eq:min_assumption2}
\vartheta \in W^{1,\infty}_{loc}((0,T_*);W^{1,\infty}(\overline{\Omega}))
\cap L^{\infty}_{loc}((0,T_*);W^{2,\infty}(\overline{\Omega}))
\end{align}
satisfying as a bare minimum
\begin{align}
\label{eq:min_assumption3}
\xi &= \nabla s_{\p\A} \text{ in } \bigcup_{t\in[0,T_*)} B_{\frac{\ell(t)}{2}}(\p\A(t)) {\times} \{t\},
&&|\xi| \leq 1 \text{ on } \overline{\Omega} {\times} [0,T_*),
\\
\label{eq:min_assumption4}
\vartheta &= -\frac{s_{\p\A}}{\ell^2} \text{ in } \bigcup_{t\in[0,T_*)} B_{\frac{\ell(t)}{2}}(\p\A(t)) {\times} \{t\},
\\
\label{eq:min_assumption5}
\vartheta &< 0 \text{ in } \bigcup_{t\in[0,T_*)}\A(t) {\times} \{t\},
&&\vartheta > 0 \text{ in } \bigcup_{t\in[0,T_*)}\big(\Omega\setminus\overline{\A(t)}\big) {\times} \{t\},
\\
\label{eq:min_assumption6}
\xi &= 0 \text{ in } \bigcup_{t\in[0,T_*)}
\big(\Omega {\setminus} B_{\frac{3}{4}\ell(t)}(\p\A(t))\big) {\times} \{t\}.
\end{align}
Based on such a pair $(\xi,\vartheta)$, we define for all $t \in [0,T_*)$
\begin{align}
\label{eq:rel_entropy}
E_{rel}[A,\mu|\A](t) &:= E[\mu_t]
- \int_{\p^* A(t) \cap \Omega} \n_{\p^* A}(\cdot,t)\cdot\xi(\cdot,t) \,\dH[d-1],
\\
\label{eq:bulk_error}
E_{vol}[A|\A](t) &:= \int_{A(t) \Delta \A(t)} |\vartheta(\cdot,t)| \,dx,
\\
\label{eq:overall_error}
E[A,\mu|\A](t) &:= E_{rel}[A,\mu|\A](t) + E_{vol}[A|\A](t).
\end{align}
Defining the Radon--Nikod\'ym derivative
\begin{align}
\label{def:multiplicity}
\varrho_t := \frac{\mathcal{H}^{d-1} \llcorner (\partial^*A(t) \cap \Omega)}
{|\mu_t|_{\mathbb{S}^{d-1}} \llcorner \Omega}
 \in [0,1],
\end{align}
it follows from~\cite[Definition~3, item~ii)]{Hensel2022} 
that $\varrho_t =0$ outside $\partial^*A(t) \cap \Omega$ and
\begin{align}
\label{eq:valuesMultiplicity}
\varrho_t \in (2\mathbb{N} - 1)^{-1} \quad \text{on $\partial^*A(t) \cap \Omega$},
\end{align}
so that we may rewrite the relative entropy~\eqref{eq:rel_entropy} in the form of
\begin{equation}
\label{eq:repRelEntropy2}
\begin{aligned}
E_{rel}[A,\mu|\A](t) &= |\mu_t|_{\mathbb{S}^{d-1}}(\partial\Omega)
+ \int_{\Omega \cap \{\varrho_t \leq \frac{1}{3}\}} (1- \rho_t) \,d|\mu_t|_{\mathbb{S}^{d-1}}
\\&~~~
+ \int_{\partial^*A(t) \cap \Omega} (1 - \n_{\p^*A}\cdot\xi)(\cdot,t) \,\dH[d-1].
\end{aligned}
\end{equation}
Furthermore, due to~\eqref{eq:min_assumption6}
it follows again from~\cite[Definition~3, item~ii)]{Hensel2022} that
\begin{align}
\label{eq:repRelEntropy3}
E_{rel}[A,\mu|\A](t) &= \int_{\overline{\Omega}{\times}\mathbb{S}^{d-1}}
1 - \pvec \cdot \xi(x,t) \,d\mu_t(x,\pvec).
\end{align}

In order to suitably control the time evolution of the error
functional~$E[A,\mu|\A]$, a third construction will play a prominent
role given by a vector field
\begin{align}
\label{eq:min_assumption7}
B \in L^{\infty}_{loc}((0,T_*);W^{2,\infty}(\overline{\Omega};\Rd))
\end{align}
satisfying at least
\begin{align}
\label{eq:min_assumption8}
B &= 0 \text{ in } \bigcup_{t\in[0,T_*)}
\big(\Omega \setminus B_{\frac{3}{4}\ell(t)}(\p\A(t))\big) {\times} \{t\}.
\end{align}
One should think of~$B$ as an extension of the normal velocity
of the smoothly evolving solution~$\A$:
\begin{align}
\label{eq:min_assumption9}
B(\cdot,t) &= \V_{\p\A}(P_{\p\A}(\cdot,t),t)
\text{ on } B_{\frac{\ell(t)}{2}}(\p\A(t)),\,t \in [0,T_*).
\end{align}

With these definitions in place, we will now give an overview on our
strategy for the proof of Theorem~\ref{theo:mainResult}.

\subsection{Preliminary relative entropy inequality}
The first ingredient is a preliminary estimate for the time
evolution of the error functional~$E[A,\mu|\A]$ defined in~\eqref{eq:overall_error}, just working
under the minimal assumptions~\eqref{eq:min_assumption1}--\eqref{eq:min_assumption6} 
and~\eqref{eq:min_assumption7}--\eqref{eq:min_assumption8}
for the associated constructions $(\xi,\vartheta,B)$.
The goal here is to identify the structure of the terms we will need to estimate by the 
error functional.

\begin{lemma}[Preliminary relative entropy inequality]
\label{lem:relative_entropy_inequality} For a.e.\ $T' \in (0,T)$, it holds
\begin{equation}
\begin{aligned}
\label{eq:preliminary_estimate_rel_entropy}
&E_{rel}[A,\mu|\A](T') 
\\&
\leq E_{rel}[A,\mu|\A](0)
+ \int_{0}^{T'} R_{dissip}(t) {+} R_{\partial_t\xi}(t) {+} R_{\nabla B}(t)
{+} R_{varifold/BV}(t) \,dt
\end{aligned}
\end{equation}
and
\begin{equation}
\begin{aligned}
\label{eq:preliminary_estimate_bulk}
&E_{vol}[A|\A](T') 
= E_{vol}[A|\A](0)
+ \int_{0}^{T'} U_{dissip}(t) + U_{\partial_t\vartheta}(t) + U_{\nabla B}(t)  \,dt,
\end{aligned}
\end{equation}
where we defined
\begin{align}
\label{eq:R_dissip}
R_{dissip}(t) 
&= - \int_{\Omega} \frac{1}{2}|\nabla u(\cdot,t)|^2 \,dx
   - \int_{\Omega} \frac{1}{2}|\nabla w(\cdot,t)|^2 \,dx
\\&~~~ \nonumber
- \int_{\Omega} \nabla u(\cdot,t) \cdot \nabla (\nabla\cdot\xi)(\cdot,t) \,dx
\\&~~~ \nonumber
+ \int_{\p^*A(t) \cap \Omega} 
(B \cdot \n_{\p^* A})(\cdot,t) \big(w + (\nabla \cdot \xi)\big)(\cdot,t) \,\dH[d-1],
\\ \label{eq:R_dtxi}
R_{\partial_t\xi}(t) 
&= - \int_{\p^*A(t) \cap \Omega} 
\big(\p_t\xi {+} (B\cdot\nabla)\xi {+} (\nabla B)^\mathsf{T}\xi\big)(\cdot,t) \cdot (\n_{\p^* A} {-} \xi)(\cdot,t) \,\dH[d-1]
\\&~~~ \nonumber
- \int_{\p^*A(t) \cap \Omega} 
\big(\p_t\xi {+} (B\cdot\nabla)\xi\big)(\cdot,t) \cdot \xi(\cdot,t) \,\dH[d-1],
\\ \label{eq:R_nablaB}
R_{\nabla B}(t) &= 
- \int_{\p^*A(t) \cap \Omega} 
\big((\n_{\p^* A} {-} \xi) \otimes (\n_{\p^* A} {-} \xi)\big(\cdot,t) : \nabla B(\cdot,t) \,\dH[d-1]
\\&~~~ \nonumber
- \int_{\p^*A(t) \cap \Omega} 
(\n_{\p^* A} \cdot \xi - 1)(\cdot,t) (\nabla \cdot B)(\cdot,t) \,\dH[d-1],
\\ \label{eq:R_varifoldBV}
R_{varifold/BV}(t) &=
-\int_{\p^*A(t) \cap \Omega} 
(\mathrm{Id} {-} \n_{\p^* A} \otimes \n_{\p^* A} )(\cdot,t) : \nabla B(\cdot,t) \,\dH[d-1]
\\&~~~ \nonumber
+ \int_{\overline{\Omega} {\times} \mathbb{S}^{d-1}}
					 (\mathrm{Id} {-} \pvec\otimes \pvec) : \nabla B(\cdot,t) \,d\mu_t(\cdot,\pvec)
\end{align}
as well as
\begin{align}
\label{eq:A_dissip}
U_{dissip}(t) 
&= - \int_{\Omega} \nabla u(\cdot,t) \cdot \nabla \vartheta(\cdot,t) \,dx
\\&~~~ \nonumber
+ \int_{\p^*A(t)\cap\Omega} \vartheta(\cdot,t) (\n_{\p^* A}\cdot B)(\cdot,t) \,\dH[d-1],
\\ \label{eq:A_dtvartheta}
U_{\partial_t\vartheta}(t) 
&= \int_{\Omega} (\chi_A {-} \chi_{\A})(\cdot,t)
(\partial_t\vartheta {+} (B\cdot\nabla)\vartheta)(\cdot,t) \,dx,
\\ \label{eq:A_nablaB}
U_{\nabla B}(t) &= 
\int_{\Omega} (\chi_A {-} \chi_{\A})(\cdot,t) \vartheta(\cdot,t) (\nabla\cdot B)(\cdot,t) \,dx.
\end{align}
\end{lemma}

\subsection{Time splitting argument: Definition of good and bad times}
In order to suitably estimate the right hand sides of~\eqref{eq:preliminary_estimate_rel_entropy}
and~\eqref{eq:preliminary_estimate_bulk}, respectively, we essentially aim to reduce
the task to a perturbative setting, where the interface of the weak solution can
be represented as a graph over the interface of the smooth solution (with arbitrarily small $C^1$ norm).
This reduction argument is implemented based on a case distinction for times~$t \in (0,T)$,
with the non-perturbative regime corresponding to either disproportionally large dissipation
of energy of the weak solution or a lower bound for the error~$E[A,\mu|\A]$. We formalize this as follows.

Let $\Lambda \in (0,\infty)$ and $M \in (1,\infty)$ constants we will determine later. We
then construct a disjoint decomposition
\begin{align}
(0,T_*) = \Tbad(\Lambda,M) \cup \Tgood(\Lambda,M)
\end{align}
by means of 
\begin{align}
\label{eq:def_bad_times}
\Tbad(\Lambda,M) &:= \Tbad^{(1)}(\Lambda) \cup \Tbad^{(2)}(\Lambda,M),
\\
\label{eq:def_bad_times_dissipation}
\Tbad^{(1)}(\Lambda) &:= \Big\{t \in (0,T_*) \colon \int_{\Omega} \frac{1}{2}|\nabla w(\cdot,t)|^2 \,dx > \Lambda\Big\},
\\
\label{eq:def_bad_times_rel_entropy}
\Tbad^{(2)}(\Lambda,M) &:= \Big\{t \in (0,T_*) \setminus \Tbad^{(1)}(\Lambda) \colon 
E[A,\mu|\A](t) > \frac{\ell(t)^{d{-}1}}{M}\Big\},
\\
\label{eq:def_good_times}
\Tgood(\Lambda,M) &:= (0,T_*) \setminus \Tbad(\Lambda,M).
\end{align}

The merit of these definitions is that on one side, the estimate
for the right hand sides of~\eqref{eq:preliminary_estimate_rel_entropy}
and~\eqref{eq:preliminary_estimate_bulk} is easily closed for times $t\in \Tbad(\Lambda,M)$,
whereas on the other side for times $t\in \Tgood(\Lambda,M)$ one can actually \textit{prove}
that one has to be in a perturbative setting; cf.\ Proposition~\ref{prop:perturbative_graph_regime} below.
The latter is a powerful tool to estimate the (not yet fully unraveled) ``non-local contributions'' in the
right hand sides of~\eqref{eq:preliminary_estimate_rel_entropy}
and~\eqref{eq:preliminary_estimate_bulk} originating from the 
gradient flow structure 
$(\V_{\p A},\MC_{\p A}) \in H^{-1/2}_{MS}(\p A \cap \Omega) 
{\times} H^{1/2}_{MS}(\p A \cap \Omega)$ 
of Mullins--Sekerka flow~\eqref{eq:MS1}--\eqref{eq:MS4}. 

\subsection{Stability estimates at bad times}
In a first step, we will now estimate the right hand sides of~\eqref{eq:preliminary_estimate_rel_entropy}
and~\eqref{eq:preliminary_estimate_bulk} in the non-perturbative setting.

\begin{lemma}[Stability estimates at bad times]
\label{lem:stability_estimate_bad_times}
For all $T' \in (0,T_*)$ there exists $\Lambda = \Lambda(\mathscr{A},A(0),T') \in (0,\infty)$
such that for all $M \in (1,\infty)$
and all $t \in \Tbad(\Lambda,M) \cap (0,T')$, it holds
\begin{align}
\nonumber
&R_{dissip}(t) + R_{\partial_t\xi}(t) + R_{\nabla B}(t)
+ R_{varifold/BV}(t)
\\&~ \label{eq:rel_entropy_bad_times_estimate}
\leq
\begin{cases}
- \int_{\Omega} \frac{3}{8}|\nabla u(\cdot,t)|^2 
+  \frac{3}{8}|\nabla w(\cdot,t)|^2 \,dx,
& t \in \Tbad^{(1)}(\Lambda),
\\[1ex]
- \int_{\Omega} \frac{3}{8}|\nabla u(\cdot,t)|^2 
+ \frac{1}{2}|\nabla w(\cdot,t)|^2 \,dx
+ \frac{\Lambda}{4} \frac{M}{\ell(t)^{d{-}1}} E[A,\mu|\A](t), 
& \text{else},
\end{cases} 
\end{align}
and
\begin{align}
\nonumber
&U_{dissip}(t) + U_{\partial_t\vartheta}(t) + U_{\nabla B}(t)
\\&~ \label{eq:bulk_bad_times_estimate}
\leq
\begin{cases}
 \int_{\Omega} \frac{1}{8}|\nabla u(\cdot,t)|^2 
+ \frac{1}{8}|\nabla w(\cdot,t)|^2 \,dx,
& t \in \Tbad^{(1)}(\Lambda),
\\[1ex]
\int_{\Omega} \frac{1}{8}|\nabla u(\cdot,t)|^2  \,dx
+ \frac{\Lambda}{4} \frac{M}{\ell(t)^{d{-}1}} E[A,\mu|\A](t),
& \text{else}.
\end{cases} 
\end{align}
\end{lemma}

\subsection{Stability estimates at good times}
As already said, a key ingredient for our approach to weak-strong uniqueness
and stability of the Mullins--Sekerka problem~\eqref{eq:MS1}--\eqref{eq:MS4}
consists of a reduction argument to the interface of the weak solution being
given as a graph over the interface of the smoothly evolving solution
with arbitrarily small $C^1$ norm. It is a non-trivial result that
our definition~\eqref{eq:def_good_times} of good times is actually
a sufficient condition to imply such a perturbative graph setting, an implication which we
formalize in the following result. 

\begin{proposition}[Perturbative graph setting at good times]
\label{prop:perturbative_graph_regime}
Fix $T' \in (0,T_*)$, and
let $\Lambda = \Lambda(\mathscr{A},A(0),T') \in (0,\infty)$ be the constant 
from Lemma~\ref{lem:stability_estimate_bad_times}. For every
$C \in (1,\infty)$ there exists a constant 
\begin{align}
\label{eq:choiceSmallness}
M = M(\Omega,A(0),\mathscr{A},C,\Lambda,T') \in (1,\infty)
\end{align} 
such that for almost every
$t \in \Tgood(\Lambda,M) \cap (0,T')$ there exists a height function 
\begin{align}
\label{eq:regularity_graph}
h(\cdot,t) \in (C^{1,1{-}\frac{d-1}{4}} \cap H^{2})\big(\partial\A(t)\big)
\end{align}
such that
\begin{align}
\label{eq:graph_representation}
\partial^*A(t) &= \big\{x {+} h(x,t) \n_{\partial\A}(x,t)\colon x \in \partial\A(t)\big\}
\subset \{x\in\Omega\colon\dist(x,\p\Omega) {>} \ell(t)\},
\\
\label{eq:representation_varifold}
|\mu_t|_{\mathbb{S}^{d-1}} 
&= \mathcal{H}^{d-1} \llcorner \partial^*A(t),
\end{align}
and
\begin{align}
\label{eq:perturbative_regime1}
\frac{1}{\ell(t)}\|h(\cdot,t)\|_{L^\infty(\partial\A(t))}
+ \|\nabla_{\p\A}h(\cdot,t)\|_{L^\infty(\partial\A(t))} &\leq \frac{1}{16C}.
\end{align}
\end{proposition}

We continue on how to close the stability estimate for times
at which the
previous result holds, treating it as a black-box.
The first essential step consists of suitably rewriting
the terms constituting~$R_{dissip}$ and~$U_{dissip}$ 
from~\eqref{eq:R_dissip} and~\eqref{eq:A_dissip}, respectively.

To this end, one may draw inspiration from the
gradient flow structure of Mullins--Sekerka flow~\eqref{eq:MS1}--\eqref{eq:MS4},
which, for the smoothly evolving~$\p\A$, is given as follows.
Denote by $H^{1/2}_{MS}(\p\A)$ the set of all maps $f\colon\p\A\to\mathbb{R}$
arising as traces of maps from $H^1(\A)$ such that $\dashint_{\p\A} f \,\dH[d-1]=0$. 
One then defines a Hilbert space structure on~$H^{1/2}_{MS}(\p\A)$ by means of
the minimal Dirichlet energy extension to~$\Omega$:
\begin{align}
\label{eq:def_Hilbert_space_Mullins_Sekerka} 
\langle f, \widetilde{f} \rangle_{H^{1/2}_{MS}(\p\A)} 
:= \int_{\Omega} \nabla u_f \cdot \nabla u_{\widetilde{f}} \,dx, 
\end{align}
where, for given $f \in H^{1/2}_{MS}(\p\A)$,
the associated potential $u_f \in H^1(\Omega)$ satisfies
\begin{align*}
\Delta u_f &= 0 &&\text{in } \Omega\setminus\p\A,
\\
u_f &= f &&\text{on } \p\A,
\\
(\n_{\p\Omega}\cdot\nabla)u_f &= 0 &&\text{on } \p\Omega. 
\end{align*}
Denoting by $H^{-1/2}_{MS}(\p\A)$ the dual of $H^{1/2}_{MS}(\p\A)$,
it is immediate to recognize that the Riesz isomorphism 
$\mathcal{R}\colon H^{-1/2}_{MS}(\p\A) \to H^{1/2}_{MS}(\p\A)$ is formally realized
through the inverse of the (two-phase) Dirichlet-to-Neumann operator:
\begin{align*}
\langle F, f \rangle_{H^{-1/2}_{MS}(\p\A) \times H^{1/2}_{MS}(\p\A)}
= \langle \mathcal{R}(F), f \rangle_{H^{1/2}_{MS}(\p\A)}
= -\int_{\p\A} (\n_{\p\A}\cdot\jump{\nabla u_{\mathcal{R}(F)}}) f \,\dH[d-1],
\end{align*}
as can be seen similarly to the computation \eqref{eq:time_derivative_A}.
Note that the integration by parts is rigorous for those $F$
such that $\mathcal{R}(F)$ is the trace of a, say, $H^2(\A)$ function.

The idea for the rewriting of, say, $R_{dissip}$ is now as follows.
What one likes to achieve in~\eqref{eq:R_dissip} is to complete
squares in the two quadratic terms originating from the energy dissipation
inequality of a weak solution such that the resulting quadratic terms
penalize the \textit{difference of curvatures}
of the two evolving geometries as measured through the correct
gradient flow norm. This is indeed in direct analogy of 
how we proceeded in the case of (multiphase) mean curvature
flow being the $L^2(\p\A)$ gradient flow of the perimeter;
cf.\ Section~3.1 of \cite{Fischer2020a}. The next result
implements this procedure for Mullins--Sekerka flow.

\begin{lemma}[Structure of dissipative terms in perturbative regime]
\label{lem:post_processed_relative_entropy_inequality}
Fix $t \in (0,T_*)$, and assume that for~$t$ the 
conclusions~\emph{\eqref{eq:regularity_graph}--\eqref{eq:perturbative_regime1}}
of Proposition~\ref{prop:perturbative_graph_regime} hold true.
Furthermore, assume that~$t$ is a Lebesgue point for~$s\mapsto u(\cdot,s)$ 
in the sense that
\begin{align}
\label{PropertyLebesguePoint}
\lim_{\tau\searrow 0} \dashint_{t-\tau}^{t+\tau}
\int_{\Omega} |\nabla u-\nabla u(\cdot,t)|^2 \,dx ds &= 0.
\end{align}

\begin{itemize}[leftmargin=0.7cm]
\item[i)] Defining a chemical potential~$\widetilde w = \widetilde w(\cdot,t) \in H^1(\Omega)$ 
by means of
\begin{align}
\label{eq:auxChemPotential1}
- \Delta \widetilde w &= 0 
&& \text{in } \Omega \setminus \p^* A(t),
\\ \label{eq:auxChemPotential2}
\mathrm{tr}_{\p^*A(t)}\widetilde w &= - (\nabla \cdot \xi)(\cdot,t)
&& \text{on } \p^* A(t) \cap \Omega,
\\ \label{eq:auxChemPotential3}
(\n_{\p\Omega} \cdot \nabla) \widetilde w &= 0
&& \text{on } \p\Omega,
\end{align} 
we obtain the following alternative representation for~$R_{dissip}$ from~\eqref{eq:R_dissip}:
\begin{align}
\label{eq:R_dissip_combined}
R_{dissip}(t) &= - \int_{\Omega} \frac{1}{2}|(\nabla u {-} \nabla\widetilde w)(\cdot,t)|^2 \,dx
								 - \int_{\Omega} \frac{1}{2}|(\nabla w {-} \nabla\widetilde w)(\cdot,t)|^2 \,dx
\\&~~~ \nonumber
+ \int_{\p^*A(t) \cap \Omega} (w - \widetilde{w})(\cdot,t)
\big((B \cdot \n_{\p^* A}) + \jump{(\n_{\p^* A} \cdot \nabla)\widetilde w}\big)(\cdot,t) \,\dH[d-1].
\end{align} 
\item[ii)] Similarly, by defining another chemical 
potential $\widetilde w_\vartheta = \widetilde w_\vartheta(\cdot,t)$
through the problem
\begin{align}
\label{eq:auxChemPotentialWeight1}
- \Delta \widetilde w_\vartheta &= 0 
&& \text{in } \Omega \setminus \p^* A(t),
\\ \label{eq:auxChemPotentialWeight2}
\mathrm{tr}_{\p^*A(t)}\widetilde w_\vartheta &= \vartheta(\cdot,t)
&& \text{on } \p^* A(t) \cap \Omega,
\\ \label{eq:auxChemPotentialWeight3}
(\n_{\p\Omega} \cdot \nabla) \widetilde w_\vartheta &= 0
&& \text{on } \p\Omega,
\end{align} 
we obtain that~$U_{dissip}$ from~\eqref{eq:A_dissip} can be rewritten as follows:
\begin{align}
\label{eq:A_dissip_combined}
U_{dissip}(t) &= 
\int_{\p^*A(t)\cap\Omega} \vartheta(\cdot,t)
\big((B\cdot\n_{\p^* A}) + \jump{(\n_{\p^* A}\cdot\nabla)\widetilde w}\big)(\cdot,t) \,\dH[d-1]
\\&~~~ \nonumber
+ \int_{\p^*A(t)\cap\Omega} (u {-} \widetilde w)(\cdot,t) 
\jump{(\n_{\p^* A}\cdot\nabla)\widetilde w_\vartheta}(\cdot,t) \,\dH[d-1].
\end{align} 
\end{itemize}
\end{lemma}

The second, and most straightforward, step is to estimate the
terms on the right hand sides of~\eqref{eq:preliminary_estimate_rel_entropy} 
and~\eqref{eq:preliminary_estimate_bulk} that only feature local terms.
That is, they only depend on the kinematics of the problem and do not need to be treated using the nonlocal gradient flow structure of Mullins-Sekerka flow via the dissipative terms
from~\eqref{eq:R_dissip_combined}.

\begin{lemma}[Stability estimate for local terms in perturbative regime]
\label{lem:stability_estimate_local_terms}
Fix $t \in (0,T_*)$, and assume that for~$t$ the 
conclusions~\emph{\eqref{eq:regularity_graph}--\eqref{eq:perturbative_regime1}}
of Proposition~\ref{prop:perturbative_graph_regime} hold true.
For any collection~$(\xi,\vartheta,B)$ satisfying only
the assumptions~\emph{\eqref{eq:min_assumption1}--\eqref{eq:min_assumption6}} 
and~\emph{\eqref{eq:min_assumption7}--\eqref{eq:min_assumption9}}, 
we have the estimates
\begin{align}
\label{eq:estimate_Rdtxi}
R_{\partial_t\xi}(t) &= 0,
\\
\label{eq:estimate_RnablaB}
R_{\nabla B}(t) &\leq \big\|(\nabla\cdot B)(\cdot,t)\big\|_{L^\infty(B_{\ell(t)/2}(\p\A(t)))} 
											E_{rel}[A,\mu|\A](t)		
\\&~~~ \nonumber
+ \big\|(\nabla B)(\cdot,t){+}
(\nabla B)^\mathsf{T}(\cdot,t)\big\|_{L^\infty(B_{\ell(t)/2}(\p\A(t)))}
E_{rel}[A,\mu|\A](t),
\\
\label{eq:estimate_RvarifoldBV}
R_{varifold/BV}(t) &\leq \big\|(\nabla\cdot B)(\cdot,t)\big\|_{L^\infty(B_{\ell(t)/2}(\p\A(t)))}
E_{rel}[A,\mu|\A](t)
\\&~~~ \nonumber 
+ 4 \big\|(\nabla B)(\cdot,t){+}
(\nabla B)^\mathsf{T}(\cdot,t)\big\|_{L^\infty(B_{\ell(t)/2}(\p\A(t)))} 
E_{rel}[A,\mu|\A](t),
\end{align}
as well as
\begin{align}
\label{eq:estimate_Adtvartheta}
U_{\partial_t\vartheta}(t) &\leq \frac{|\partial_t \ell(t)|}{\ell(t)}E_{vol}[A|\A](t),
\\
\label{eq:estimate_AnablaB}
U_{\nabla B}(t) &\leq \big\|(\nabla\cdot B)(\cdot,t)\big\|_{L^\infty(B_{\ell(t)/2}(\p\A(t)))} E_{vol}[A|\A](t).
\end{align}
\end{lemma}

In a third and final step, we have to provide an estimate for the remaining terms
given by~\eqref{eq:R_dissip_combined} and~\eqref{eq:A_dissip_combined}.
Focusing for the moment on~$R_{dissip}$, the idea again draws heavily from
the gradient flow structure of Mullins--Sekerka flow and goes as follows. 

First, as already discussed above, the motivation for the argument leading
to~\eqref{eq:R_dissip_combined} is that $-\int_{\Omega}\frac{1}{2}|\nabla (w{-}\widetilde{w})|^2\,dx$
penalizes the difference of curvatures of the two solutions as measured through
the gradient flow structure associated with Mullins--Sekerka flow. Indeed,
$-\nabla\cdot\xi$ serves as a proxy for~$H_{\p\A}$ away from~$\p\A$ since
$\xi$ represents an extension of the unit normal vector field~$\n_{\p\A}$,
and $w - \widetilde{w}$ is by construction of~$\widetilde{w}$ and the Gibbs--Thomson law~\eqref{WeakFormGibbsThomson}
the minimal Dirichlet extension of $H_{\p^*A\cap\Omega} + \nabla\cdot\xi$
away from the interface $\p^*A\cap\Omega$. In the perturbative graph setting
of Proposition~\ref{prop:perturbative_graph_regime}, one may thus expect that to leading order
(and up to terms controlled by the interface error~$E[A,\mu|\A]$)
\begin{align}
\label{eq:strategy_dissipative_terms1}
-\int_{\Omega}\frac{1}{2}|\nabla (w{-}\widetilde{w})|^2\,dx \approx
- \frac{1}{2}\|\Delta_{\p\A} h\|^2_{H^{1/2}_{MS}(\p\A)} 
= -\frac{1}{2}\|\nabla u_{\Delta_{\p\A} h}\|^2_{L^2(\Omega)},
\end{align}
where as before $u_f \in H^1(\Omega)$ denotes
the minimal Dirichlet energy extension of $f\in H^{1/2}_{MS}(\p\A)$ to~$\Omega$.

Second, since $B$ is an extension of the normal velocity of
the smoothly evolving~$\p\A$ (i.e., an extension of the jump of the 
Neumann data of the chemical potential of the strong solution, the latter
being equal to~$H_{\p\A}$ along~$\p\A$), in the perturbative graph 
setting of Proposition~\ref{prop:perturbative_graph_regime}, one should think of
the term $\n_{\p^*A}\cdot(B + \jump{\nabla \widetilde{w}})$ along $\p^*\A$
as a proxy for $\n_{\p\A}\cdot\jump{\nabla u_{(-\mathrm{tr}(\mathrm{L}^2_{\p\A}) h 
- \dashint_{\p\A} -\mathrm{tr}(\mathrm{L}^2_{\p\A}) h\,\dH[d-1])}}$
along~$\p\A$, since to leading order the associated Dirichlet data is given by 
$H_{\p\A} + \nabla\cdot\xi|_{\p^*A\cap\Omega} 
\approx -\mathrm{tr}(\mathrm{L}^2_{\p\A}) h$,
where we denote by $\mathrm{L}_{\p\A}$ the Weingarten tensor of~$\p\A$ 
(see, e.g., \cite[Equation~(2.7), page~47]{Pruess2016}).
Hence, by the very definition~\eqref{eq:def_Hilbert_space_Mullins_Sekerka} of the 
Hilbert space structure on~$H^{1/2}_{MS}(\p\A)$ one may expect that to leading order
(and up to terms controlled by the interface error~$E[A,\mu|\A]$)
\begin{equation}
\begin{aligned}
\label{eq:strategy_dissipative_terms2}
&\int_{\p^*A(t) \cap \Omega} (w - \widetilde{w})(\cdot,t)
\big((B \cdot \n_{\p^* A}) + \jump{(\n_{\p^* A} \cdot \nabla)\widetilde w}\big)(\cdot,t) \,\dH[d-1]
\\&~~
\approx \Big\langle \Delta_{\p\A} h, 
-\mathrm{tr}(\mathrm{L}^2_{\p\A}) h - 
\dashint_{\p\A} -\mathrm{tr}(\mathrm{L}^2_{\p\A}) h\,\dH[d-1] \Big\rangle_{H^{1/2}_{MS}(\p\A)} 
\\&~~
= \int_{\Omega} \nabla u_{\Delta_{\p\A} h} \cdot
\nabla u_{(-\mathrm{tr}(\mathrm{L}^2_{\p\A}) h 
- \dashint_{\p\A} -\mathrm{tr}(\mathrm{L}^2_{\p\A}) h\,\dH[d-1])} \,dx.
\end{aligned}
\end{equation}

Hence, in view of the representation~\eqref{eq:R_dissip_combined} for~$R_{dissip}$,
one may expect in the perturbative graph 
setting of Proposition~\ref{prop:perturbative_graph_regime} to obtain
to leading order
(and up to terms controlled by the interface error~$E[A,\mu|\A]$)
an estimate of the form
\begin{equation}
\begin{aligned}
\label{eq:strategy_dissipative_terms3}
R_{dissip} &\approx -\int_{\Omega}\frac{1}{2}|\nabla (u{-}\widetilde{w})|^2\,dx
-\int_{\Omega}\frac{3}{8}|\nabla (w{-}\widetilde{w})|^2\,dx
\\&~~~~
+ \widetilde{C} \bigg\|-\mathrm{tr}(\mathrm{L}^2_{\p\A}) h - 
\dashint_{\p\A} -\mathrm{tr}(\mathrm{L}^2_{\p\A}) h\,\dH[d-1]\bigg\|^2_{H^{1/2}_{MS}(\p\A)}
\end{aligned}
\end{equation}
for some constant~$\widetilde{C} > 0$.

Finally, in the perturbative graph 
setting of Proposition~\ref{prop:perturbative_graph_regime},
it is not hard to show 
(by a change to tubular neighborhood coordinates for $E_{vol}$,
and by expressing $\n_{\p^*A} \cdot \xi$ to leading order in
terms of the graph function~$h$ and its derivative~$h'$ for $E_{rel}$)
that our error functional behaves to leading order as 
\begin{align}
\label{eq:strategy_dissipative_terms4}
E[A,\mu|\A] \approx \left( \frac{1}{2}\bigg\|\frac{h}{\ell}\bigg\|^2_{L^2(\p\A)} 
+ \frac{1}{2}\|h'\|^2_{L^2(\p\A)} \right).
\end{align}
Hence, once one is provided with an interpolation estimate of the form
\begin{align}
\label{eq:strategy_dissipative_terms5}
\|f\|^2_{H^{1/2}_{MS}(\p\A)} \lesssim \|f\|_{H^1(\p\A)}\|f\|_{H^2(\p\A)},
\end{align}
one may expect based on~\eqref{eq:strategy_dissipative_terms3}--\eqref{eq:strategy_dissipative_terms5}
\begin{align}
\label{eq:strategy_dissipative_terms6}
R_{dissip} &\leq -\int_{\Omega}\frac{1}{2}|\nabla (u{-}\widetilde{w})|^2\,dx
-\int_{\Omega}\frac{1}{4}|\nabla (w{-}\widetilde{w})|^2\,dx + \widetilde{C}E[A,\mu|\A],
\end{align}
which renders the following rigorous version of this heuristic conceivable. 
Starting point for its proof will be a domain mapping argument
(not surprisingly in the form of a Hanzawa transformation), cf.\ the more
technical discussion of Subsection~\ref{subsec:estimate_dissipative_terms}.

\begin{proposition}[Stability estimates for non-local terms in perturbative regime]
\label{prop:stability_estimate_nonlocal_terms}
Let $\delta \in (0,1)$, fix $t \in (0,T_*)$, and assume that for~$t$ the 
conclusions~\emph{\eqref{eq:regularity_graph}--\eqref{eq:perturbative_regime1}}
of Proposition~\ref{prop:perturbative_graph_regime} hold true.
In particular, recall that for a given $C \in (1,\infty)$ it holds
\begin{align}
\label{eq:perturbative_regime1_again}
\frac{1}{\ell(t)}\|h(\cdot,t)\|_{L^\infty(\partial\A(t))}
+ \|\nabla_{\p\A}h(\cdot,t)\|_{L^\infty(\partial\A(t))} &\leq \frac{1}{16C}.
\end{align}
Furthermore, denote by $\Lambda \in (1,\infty)$ and $M \in (1,\infty)$
two constants such that
\begin{align}
\label{eq:bounded_dissipation}
\int_{\Omega} \frac{1}{2}|\nabla w(\cdot,t)|^2 &\leq \Lambda,
\\
\label{eq:small_error}
E[A,\mu|\A](t) &\leq \frac{\ell(t)^{d{-}1}}{M}.
\end{align}
Finally, fix~$(\xi,\vartheta,B)$ satisfying 
the assumptions~\emph{\eqref{eq:min_assumption1}--\eqref{eq:min_assumption6}} 
and~\emph{\eqref{eq:min_assumption7}--\eqref{eq:min_assumption9}},
and let the chemical potentials~$\widetilde w$ and $\widetilde w_{\vartheta}$ be 
defined according to Lemma~\ref{lem:post_processed_relative_entropy_inequality}.

Then, one may choose $C \gg_{\delta} 1$
and $M \gg_{\delta} 1$ locally uniformly in $[0,T_*)$ such that
there exists a constant $ \widetilde C = \widetilde C(\A,\Omega,\Lambda, \ell, M, T^\ast) \in (1,\infty)$ 
such that
\begin{equation}
\begin{aligned}
&\int_{\p^*A(t) \cap \Omega} (w - \widetilde{w})(\cdot,t)
\big((B \cdot \n_{\p^* A}) + \jump{(\n_{\p^* A} \cdot \nabla)\widetilde w}\big)(\cdot,t) \,\dH[d-1]
\\&~ \label{eq:estimate_R_dissip}
\leq \delta \int_{\Omega} |\nabla (w {-} \widetilde{w})(\cdot,t)|^2 \,dx
+ \frac{\widetilde C}{\delta} E[A,\mu|\A](t),
\end{aligned}
\end{equation}
and
\begin{equation}
\begin{aligned}
&\int_{\p^*A(t)\cap\Omega} (u {-} \widetilde w)(\cdot,t) 
\jump{(\n_{\p^* A}\cdot\nabla)\widetilde w_\vartheta}(\cdot,t) \,\dH[d-1]
\\&~ \label{eq:estimate_A_dissip1}
\leq \delta \int_{\Omega} |\nabla (u {-} \widetilde{w})(\cdot,t)|^2 \,dx
+ \delta \int_{\Omega} |\nabla (w {-} \widetilde{w})(\cdot,t)|^2 \,dx
+ \frac{\widetilde C}{\delta} E[A,\mu|\A](t),
\end{aligned}
\end{equation}
as well as
\begin{equation}
\begin{aligned}
&\int_{\p^*A(t)\cap\Omega} \vartheta(\cdot,t)
\big((B\cdot\n_{\p^* A}) + \jump{(\n_{\p^* A}\cdot\nabla)\widetilde w}\big)(\cdot,t) \,\dH[d-1]
 \label{eq:estimate_A_dissip2}
\\&~
\leq \delta \int_{\Omega} |\nabla (w {-} \widetilde{w})(\cdot,t)|^2 \,dx
+ \frac{\widetilde C}{\delta} E[A,\mu|\A](t),.
\end{aligned}
\end{equation} 
\end{proposition}

\section{Proof of Theorem~\ref{theo:mainResult}}
We proceed in three steps.

\textit{Step~1: Construction of triple $(\xi,\vartheta,B)$
satisfying~\emph{\eqref{eq:min_assumption1}--\eqref{eq:min_assumption6}} 
and~\emph{\eqref{eq:min_assumption7}--\eqref{eq:min_assumption9}}.}
Let $\bar\eta \in C^\infty_{cpt}(\mathbb{R};[0,1])$
be such that $\supp\bar\eta \subset [-3/4,3/4]$, $\bar\eta \equiv 1$
on $[-1/2,1/2]$ and $|\bar\eta'| \leq 8$. We then define
the vector fields~$\xi$ and~$B$ by means of
\begin{align}
\label{eq:construction_xi}
\xi(x,t) &:= \bar{\eta}\Big(\frac{s_{\p\A}(x,t)}{\ell(t)}\Big)\nabla s_{\p\A}(x,t),
\\
\label{eq:construction_B}
B(x,t) &:= \bar{\eta}\Big(\frac{s_{\p\A}(x,t)}{\ell(t)}\Big)
\V_{\p\A}(P_{\p\A}(x,t),t),
\end{align}
where $(x,t) \in \overline{\Omega} {\times} [0,T_*)$.
We further choose $\bar\vartheta\in C^\infty(\mathbb{R};[-1,1])$
such that $\bar\vartheta(s)=s$ on $[-1/2,1/2]$, $\bar\vartheta \equiv -1$
on $(-\infty,-1]$, $\bar\vartheta \equiv 1$ on $[1,\infty)$, and
finally $\bar\vartheta' > 0$ in $(-1,1)$. Based on this auxiliary map,
we define the weight~$\vartheta$ as follows:
\begin{align}
\label{eq:construction_vartheta}
\vartheta(x,t) := -\frac{1}{\ell(t)}
\bar\vartheta\Big(\frac{s_{\p\A}(x,t)}{\ell(t)}\Big),
\end{align}
where $(x,t) \in \overline{\Omega} {\times} [0,T_*)$.
The properties~\eqref{eq:min_assumption1}--\eqref{eq:min_assumption6}
and~\eqref{eq:min_assumption7}--\eqref{eq:min_assumption9}
are immediate consequences of the definitions~\eqref{eq:construction_xi}--\eqref{eq:construction_vartheta}
and the regularity of the smoothly evolving phase~$\A$.

\textit{Step~2: Proof of stability estimate~\eqref{eq:stability_estimate}.}
Given the triple~$(\xi,\vartheta,B)$ from the previous step,
define the functional $(0,T_*) \ni t \mapsto E[A,\mu|\A](t) \in L^\infty_{loc}((0,T_*);[0,\infty))$ 
by~\eqref{eq:overall_error}. 

Starting point for the proof of the stability estimate~\eqref{eq:stability_estimate}
are the preliminary stability estimates~\eqref{eq:preliminary_estimate_rel_entropy}
and~\eqref{eq:preliminary_estimate_bulk} from Lemma~\ref{lem:relative_entropy_inequality}.
Post-processing these by means of Lemma~\ref{lem:stability_estimate_bad_times},
Proposition~\ref{prop:perturbative_graph_regime}, 
Lemma~\ref{lem:post_processed_relative_entropy_inequality},
Lemma~\ref{lem:stability_estimate_local_terms},
and Proposition~\ref{prop:stability_estimate_nonlocal_terms},
we infer that for almost every $T' \in (0,T_*)$ there is $C(T') > 0$
such that 
\begin{equation}
\begin{aligned}
\label{eq:stability_estimate_improved}
&E[A,\mu|\A](T') 
+ \int_{0}^{T'} \chi_{\mathcal{T}_{\mathrm{bad}}} \bigg(\int_{\Omega} \frac{1}{4}|\nabla u|^2 \,dx
+ \int_{\Omega} \frac{1}{4}|\nabla w|^2 \,dx \bigg) \,dt
\\&~~
+ \int_{0}^{T'} \chi_{\mathcal{T}_{\mathrm{good}}}
	\bigg(\int_{\Omega} \frac{1}{4}|(\nabla u {-} \nabla\widetilde w)|^2 \,dx
	+ \int_{\Omega} \frac{1}{4}|(\nabla w {-} \nabla\widetilde w)|^2 \,dx\bigg) \,dt
\\&~~~~~~~ 
\leq E[A,\mu|\A](0) + C(T') \int_{0}^{T'} E[A,\mu|\A](t) \,dt.
\end{aligned}
\end{equation}
This of course implies the claim~\eqref{eq:stability_estimate}.

\textit{Step~3: Weak-strong uniqueness.} By Gr\"{o}nwall's inequality, it follows 
from the weak-strong stability estimate~\eqref{eq:stability_estimate} that
\begin{align*}
E[A,\mu|\A](T') \leq E[A,\mu|\A](0) e^{\int_{0}^{T'} C(t)\,dt}
\end{align*}
for almost every $T' \in (0,T_*)$. Weak-strong uniqueness in the form 
of~\eqref{eq:uniqueness} thus follows from~\eqref{eq:baseline_coercivity}.
\qed

\section{Weak-strong stability estimates for Mullins--Sekerka flow}

\subsection{Proof of Lemma~\ref{lem:relative_entropy_inequality}: Preliminary relative entropy inequality}
We proceed in two steps.

\textit{Step 1: Control by the dissipation estimate and by testing the weak formulation.}
Starting point for the derivation of~\eqref{eq:preliminary_estimate_rel_entropy}
is the following representation of the relative entropy:
\begin{align*}
E_{rel}[A,\mu|\A](T')
= E[\mu_{T'}]
+ \int_{A(T')} (\nabla \cdot \xi)(\cdot,T') \,dx
\end{align*}
for all $T' \in [0,T_*)$, where we integrated by parts in
the definition~\eqref{eq:rel_entropy} of the relative entropy 
and used that~$\xi(\cdot,T')$ is compactly supported within~$\Omega$, see~\eqref{eq:min_assumption6}. 
We thus infer from the dissipation estimate~\eqref{WeakFormDissipation} and
the weak formulation~\eqref{WeakFormMullinsSekerka} applied to $\nabla \cdot \xi$ that
\begin{align}
\nonumber
&E_{rel}[A,\mu|\A](T')
\\& \label{eq:auxRelEntropy1}
\leq E_{rel}[A,\mu|\A](0)
- \int_{0}^{T'} \int_{\Omega} \frac{1}{2}|\nabla u|^2 \,dx dt
- \int_{0}^{T'} \int_{\Omega} \frac{1}{2}|\nabla w|^2 \,dx dt
\\&~~~ \nonumber
- \int_{0}^{T'} \int_{\Omega} \nabla u \cdot \nabla (\nabla\cdot\xi) \,dx dt
- \int_{0}^{T'} \int_{\p^*A(t) \cap \Omega} \p_t\xi \cdot \n_{\p^* A} \,\dH[d-1] dt 
\end{align}
for a.e.\ $T' \in [0,T_*)$.

Starting point for the proof of~\eqref{eq:preliminary_estimate_bulk} is in 
turn the observation
\begin{align*}
E_{vol}[A|\A](T')
= \int_{\Omega} (\chi_A - \chi_{\A})(\cdot,T') \vartheta(\cdot,T') \,dx
\end{align*}
due to the sign conditions~\eqref{eq:min_assumption5}. As $\vartheta$ vanishes on 
$\partial \A$ by~\eqref{eq:min_assumption4}, we may use the
the weak formulation~\eqref{WeakFormMullinsSekerka} applied to the strong solution and 
$\vartheta$, along with an integration by parts in the potential term, to get
\begin{align*}
	\int_0^{T'} \int_{\Omega} \chi_{\A}\partial_t\vartheta \,dxdt =0,
\end{align*} 
so that the weak formulation~\eqref{WeakFormMullinsSekerka} applied 
to the weak solution gives
\begin{align}
\label{eq:auxBulkError1}
E_{vol}[A|\A](T')
&= E_{vol}[A|\A](0)
+ \int_0^{T'} \int_{\Omega} (\chi_A {-} \chi_{\A})\partial_t\vartheta \,dxdt
\\&~~~ \nonumber
- \int_0^{T'} \int_{\Omega} \nabla u  \cdot \nabla\vartheta \,dxdt.
\end{align}

\textit{Step 2: Inserting expected PDEs satisfied by~$\partial_t\xi$ and~$\partial_t\vartheta$.}
Since the vector fields~$\xi(\cdot,t)$ and~$B(\cdot,t)$ 
are compactly supported within~$\Omega$ for all $t \in [0,T_*)$,
see~\eqref{eq:min_assumption6} and~\eqref{eq:min_assumption8},
by now standard computations for the relative entropy method
in curvature driven interface evolution yield
(cf.\ \cite[identity~(2.11)]{Hensel2021a})
\begin{align}
\nonumber
&- \int_{0}^{T'} \int_{\p^*A(t) \cap \Omega} \p_t\xi \cdot \n_{\p^* A} \,\dH[d-1] dt
\\& \label{eq:auxRelEntropy4}
= - \int_{0}^{T'} \int_{\p^*A(t) \cap \Omega} 
(\mathrm{Id} {-} \n_{\p^* A} \otimes \n_{\p^* A} ) : \nabla B \,\dH[d-1] dt
\\&~~~ \nonumber
+ \int_{0}^{T'} \int_{\p^*A(t) \cap \Omega} 
(B \cdot \n_{\p^* A}) (\nabla \cdot \xi) \,\dH[d-1] dt
\\&~~~ \nonumber
- \int_{0}^{T'} \int_{\p^*A(t) \cap \Omega} 
\big(\p_t\xi {+} (B\cdot\nabla)\xi {+} (\nabla B)^\mathsf{T}\xi\big) \cdot (\n_{\p^* A} {-} \xi) \,\dH[d-1] dt
\\&~~~ \nonumber
- \int_{0}^{T'} \int_{\p^*A(t) \cap \Omega} 
\big(\p_t\xi {+} (B\cdot\nabla)\xi\big) \cdot \xi \,\dH[d-1] dt
\\&~~~ \nonumber
- \int_{0}^{T'} \int_{\p^*A(t) \cap \Omega} 
(\n_{\p^* A} {-} \xi) \otimes (\n_{\p^* A} {-} \xi) : \nabla B \,\dH[d-1] dt
\\&~~~ \nonumber
- \int_{0}^{T'} \int_{\p^*A(t) \cap \Omega} 
(\n_{\p^* A} \cdot \xi - 1) (\nabla \cdot B) \,\dH[d-1] dt
\end{align}
for a.e.\ $T' \in [0,T_*)$. Adding zero, then making 
use of the isotropic Gibbs--Thomson law~\eqref{WeakFormGibbsThomson} 
and finally plugging everything
back into~\eqref{eq:auxRelEntropy1} yields~\eqref{eq:preliminary_estimate_rel_entropy}.

Next, adding and subtracting $(\chi_A {-} \chi_{\A}) (B\cdot\nabla)\vartheta$, as well as integration by parts seperately for both solutions and using the facts that $\vartheta$ vanishes on $\partial \A$ and that~$B$ has compact support, as per identities~\eqref{eq:min_assumption4}
and~\eqref{eq:min_assumption8}, shows that
\begin{align}
\label{eq:auxBulkError5}
\int_{0}^{T'} \int_{\Omega} (\chi_A {-} \chi_{\A})\partial_t\vartheta \,dxdt
&= \int_{0}^{T'} \int_{\Omega} (\chi_A {-} \chi_{\A})
(\partial_t\vartheta {+} (B\cdot\nabla)\vartheta) \,dxdt 
\\&~~~ \nonumber
+ \int_{0}^{T'} \int_{\Omega} (\chi_A {-} \chi_{\A})
\vartheta \nabla\cdot B \,dxdt 
\\&~~~ \nonumber
+ \int_{0}^{T'} \int_{\p^*A(t)\cap\Omega} \vartheta \n_{\p^* A}\cdot B \,\dH[d-1] dt.
\end{align}
Plugging~\eqref{eq:auxBulkError5} back into~\eqref{eq:auxBulkError1} concludes the proof. \qed

\subsection{Proof of Lemma~\ref{lem:stability_estimate_bad_times}: Stability estimates at bad times}
For notational simplicity, we drop the time dependence in all occurring terms.
Looking back to the first step of the proof of Lemma~\ref{lem:relative_entropy_inequality}
(or alternatively, simply choosing $B \equiv 0$ at bad times),
we see that
\begin{align}
\label{eq:RHS_rel_entropy_bad_times}
R_{dissip} + R_{\partial_t\xi} + R_{\nabla B} 
&= - \int_{\Omega} \frac{1}{2}|\nabla u|^2 \,dx
- \int_{\Omega} \frac{1}{2}|\nabla w|^2 \,dx
\\&~~~ \nonumber
- \int_{\Omega} \nabla u \cdot \nabla (\nabla\cdot\xi) \,dx
\\&~~~ \nonumber
- \int_{\p^*A \cap \Omega} \p_t\xi \cdot \n_{\p^* A} \,\dH[d-1],
\\
\label{eq:RHS_bulk_bad_times}
U_{dissip} + U_{\partial_t\vartheta} + U_{\nabla B}
&= \int_{\Omega} (\chi_A {-} \chi_{\A})\partial_t\vartheta \,dx
-  \int_{\Omega} \nabla u \cdot \nabla\vartheta \,dx.
\end{align}

\textit{Step 1: Estimate for times $\Tbad^{(1)}(\Lambda)$.}
We simply estimate
\begin{equation}
\begin{aligned}
\label{eq:aux_bad_times_estimate1}
&- \int_{\Omega} \nabla u \cdot \nabla (\nabla\cdot\xi) \,dx
- \int_{\p^*A \cap \Omega} \p_t\xi \cdot \n_{\p^* A} \,\dH[d-1]
\\&~
\leq \int_{\Omega} \frac{1}{8}|\nabla u|^2 \,dx
+ 2 \mathcal{L}^d(\supp\xi) \|\nabla(\nabla\cdot\xi)\|_{L^{\infty}(\overline{\Omega})}^2
+ E[A(0)] \|\partial_t\xi\|_{L^\infty(\overline{\Omega})}
\end{aligned}
\end{equation}
and
\begin{equation}
\begin{aligned}
\label{eq:aux_bad_times_estimate2}
&\int_{\Omega} (\chi_A {-} \chi_{\A})\partial_t\vartheta \,dx
-  \int_{\Omega} \nabla u \cdot \nabla\vartheta \,dx
\\&~
\leq \int_{\Omega} \frac{1}{8} |\nabla u|^2 \,dx
+ 2 \mathcal{L}^d(\supp\nabla\vartheta) \|\nabla\vartheta\|_{L^\infty(\overline{\Omega})}^2
+ \mathcal{L}^d(\supp\partial_t\vartheta) \|\partial_t\vartheta\|_{L^\infty(\overline{\Omega})}^2.
\end{aligned}
\end{equation}
Hence, in view of~\eqref{eq:RHS_rel_entropy_bad_times} and~\eqref{eq:RHS_bulk_bad_times},
any $\Lambda > 0$ satisfying
\begin{align*}
\Lambda &\geq 4 \esssup_{t \in (0,T')} \big( 2 \mathcal{L}^d(\supp\xi) 
\|\nabla(\nabla\cdot\xi)\|_{L^{\infty}(\overline{\Omega})}^2
+ E[A(0)] \|\partial_t\xi\|_{L^\infty(\overline{\Omega})} \big)(t),
\\
\Lambda &\geq 4 \esssup_{t \in (0,T')} \big( 2 \mathcal{L}^d(\supp\nabla\vartheta) 
\|\nabla\vartheta\|_{L^\infty(\overline{\Omega})}^2
+ \mathcal{L}^d(\supp\partial_t\vartheta) \|\partial_t\vartheta\|_{L^\infty(\overline{\Omega})}^2 \big)(t)
\end{align*}
does the job for~\eqref{eq:rel_entropy_bad_times_estimate} and~\eqref{eq:bulk_bad_times_estimate}
due to definition~\eqref{eq:def_bad_times_dissipation}.

\textit{Step 2: Estimate for times $\Tbad^{(2)}(\Lambda,M)$.}
Just looking back at~\eqref{eq:RHS_rel_entropy_bad_times}--\eqref{eq:aux_bad_times_estimate2}
and recalling the admissible choices for~$\Lambda > 0$ from the previous display,
we see that~\eqref{eq:rel_entropy_bad_times_estimate} and~\eqref{eq:bulk_bad_times_estimate}
are immediate consequences of definition~\eqref{eq:def_bad_times_rel_entropy}.
\qed 

\subsection{Proof of Lemma~\ref{lem:post_processed_relative_entropy_inequality}: Structure of dissipative terms}
Before we derive the representation formulas~\eqref{eq:R_dissip_combined} 
and~\eqref{eq:A_dissip_combined}, we state two helpful intermediate results
whose proofs are postponed until the end of this subsection.
The first infers from the motion law~\eqref{WeakFormMullinsSekerka} harmonicity
of the chemical potential~$u(\cdot,t)$ away from the interface.

\begin{lemma}[Harmonicity of the chemical potential in perturbative regime]
\label{lem:harmonicityRegTimes}
Fix $t \in (0,T_*)$ subject to the assumptions of Lemma~\ref{lem:post_processed_relative_entropy_inequality}.
Then $u(\cdot,t)$ satisfies
\begin{align}
\label{eq:harmonicityInterior}
\int_{A(t)} \nabla u(\cdot,t) \cdot \nabla \zeta \,dx = 0
\end{align}
for all $\zeta \in C^\infty_{cpt}(A(t))$, and
\begin{align}
\label{eq:harmonicityExterior}
\int_{\Omega \setminus A(t)} \nabla u(\cdot,t) \cdot \nabla \zeta \,dx = 0
\end{align}
for all $\zeta \in C^\infty_{cpt}(\overline{\Omega}\setminus\overline{A(t)})$.
\end{lemma}

The second auxiliary result exploits the perturbative setting to facilitate 
a rigorous integration by parts involving the jumps of the Neumann data for
the chemical potentials along the interface~$\partial^*A(t)\cap\Omega$, 
which in turn then allows, among other things, to ``smuggle in'' the chemical potential~$\widetilde w(\cdot,t)$ 
by means of its boundary condition~\eqref{eq:auxChemPotential2}.

\begin{lemma}[Integration by parts formulas]
\label{lem:integration_by_parts}
Fix $t \in (0,T_*)$ subject to the assumptions of 
Lemma~\ref{lem:post_processed_relative_entropy_inequality}. Then
\begin{align}
&- \int_{\Omega} \nabla u(\cdot,t) \cdot \nabla (\nabla\cdot\xi)(\cdot,t) \,dx
\label{eq:auxRelEntropy11}
= \int_{\Omega} \nabla u(\cdot,t) \cdot \nabla \widetilde w(\cdot,t) \,dx
\end{align}
and
\begin{equation}
\begin{aligned}
\label{eq:int_by_parts1}
&- \int_{\Omega} \nabla \widetilde w(\cdot,t) \cdot \nabla (w {-} \widetilde w)(\cdot,t)  \,dx
\\&~
= \int_{\p^* A(t) \cap \Omega} (w {-} \widetilde w)(\cdot,t) \jump{(\n_{\p^* A} \cdot \nabla)\widetilde w}(\cdot,t) \,\dH[d-1],
\end{aligned}
\end{equation}
as well as
\begin{align}
\label{eq:int_by_parts2}
-\int_{\Omega} \nabla \widetilde w \cdot \nabla\vartheta \,dx &=
\int_{\p^*A\cap\Omega} \vartheta \jump{(\n_{\p^* A}\cdot\nabla)\widetilde w} \,\dH[d-1],
\\ 
\label{eq:int_by_parts3}
-\int_{\Omega} (\nabla u {-} \nabla \widetilde w) \cdot \nabla\vartheta \,dx &=
\int_{\p^*A\cap\Omega} (u {-} \widetilde w) 
\jump{(\n_{\p^* A}\cdot\nabla)\widetilde w_\vartheta} \,\dH[d-1].
\end{align}
\end{lemma} 

We now combine the information from the previous two results to a proof
of the representation formulas~\eqref{eq:R_dissip_combined} 
and~\eqref{eq:A_dissip_combined}.

\begin{proof}[Proof of Lemma~\ref{lem:post_processed_relative_entropy_inequality}]
The proof is split into two steps. For notational ease, we drop in the notation the time dependence
of all quantities.

\textit{Step 1: Formula for~$R_{dissip}$.} For convenience, let us restate the definition
of~$R_{dissip}$ from~\eqref{eq:R_dissip}:
\begin{align}
\label{eq:R_dissip_again}
R_{dissip}
&= - \int_{\Omega} \frac{1}{2}|\nabla u|^2 \,dx
   - \int_{\Omega} \frac{1}{2}|\nabla w|^2 \,dx
- \int_{\Omega} \nabla u \cdot \nabla (\nabla\cdot\xi) \,dx
\\&~~~ \nonumber
+ \int_{\p^*A \cap \Omega} (B \cdot \n_{\p^* A}) \big(w + (\nabla \cdot \xi)\big) \,\dH[d-1].
\end{align}
Inserting the boundary condition~\eqref{eq:auxChemPotential2}, we observe that
\begin{align}
\label{eq:R_dissip_aux1}
\int_{\p^*A \cap \Omega} (B \cdot \n_{\p^* A}) \big(w + (\nabla \cdot \xi)\big) \,\dH[d-1]
= \int_{\p^*A \cap \Omega} (B \cdot \n_{\p^* A}) (w - \widetilde w) \,\dH[d-1].
\end{align}
Next, by inserting the integration by parts formula~\eqref{eq:auxRelEntropy11}
and completing squares, we get
\begin{align}
\nonumber
&- \int_{\Omega} \frac{1}{2}|\nabla u|^2 \,dx
- \int_{\Omega} \frac{1}{2}|\nabla w|^2 \,dx
- \int_{\Omega} \nabla u \cdot \nabla (\nabla\cdot\xi) \,dx
\\&~ \label{eq:auxRelEntropy12}
= - \int_{\Omega} \frac{1}{2}|\nabla u {-} \nabla \widetilde w|^2 \,dx
- \int_{\Omega} \frac{1}{2}|\nabla w {-} \nabla \widetilde w|^2 \,dx
- \int_{\Omega} \nabla \widetilde w \cdot \nabla (w {-} \widetilde w)  \,dx.
\end{align}
Hence, \eqref{eq:R_dissip_combined} follows from~\eqref{eq:R_dissip_again}--\eqref{eq:auxRelEntropy12}
and the integration by parts formula~\eqref{eq:int_by_parts1}.

\textit{Step 2: Formula for~$U_{dissip}$.} Recall first from~\eqref{eq:A_dissip} that
\begin{align}
\label{eq:A_dissip_again}
U_{dissip}
&= - \int_{\Omega} \nabla u \cdot \nabla (\nabla\cdot\xi)\,dx
+ \int_{\p^*A\cap\Omega} \vartheta (\n_{\p^* A}\cdot B)\,\dH[d-1].
\end{align}
To upgrade this to~\eqref{eq:A_dissip_combined}, we simply
start by adding zero in the form of
\begin{align*}
&- \int_{\Omega} \nabla u \cdot \nabla\vartheta \,dx
+ \int_{\p^*A\cap\Omega} \vartheta \n_{\p^* A}\cdot B \,\dH[d-1]
\\&~
= - \int_{\Omega} (\nabla u {-} \nabla \widetilde w) \cdot \nabla\vartheta \,dx
- \int_{\Omega} \nabla \widetilde w \cdot \nabla\vartheta \,dx
+ \int_{\p^*A\cap\Omega} \vartheta \n_{\p^* A}\cdot B \,\dH[d-1].
\end{align*}
Hence, \eqref{eq:A_dissip_combined} follows from the
integration by parts formulas~\eqref{eq:int_by_parts2} and~\eqref{eq:int_by_parts3}.
\end{proof}

\begin{proof}[Proof of Lemma~\ref{lem:harmonicityRegTimes}]
Let $\zeta\in C^\infty_{cpt}(\overline{\Omega})$ 
be a nonnegative test function supported in $\overline{\Omega}\setminus\overline{A(t)}$. 
We then have by the nonnegativity of~$\zeta$, the property $\zeta \equiv 0$ 
in~$A(t)$, and the weak formulation of the Mullins-Sekerka equation~\eqref{WeakFormMullinsSekerka}
\begin{align*}
0 &\leq \int_{A(\tilde t)} \zeta \,dx
=\int_{A(\tilde t)} \zeta \,dx
-\int_{A(t)} \zeta \,dx
\stackrel{\eqref{WeakFormMullinsSekerka}}{=}
- \int_{t}^{\tilde t} \int_{\Omega} \nabla u \cdot \nabla \zeta \,dx ds.
\end{align*}
Dividing by $|\tilde t-t|$ and adding zero, we obtain
\begin{align*}
0&\leq
- \sign(\tilde t - t) \int_{\Omega \setminus A(t)} \nabla u(\cdot,t) \cdot \nabla \zeta \,dx
\\&~~~
- \frac{1}{|\tilde t-t|} \int_t^{\tilde t} \int_{\Omega} 
(\nabla u(\cdot,s)-\nabla u(\cdot,t))\cdot \nabla \zeta \,dx ds.
\end{align*}
Passing to the limit $\tilde t\searrow t$ respectively $\tilde t \nearrow t$, 
we deduce by~\eqref{PropertyLebesguePoint}
\begin{align*}
- \int_{\Omega \setminus A(t)} \nabla u(\cdot,t) \cdot \nabla \zeta \,dx \geq 0
\end{align*}
respectively
\begin{align*}
- \int_{\Omega \setminus A(t)} \nabla u(\cdot,t) \cdot \nabla \zeta \,dx \leq 0.
\end{align*}
This implies~\eqref{eq:harmonicityExterior}. 
Analogously, one shows~\eqref{eq:harmonicityInterior} 
which then implies our lemma.
\end{proof}

\begin{proof}[Proof of Lemma~\ref{lem:integration_by_parts}]
We again drop the time dependence of all objects in the notation
and proceed in three steps.

\textit{Step 1: Proof of~\eqref{eq:auxRelEntropy11}.} We note that due to~\eqref{eq:graph_representation}
it holds $\p^*A \cap \Omega = \p^*A = \p A \subset\subset \Omega$.
Furthermore, Lemma~\ref{lem:harmonicityRegTimes} tells us that
$\nabla \cdot \nabla u$ exists in the sense of square-integrable weak derivatives
within the open sets~$A$ and~$\Omega\setminus\overline{A}$, respectively,
and that in fact 
\begin{align}
\label{eq:harmonicity}
\nabla \cdot \nabla u = 0
\quad \text{throughout } A \text{ and } \Omega\setminus\overline{A}, \text{ respectively.}
\end{align}
In other words, $\nabla u \in Y^2_{div}(D) := \{f \in L^2(D;\Rd)\colon \nabla \cdot f \in L^2(D;\Rd)\}$
for $D \in \{A,\Omega\setminus\overline{A}\}$. It is a standard result 
(for instance in the context of mathematical fluid mechanics in connection with
the Helmholtz decomposition; see, e.g., Simader and Sohr~\cite[Theorem~5.3]{Simader1992})
that there exists a continuous linear operator $T_{D}\colon Y^2_{div}(D) \to W^{-\frac{1}{2},2}(\p D)$
such that for all $f\in Y^2_{div}(D)$ and all $\zeta \in H^1(D)$
\begin{align}
\label{eq:generalizedGaussGreen}
\int_D \zeta (\nabla \cdot f) \,dx = - \int_D \nabla \zeta \cdot f \,dx
- \langle T_D(f), \zeta \rangle,
\end{align}
i.e., one may interpret $T_Df$ as $(\n_{\p D}\cdot f)|_{\p D}$ in a weak sense.
Denoting now by $\eta\colon\overline{\Omega} \to [0,1]$ a $C^1$-cutoff
such that $\eta \equiv 1$ on~$\overline{A}$ and $\eta \equiv 0$ on~$\p\Omega$,
we thus obtain from~\eqref{eq:harmonicity}, \eqref{eq:generalizedGaussGreen}
and using that~$\xi$ is compactly supported within~$\Omega$
\begin{align}
\label{eq:auxRelEntropy7}
&- \int_{\Omega} \nabla u \cdot \nabla (\nabla\cdot\xi) \,dx
\\& \nonumber
= - \int_{A} \nabla u \cdot \nabla (\nabla\cdot\xi) \,dx
-  \int_{\Omega \setminus \overline{A}} \nabla u \cdot 
	 \nabla (\eta + (1{-}\eta))(\nabla\cdot\xi) \,dx
\\& \nonumber
= \langle T_{A}(\nabla u), (\nabla\cdot\xi) \rangle
+ \langle T_{\Omega\setminus\overline{A}}(\nabla u), \eta (\nabla\cdot\xi) \rangle
+ \langle T_{\Omega\setminus\overline{A}}(\nabla u), (1{-}\eta) (\nabla\cdot\xi) \rangle.
\end{align}
Since~$(1{-}\eta)\equiv 0$ on~$\p A$ by its choice and $\nabla\cdot\xi \equiv 0$
on~$\p\Omega$, it holds~$(1{-}\eta)(\nabla\cdot\xi) \in H^1_0(\Omega\setminus\overline{D})$
and therefore
\begin{align}
\label{eq:auxRelEntropy8}
\langle T_{\Omega\setminus\overline{A}}(\nabla u), (1{-}\eta) (\nabla\cdot\xi) \rangle = 0.
\end{align}
Moreover, due to the boundary condition~\eqref{eq:auxChemPotential2}
and~$\eta \equiv 0$ on~$\p\Omega$ by its choice,
it holds $\nabla\cdot\xi = -\widetilde w \in W^{\frac{1}{2},2}(\p A)$ as well as
$\eta \nabla\cdot\xi = - \eta\widetilde w \in W^{\frac{1}{2},2}(\p(\Omega\setminus\overline{A}))$
and therefore
\begin{align*}
&\langle T_{A}(\nabla u), \nabla\cdot\xi \rangle
+ \langle T_{\Omega\setminus\overline{A}}(\nabla u), \eta (\nabla\cdot\xi) \rangle
\\&
= - \langle T_{A}(\nabla u), \widetilde w \rangle
- \langle T_{\Omega\setminus\overline{A}}(\nabla u), \eta\widetilde w \rangle
\\&
=  - \langle T_{A}(\nabla u), \widetilde w \rangle 
- \langle T_{\Omega\setminus\overline{A}}(\nabla u), \widetilde w \rangle
+ \langle T_{\Omega\setminus\overline{A}}(\nabla u), (1{-}\eta)\widetilde w \rangle.
\end{align*}
Based on~\eqref{eq:auxChemPotential1} and~\eqref{eq:generalizedGaussGreen},
the previous display admits the upgrade
\begin{align}
\nonumber
&\langle T_{A}(\nabla u), \nabla\cdot\xi \rangle
+ \langle T_{\Omega\setminus\overline{A}}(\nabla u), \eta (\nabla\cdot\xi) \rangle
\\& \nonumber
= \int_{A} \nabla u \cdot \nabla \widetilde w \,dx 
+ \int_{\Omega\setminus\overline{A}} \nabla u \cdot \nabla \widetilde w \,dx
+ \langle T_{\Omega\setminus\overline{A}}(\nabla u), (1{-}\eta)\widetilde w \rangle
\\& \label{eq:auxRelEntropy9}
= \int_{\Omega} \nabla u \cdot \nabla \widetilde w \,dx
+ \langle T_{\Omega\setminus\overline{A}}(\nabla u), (1{-}\eta)\widetilde w \rangle.
\end{align}
It remains to identify $\langle T_{\Omega\setminus\overline{A}}(\nabla u), (1{-}\eta)\widetilde w \rangle$.
However, comparing~\eqref{eq:generalizedGaussGreen} applied to the data $(f,\zeta) 
= (\nabla u,(1{-}\eta)\widetilde w),
\,\nabla\cdot f = 0,\,D=\Omega\setminus\overline{A}$, with~\eqref{eq:harmonicityExterior}
applied to the admissible data~$\zeta=(1{-}\eta)\widetilde w$ yields 
(recall also that $\eta\equiv 0$ along~$\p\Omega$)
\begin{align}
\label{eq:auxRelEntropy10}
\langle T_{\Omega\setminus\overline{A}}(\nabla u), (1{-}\eta)\widetilde w \rangle
= 0.
\end{align}
The identities~\eqref{eq:auxRelEntropy7}--\eqref{eq:auxRelEntropy10} together finally imply
the claim~\eqref{eq:auxRelEntropy11}.

\textit{Step 2: Proof of~\eqref{eq:int_by_parts1}.}
Denote again by $\eta\colon\overline{\Omega} \to [0,1]$ a $C^1$-cutoff
such that $\eta \equiv 1$ on~$\overline{A}$ and $\eta \equiv 0$ on~$\p\Omega$.
Analogous to the above arguments, we then obtain on one side
\begin{align*}
&- \int_{\Omega\setminus\overline{A}} \nabla \widetilde w 
\cdot \nabla \big((1{-}\eta)(w - \widetilde w)\big)  \,dx
= \langle T_{\Omega\setminus\overline{A}}(\nabla \widetilde w), (1{-}\eta)(w - \widetilde w) \rangle
= 0.
\end{align*}
On the other side, we may exploit that the interface~$\p A = \p^* A = \p^*A \cap \Omega$ is~$C^{1,\alpha}$
due to assumption~\eqref{eq:regularity_graph}. Since
the boundary data in~\eqref{eq:auxChemPotential2} is smooth, we get by
standard Schauder theory that the auxiliary chemical potential is $C^1$
up to the interface~$\p A$ from both sides, so that
\begin{align*}
&- \int_{A} \nabla \widetilde w \cdot \nabla (w - \widetilde w)  \,dx
- \int_{\Omega\setminus\overline{A}} \nabla \widetilde w 
\cdot \nabla \big(\eta(w - \widetilde w)\big)  \,dx
\\&
= - \int_{\Omega} \nabla \widetilde w \cdot \nabla \big(\eta(w - \widetilde w)\big)  \,dx
= \int_{\p^* A \cap \Omega} \jump{(\n_{\p^* A} \cdot \nabla)\widetilde w} (w - \widetilde w) \,\dH[d-1].
\end{align*}
The previous two displays obviously imply~\eqref{eq:int_by_parts1}.

\textit{Step 3: Proof of~\eqref{eq:int_by_parts2} and~\eqref{eq:int_by_parts3}.}
These two follow from analogous arguments.
\end{proof}

\subsection{Proof of Lemma~\ref{lem:stability_estimate_local_terms}: Stability estimate for local terms}
Thanks to~\eqref{eq:perturbative_regime1} and~\eqref{eq:graph_representation}, 
we know $\partial^*A\cap\Omega \subset B_{\ell/4}(\p\A)$ as well as
$A\Delta\A \subset B_{\ell/4}(\p\A)$.
Hence, since $\frac{1}{2}|n{-}\xi|^2=1-n\cdot\xi$ due to the first item of~\eqref{eq:min_assumption3},
we immediately obtain the estimates~\eqref{eq:estimate_RnablaB} and~\eqref{eq:estimate_AnablaB};
recall for this also the definitions~\eqref{eq:rel_entropy} and~\eqref{eq:bulk_error} 
of $E_{rel}$ and $E_{vol}$, respectively. Furthermore, the identity~\eqref{eq:estimate_Rdtxi}
is a consequence of the following identities throughout~$B_{\ell/4}(\p\A)$:
\begin{align*}
\xi \cdot (\partial_t\xi + (B\cdot\nabla\xi)) &= 0, 
\\
\partial_t \xi + (B\cdot\nabla)\xi + (\nabla B)^\mathsf{T}\xi &= 0.
\end{align*}
The first is true simply because of~$|\xi| \equiv 1$ within~$B_{\ell/2}(\p\A)$, see
again the first item of~\eqref{eq:min_assumption3}. The second also
follows from the latter by taking the spatial gradient of 
$\partial_t s_{\p\A} + (B\cdot\nabla)s_{\p\A} = 0$, which itself
is true within~$B_{\ell/2}(\p\A)$ thanks to~\eqref{eq:min_assumption9}.
Note also in this context that the transport equation satisfied by the signed distance 
also directly implies~\eqref{eq:estimate_Adtvartheta}
due to~\eqref{eq:min_assumption4}. 

It remains to prove~\eqref{eq:estimate_RvarifoldBV}. 
To this end, we first obtain
using~\eqref{def:multiplicity} and~\eqref{eq:repRelEntropy2}
\begin{align*}
&-\int_{\p^*A \cap \Omega} \nabla \cdot B \,\dH[d-1]
\\&~
= -\int_{\Omega} \varrho \nabla \cdot B \,d|\mu|_{\mathbb{S}^{d-1}}
\leq -\int_{\overline{\Omega}{\times}\mathbb{S}^{d-1}} \nabla \cdot B \,d\mu
+ \|\nabla\cdot B\|_{L^\infty(B_{\ell/2}(\p\A))} E_{rel}[A,\mu|\A].
\end{align*}
Adding zero several times, applying the compatibility condition
from~\cite[Definition~3, item~ii)]{Hensel2022}
with test function $\eta = (\nabla B)^\mathsf{T}\xi$, and estimating
based on the properties $(\xi\cdot\nabla)B=0$ and
$\frac{1}{2}|\xi-\pvec|^2\leq1-\xi\cdot \pvec,\ \pvec\in\mathbb{S}^{d-1}$,
valid throughout $B_{\ell/2}(\p\A)$,
cf.\  \eqref{eq:min_assumption3} and~\eqref{eq:min_assumption9}, 
furthermore entails
\begin{align*}
\int_{\p^*A \cap \Omega} \n_{\p^*A} \otimes \n_{\p^*A} : \nabla B \, \dH[d-1]
&= \int_{\p^*A \cap \Omega} (\n_{\p^*A} {-} \xi) \otimes \n_{\p^*A} : \nabla B \,\dH[d-1]
\\&~~~
+ \int_{\overline{\Omega}{\times}\mathbb{S}^{d-1}} \xi \otimes \pvec : \nabla B \,d\mu
\\&
= \int_{\p^*A \cap \Omega} (\n_{\p^*A} {-} \xi) \otimes (\n_{\p^*A} {-} \xi) : \nabla B \,\dH[d-1]
\\&~~~
- \int_{\overline{\Omega}{\times}\mathbb{S}^{d-1}} 
(\pvec {-} \xi) \otimes (\pvec {-} \xi) : \nabla B \,d\mu
\\&~~~
+ \int_{\overline{\Omega}{\times}\mathbb{S}^{d-1}} 
\pvec \otimes \pvec : \nabla B \,d\mu
\\&
\leq  \int_{\overline{\Omega}{\times}\mathbb{S}^{d-1}} 
\pvec \otimes \pvec : \nabla B \,d\mu
\\&~~~
+ 4\|\nabla B{+}(\nabla B)^\mathsf{T}\|_{L^\infty(B_{\ell/2}(\p\A))}E_{rel}[A,\mu|\A].
\end{align*}
This concludes the proof.
\qed

\subsection{Proof of Proposition~\ref{prop:stability_estimate_nonlocal_terms}: 
Stability estimates for non-local terms}
\label{subsec:estimate_dissipative_terms}
For this whole subsection, let the assumptions 
and notation of Proposition~\ref{prop:stability_estimate_nonlocal_terms} be in place.
As before, let us drop the time dependence of all quantities in the notation.
In order to streamline the proof of Proposition~\ref{prop:stability_estimate_nonlocal_terms},
we start by collecting an array of auxiliary constructions and results.
Proofs for these are deferred until after we showed how the 
estimate~\eqref{eq:estimate_R_dissip} can be inferred from these. We
finally conclude this subsection with the missing proofs for 
the estimates~\eqref{eq:estimate_A_dissip1} and~\eqref{eq:estimate_A_dissip2}.

Throughout this subsection, we make use of the following
convention: for a given point $y \in \p\A$, we always denote by $(\ta^{i}_{\p\A})_{i=1,\ldots,d{-}1}$
an orthonormal basis for $\mathrm{Tan}_{y}\p\A$ such that each $\ta^{i}_{\p\A}$
represents a principal curvature direction at~$y$ with corresponding principal
curvature~$\kappa^i_{\p\A}$. 
To avoid cumbersome notation in the following, let us also define
\begin{align}
\label{eq:def_normal}
\bar{\n}_{\p\A}(x) &:= \n_{\p\A}(P_{\p\A}(x)),
\\
\bar{\ta}^{i}_{\p\A}(x) &:= \ta^{i}_{\p\A}(P_{\p\A}(x)),
\\
\bar{\kappa}^{i}_{\p\A}(x) &:= (\kappa^{i}_{\p\A})(P_{\p\A}(x)),
\\
\bar{H}_{\p\A}(x) &:= \sum_{i=1}^{d-1} \bar{\kappa}^{i}_{\p\A}(x),
\\
\bar{h}(x) &:= h(P_{\p\A}(x))
\end{align}
for $x \in B_{\frac{\ell}{2}}(\p\A)$.
In terms of this data, we then obtain the following useful formulas.

\begin{lemma}[Some elementary identities]
\label{lem:geometric_identities0}
It holds
\begin{align}
\label{eq:NablaProjection}
\nabla P_{\p\A}(x) = \sum_{i=1}^{d-1} \frac{1}{1 {-} s_{\p\A}(x) \bar{\kappa}_i(x)}
\bar{\ta}^{i}_{\p\A}(x) \otimes \bar{\ta}^{i}_{\p\A}(x).
\end{align}
Furthermore, it holds
\begin{align}
\label{eq:xi}
\xi(x) = \bar{\n}_{\p\A}(x) = \nabla s_{\p\A}(x), \quad x \in B_{\frac{\ell}{2}}(\p\A),
\end{align}
as well as
\begin{align}
\label{eq:grad_xi}
\nabla \xi(x) = -\sum_{i=1}^{d-1} \frac{\bar{\kappa}^i_{\p\A}(x)}{1 {-} s_{\p\A}(x) \bar{\kappa}^i_{\p\A}(x)}
\bar{\ta}^{i}_{\p\A}(x) \otimes \bar{\ta}^{i}_{\p\A}(x),
\quad x \in B_{\frac{\ell}{2}}(\p\A),
\end{align}
and thus
\begin{equation}
\begin{aligned}
\label{eq:div_xi}
\nabla\cdot\xi(x) 
&= - \bar{H}_{\p\A} - \sum_{i=1}^{d-1} \frac{\big(\bar{\kappa}^i_{\p\A}\big)^2(x)}{1 {-} s_{\p\A}(x) \bar{\kappa}^i_{\p\A}(x)} s_{\p\A}(x)
\quad x \in B_{\frac{\ell}{2}}(\p\A).
\end{aligned}
\end{equation}
Finally, for all $f \in L^1(B_{\frac{\ell}{2}}(\p\A);dx)$ it holds
\begin{align}
\label{eq:CoareaFormulaTubularNeighborhoodDiffeo}
\int_{B_{\frac{\ell}{2}}(\p\A)} f \,dx
= \int_{\p\A} \int_{-\frac{\ell}{2}}^{\frac{\ell}{2}}
f\big(y{+}s\n_{\p\A}(y)\big)
\prod_{i=1}^{d-1} \big(1 {-} s\kappa^i_{\p\A}(y)\big)
\,dsd\mathcal{H}^{d-1}.
\end{align}
\end{lemma}

In order to pull off the strategy for the proof of~\eqref{eq:estimate_R_dissip}
as mentioned before the statement of 
Proposition~\ref{prop:stability_estimate_nonlocal_terms}, we 
next introduce a $C^1$~diffeomorphism $\Psi^h\colon \Omega \to \Omega$ 
with the property
\begin{align}
\label{eq:mapping_property_diffeo}
\Psi^h(A) = \A.
\end{align}
This can be done as follows. Let $\bar\zeta \in C^\infty_{cpt}(\mathbb{R};[0,1])$
such that $\supp\bar\zeta \subset [-1/2,1/2]$, $\bar\zeta \equiv 1$
on $[-1/4,1/4]$ and $|\bar\zeta'| \leq 8$. Based on this auxiliary cutoff, define a map
\begin{align}
\label{eq:def_plateau_cutoff}
\zeta\colon\Rd \to [0,1], \quad 
x \mapsto \bar\zeta\Big(\frac{s_{\p\A}(x)}{\ell}\Big).
\end{align}
We may then define the desired map~$\Psi^h$ by means of
\begin{align}
\label{eq:def_diffeo}
\Psi^h(x) := x - \zeta(x)\bar{h}(x)\bar{\n}_{\p\A}(x).
\end{align}
We collect in the upcoming result some basic facts about~$\Psi^h$.

\begin{lemma}[Some further useful facts]
\label{lem:geometric_identities}
We have throughout~$\Omega$
\begin{equation}
\begin{aligned}
\label{eq:grad_diffeo}
\nabla\Psi^h = \mathrm{Id} 
&- \big(\bar{\n}_{\p\A}\cdot\nabla\zeta\big) \bar{h}
\,\bar{\n}_{\p\A}\otimes\bar{\n}_{\p\A}
\\&
+ \zeta \sum_{i=1}^{d-1} \frac{\bar{\kappa}^{i}_{\p\A}}{1{-}s_{\p\A}\bar{\kappa}^{i}_{\p\A}} \bar{h} \,
\bar{\ta}^{i}_{\p\A} \otimes \bar{\ta}^{i}_{\p\A},
\\&
- \zeta \sum_{i=1}^{d-1} \frac{\big((\bar{\ta}^{i}_{\p\A} \cdot \nabla_{\p\A})h\big)
\circ P_{\p\A}}{1{-}s_{\p\A}\bar{\kappa}^{i}_{\p\A}}
\bar{\n}_{\p\A} \otimes \bar{\ta}^{i}_{\p\A},
\end{aligned}
\end{equation}
and thus throughout~$\Omega$
\begin{align}
\label{eq:div_diffeo}
\det \nabla\Psi^h &= \big(1 - \big(\bar{\n}_{\p\A}\cdot\nabla\zeta\big) \bar{h}\big)
\prod_{i=1}^{d-1}\Big(1 + \zeta\frac{\bar{\kappa}^{i}_{\p\A}}{1{-}s_{\p\A}\kappa^{i}_{\p\A}}\bar{h}\Big).
\end{align}
In particular, $\Psi^h\colon\Omega\to\Omega$ is a 
$C^{1,1{-}\frac{d{-}1}{4}}$~diffeomorphism and
\begin{align}
\label{eq:diffeo_short}
\Psi^h(x) &= x - \bar{h}(x)\bar{\n}_{\p\A}(x), && x \in B_{2\|h\|_{L^\infty(\p\A)}}(\p\A),
\\
\label{eq:inv_diffeo_short}
(\Psi^h)^{-1}(x) &= x + \bar{h}(x)\bar{\n}_{\p\A}(x), && x \in B_{2\|h\|_{L^\infty(\p\A)}}(\p\A).
\end{align}
Furthermore, for all $f \in L^1(\p^*A;d\mathcal{H}^{d-1})$ it holds
\begin{equation}
\begin{aligned}
\label{eq:areaFormulaSurfaceIntegral}
&\int_{\p^*A} f \,d\mathcal{H}^{d-1}
\\&
= \int_{\p\A} f\circ\big.(\Psi^h)^{-1}\big|_{\p\A}
\sqrt{1 {+} \sum_{i=1}^{d-1} \Big(\frac{(\ta^{i}_{\p\A}\cdot\nabla_{\p\A})h}{1 {-} h\kappa^{i}_{\p\A}}\Big)^2}
\prod_{i=1}^{d-1} \big(1 {-} h\kappa^{i}_{\p\A}\big) \,d\mathcal{H}^{d-1}.
\end{aligned}
\end{equation}
Finally, 
\begin{align}
\label{eq:normalInterfaceWeakSolution}
\n_{\p^*A}\circ\big.(\Psi^h)^{-1}\big|_{\p\A}
= \frac{1}{\sqrt{1 {+} \sum_{i=1}^{d-1} \big(\frac{(\ta^{i}_{\p\A}\cdot\nabla_{\p\A})h}{1 {-} h\kappa^{i}_{\p\A}}\big)^2}}
\Big(\n_{\p\A} - \sum_{i=1}^{d-1} \frac{(\ta^{i}_{\p\A}\cdot\nabla_{\p\A})h}{1 {-} h\kappa^{i}_{\p\A}} \,\ta^{i}_{\p\A}\Big)
\end{align}
and for all~$\delta > 0$ one may choose $C \gg_{\delta} 1$
from~\eqref{eq:perturbative_regime1_again} such that
\begin{equation}
\begin{aligned}
\label{eq:meanCurvatureInterfaceWeakSolution}
&\bigg|(\mathrm{tr}_{\p^*A} w) \circ\big.(\Psi^h)^{-1}\big|_{\p\A}
-\bigg(H_{\p\A} + \Big(\sum_{i=1}^{d-1} (\kappa^{i}_{\p\A})^2\Big) h + \Delta_{\p\A} h\bigg)\bigg|
\\&~~~~~~~~~~~~~~~~
\leq \delta\max\big\{|h|,|\nabla_{\p\A}h|,|\mathrm{Hess}_{\p\A}h|\big\}.
\end{aligned}
\end{equation}
\end{lemma}

The next ingredient is concerned with a representation of the
error functionals in terms of the graph function.

\begin{lemma}[Error control in perturbative setting]
\label{lem:error_control_perturbative_regime}
For any $\delta \in (0,1)$ one may choose $C \gg_{\delta} 1$
from~\eqref{eq:perturbative_regime1_again} such that
\begin{align}
\label{eq:coercivity_Erel}
(1 {-} \delta) \int_{\p\A} \frac{1}{2} |\nabla_{\p\A}h|^2 \,\dH[d-1]
&\leq E_{rel}[A,\mu|\A] 
\leq (1 {+} \delta) \int_{\p\A} \frac{1}{2} |\nabla_{\p\A}h|^2 \,\dH[d-1],
\\
\label{eq:coercivity_Ebulk}
(1 {-} \delta) \int_{\p\A} \frac{1}{2} \Big(\frac{h}{\ell}\Big)^2 \,\dH[d-1]
&\leq E_{vol}[A|\A] 
\leq (1 {+} \delta) \int_{\p\A} \frac{1}{2} \Big(\frac{h}{\ell}\Big)^2 \,\dH[d-1]
\end{align}
and
\begin{align}
\label{eq:comparable_length}
(1 {-} \delta) \mathcal{H}^{d-1}(\p\A) \leq \mathcal{H}^{d-1}(\p^*A\cap\Omega)
\leq (1 {+} \delta) \mathcal{H}^{d-1}(\p\A).
\end{align}
\end{lemma}

In the following result, we record the PDEs satisfied by the
chemical potentials after pulling back the domain~$A$ to~$\A$. 

\begin{lemma}[PDEs satisfied by transformed chemical potentials]
\label{lem:transformed_chemical_potentials}
For arbitrary $v \in H^1(\Omega)$, we define $v_h := v \circ (\Psi^h)^{-1}$. 
We also define the uniformly elliptic and bounded
coefficient field $a^h := \frac{1}{|\det \Psi^h|}(\nabla\Psi^h)^\mathsf{T}\nabla\Psi^h$.
Then, $(w - \widetilde{w})_h \in H^1(\Omega)$ satisfies
\begin{align}
\label{eq:transformedPDE1}
\Delta (w - \widetilde{w})_h &= 
\nabla\cdot\big((\mathrm{Id} - a^h)\nabla(w - \widetilde{w})_h\big)
&& \text{in } \Omega \setminus \p\A,
\\
\mathrm{tr}_{\p\A} (w - \widetilde{w})_h
&= \big(\mathrm{tr}_{\p^* A} (w) + \nabla\cdot\xi\big)\circ (\Psi^h)^{-1}
&& \text{on } \p\A,
\\
(\n_{\p\Omega}\cdot\nabla)(w - \widetilde{w})_h &= 0 
&& \text{on } \p\Omega.
\end{align}
Writing $\bar u \in H^1(\Omega)$ for the chemical potential associated
with the smoothly evolving phase~$\A$, we also get for $\bar u - \widetilde w_h \in H^1(\Omega)$
\begin{align}
\label{eq:transformedPDE3}
\Delta (\bar u - \widetilde w_h) &= 
\nabla\cdot\big((a^h - \mathrm{Id})\nabla \widetilde{w}_h\big)
&& \text{in } \Omega \setminus \p\A,
\\
\mathrm{tr}_{\p\A} (\bar u - \widetilde w_h)
&= \big(H_{\p\A} + (\nabla\cdot\xi)\circ (\Psi^h)^{-1}\big)
&& \text{on } \p\A,
\\
(\n_{\p\Omega}\cdot\nabla)(\bar u - \widetilde w_h) &= 0 
&& \text{on } \p\Omega.
\end{align}
\end{lemma}

In order to exploit the Hilbert space structure of~$H^{1/2}_{MS}(\p\A)$
for an estimate of~$R_{dissip}$ and $A_{dissip}$, respectively, 
we have to smuggle in the averages of the associated chemical potentials.
The following result ensures that this can be done solely at the cost
of controlled quantities. 

\begin{lemma}[Smuggling in averages of chemical potentials]
\label{lem:bound_averages}
For any $\delta \in (0,1)$ one may choose $C \gg_{\delta} 1$
from~\eqref{eq:perturbative_regime1_again} and $M \gg_{\delta} 1$
from~\eqref{eq:small_error} such that for some universal 
constant $\widetilde{C} > 0$
\begin{equation}
\begin{aligned}
\label{eq:bound_averages}
&\bigg|\,\dashint_{\p^* A \cap \Omega} w {-} \widetilde w \,\dH[d-1]\bigg|
+ \bigg|\,\dashint_{\p\A} (w {-} \widetilde w)_h \,\dH[d-1]\bigg|
\\&~
\leq \delta C_{PS}(\A)\bigg(\int_{\Omega} |\nabla (w {-} \widetilde w)|^2 \,dx\bigg)^\frac{1}{2}
+ \frac{\widetilde{C}}{\ell}\frac{\sqrt{E[A|\A]}}{\sqrt{\mathcal{H}^{d-1}(\p^*A\cap\Omega)}},
\end{aligned}
\end{equation}
where $C_{PS}(\A)$ is the constant from the Sobolev--Poincar\'{e} trace inequality:
for any $v \in H^1(\A)$,
\begin{align}
\label{eq:Sobolev_Poincare_trace}
\bigg(\,\dashint_{\p\A} \Big(v {-} \dashint_{\p\A} v \,\dH[d-1] \Big)^2\,\dH[d-1]\bigg)^\frac{1}{2}
\leq C_{PS}(\A) \bigg(\int_{\A} |\nabla v|^2 \,dx \bigg)^\frac{1}{2}.
\end{align}
Furthermore, 
\begin{equation}
\begin{aligned}
\label{eq:average_B}
&\bigg|\,\int_{\p^*A \cap \Omega} \n_{\p^* A} \cdot \big(B - \jump{\nabla\widetilde w}\big) \,\dH[d-1]\bigg|
\\&~
\leq \widetilde{C} \big\|\n_{\p\A}\cdot\jump{\nabla\bar{u}}\big\|_{L^\infty(\p\A)} 
\sqrt{\mathcal{H}^{d-1}(\p\A)} \sqrt{E_{vol}[A|\A]}.
\end{aligned}
\end{equation}
\end{lemma}

We also rely on a regularity estimate for the transformed
chemical potential~$\widetilde w_h$.

\begin{lemma}[Schauder estimate for transformed chemical potential]
\label{lem:regularity_auxiliary_chemical_potential}
There exists a constant $C_{reg}(\A,\Lambda) \in (0,\infty)$ such that
\begin{equation}
\begin{aligned}
\label{eq:Schauder_estimate1}
\|\widetilde w_h\|_{C^{1,\frac{1}{2}}(\overline{\A \cap B_{\ell/2}(\p\A)})}
+ \|\widetilde w_h\|_{C^{1,\frac{1}{2}}(\overline{(\Omega\setminus\A) \cap B_{\ell/2}(\p\A)})} 
&\leq C_{reg}(\A,\Lambda).
\end{aligned}
\end{equation}
An analogous estimate holds for~$(\widetilde{w}_\vartheta)_h$.
\end{lemma}

Another useful input is given by energy estimates.

\begin{lemma}[Energy estimates for transformed chemical potentials]
\label{lem:energy_estimates}
Decompose $(w - \widetilde{w})_h = v^{(1)}_h + v^{(2)}_h$, where
$v^{(1)}_h \in H^1(\Omega)$ denotes the unique solution of
\begin{align}
\label{eq:auxiTransformedPDE1}
\Delta v^{(1)}_h &= 0
&& \text{in } \Omega \setminus \p\A,
\\
\mathrm{tr}_{\p\A} v^{(1)}_h
&= \big(\mathrm{tr}_{\p^* A} (w) + \nabla\cdot\xi\big)\circ (\Psi^h)^{-1}
&& \text{on } \p\A,
\\
(\n_{\p\Omega}\cdot\nabla) v^{(1)}_h &= 0 
&& \text{on } \p\Omega.
\end{align}
For any $\delta \in (0,1)$ one may choose $C \gg_{\delta} 1$
from~\eqref{eq:perturbative_regime1_again} such that
\begin{align}
\label{eq:energyEstimate1}
\int_{\Omega} |\nabla v^{(1)}_h|^2 \,dx 
&\leq (1 {+} \delta) \int_{\Omega} |\nabla (w - \widetilde w)_h|^2 \,dx,
\\
\label{eq:energyEstimate2}
\int_{\Omega} |\nabla (w - \widetilde w)_h|^2 \,dx 
&\leq (1 {+} \delta) \int_{\Omega} |\nabla (w - \widetilde w)|^2 \,dx.
\end{align}
There also exists a universal constant $\widetilde {C} > 0$ such that
\begin{equation}
\begin{aligned}
\label{eq:auxEstimate}
&\bigg|\,\int_{\Omega} \nabla v^{(1)}_h \cdot (\mathrm{Id} - a^h)\nabla \widetilde w_h\,dx\bigg|
\leq \widetilde {C}\bigg(\int_{\Omega} |\nabla v^{(1)}_h|^2 \,dx \bigg)^\frac{1}{2}
C_{reg}(\A,\Lambda)\sqrt{E[A|\A]},
\end{aligned}
\end{equation}
where $C_{reg}(\A,\Lambda)$ is the constant from Lemma~\ref{lem:regularity_auxiliary_chemical_potential}.
\end{lemma}

The final ingredient consists of an interpolation estimate,
transferring control in terms of the Hilbert space structure
on $H^{1/2}_{MS}(\p\A)$ to control in terms
of the standard $H^1$ and~$H^2$~norm on~$\p\A$ (and therefore, for our purposes,
to control in terms of our error functional and the dissipation).

\begin{lemma}[Interpolation estimate]
\label{lem:interpolation_estimate}
There exists a constant $C_{int}(\A,\Omega) \in (0,\infty)$
such that for all $f \in H^2(\p\A)$
with $\dashint_{\p\A} f \,\dH[d-1] = 0$ it holds
\begin{align}
\label{eq:interpolation_estimate}
\|f\|_{H^{1/2}_{MS}(\p\A)}
\leq C_{int}(\A,\Omega) \|f\|_{H^1(\p\A)}^\frac{1}{2}\|f\|_{H^2(\p\A)}^\frac{1}{2}.
\end{align}
\end{lemma}

We have now everything in place to proceed with the proofs.

\begin{proof}[Proof of Proposition~\ref{prop:stability_estimate_nonlocal_terms}, 
							Part I: Estimate~\eqref{eq:estimate_R_dissip}]
Next to the decomposition $(w {-} \widetilde{w})_h = v^{(1)}_h + v^{(2)}_h$ 
as defined in Lemma~\ref{lem:energy_estimates}, decompose also
$\bar u {-} \widetilde w_h = \bar{v}^{(1)}_h + \bar{v}^{(2)}_h$,
where $\bar{v}^{(1)}_h \in H^1(\Omega)$ denotes the unique solution of
\begin{align}
\label{eq:auxiTransformedPDE2}
\Delta \bar{v}^{(1)}_h &= 0
&& \text{in } \Omega \setminus \p\A,
\\
\mathrm{tr}_{\p\A} \bar{v}^{(1)}_h
&= \big(H_{\p\A} + (\nabla\cdot\xi)\circ (\Psi^h)^{-1}\big)
&& \text{on } \p\A,
\\
(\n_{\p\Omega}\cdot\nabla) \bar{v}^{(1)}_h &= 0 
&& \text{on } \p\Omega.
\end{align}
The main claim of the proof then is that the
argument for the estimate~\eqref{eq:estimate_R_dissip}
boils down to an estimate of
\begin{equation}
\begin{aligned}
\label{eq:reduction_argument}
Err
&:= \bigg|\int_{\p\A} \Big(v^{(1)}_h - \dashint_{\p\A} v^{(1)}_h \,\dH[d-1]\Big)
\n_{\p\A} \cdot \jump{\nabla \bar{v}^{(1)}_h} \,\dH[d-1]\bigg|.
\end{aligned}
\end{equation}
The estimate for~$Err$ itself will be a 
consequence of the Hilbert space structure defined on~$H^{1/2}_{MS}(\p\A)$,
the interpolation estimate~\eqref{eq:interpolation_estimate},
that our error functional controls the $H^1$~norm on~$\p\A$
of the height function~$h$, see~\eqref{eq:coercivity_Erel} and~\eqref{eq:coercivity_Ebulk},
and that the dissipation term
$\int_{\Omega} \frac{1}{2}|(\nabla w {-} \nabla\widetilde w)(\cdot,t)|^2 \,dx$ 
controls second-order derivatives of~$h$.

\textit{Step 1: Reduction argument.}
Recall that we want to estimate
\begin{align}
\widetilde{Err} := \bigg|\int_{\p^*A \cap \Omega} (w {-} \widetilde{w})
\,\n_{\p^*A}\cdot\big(B + \jump{\nabla\widetilde w}\big) \,\dH[d-1]\bigg|.
\end{align}
To this end, we first smuggle in the average~$\dashint_{\p\A} (w {-} \widetilde w)_h \,\dH[d-1]$
allowing us to deduce from~\eqref{eq:bound_averages}, \eqref{eq:average_B} 
and~\eqref{eq:comparable_length} that
\begin{equation}
\begin{aligned}
\label{eq:reduction1}
\widetilde{Err} &\leq \bigg|\int_{\p^*A \cap \Omega} \Big((w {-} \widetilde{w})
- \dashint_{\p\A} (w {-} \widetilde w)_h \,\dH[d-1]\Big)
\,\n_{\p^*A}\cdot\big(B + \jump{\nabla\widetilde w}\big) \,\dH[d-1]\bigg|
\\&~~~ 
+ \frac{\delta}{8} \int_{\Omega} |\nabla (w {-} \widetilde w)|^2 \,dx
+ \frac{\widetilde{C}}{\delta}  E[A,\mu|\A].
\end{aligned}
\end{equation}

In a second step, we will transform by an application of the 
area formula~\eqref{eq:areaFormulaSurfaceIntegral} 
the integral on~$\p^*A \cap \Omega$ into an integral on~$\p\A$.
Note to this end that, thanks to the $C^{1,\frac{1}{2}}$-regularity 
of~$\widetilde{w}_h$ established in Lemma~\ref{lem:regularity_auxiliary_chemical_potential},
the jump of~$(\nabla \widetilde{w})_h$ across~$\p\A$ in direction of~$\n_{\p^*A}\circ\big.(\Psi^h)^{-1}\big|_{\p\A}$
is equal to the jump of~$(\nabla \widetilde{w})_h$ across~$\p\A$ in direction of~$\n_{\p\A}$.
Hence, recalling also from~\eqref{eq:min_assumption9} that $B = 
(\bar{\n}_{\p\A} \cdot (\jump{\nabla\bar{u}}\circ P_{\p\A}))\bar{\n}_{\p\A}$,
we obtain by an application of~\eqref{eq:areaFormulaSurfaceIntegral}
as well as~\eqref{eq:normalInterfaceWeakSolution} that
\begin{align}
\nonumber
&\bigg|\int_{\p^*A \cap \Omega} \Big((w {-} \widetilde{w})
- \dashint_{\p\A} (w {-} \widetilde w)_h \,\dH[d-1]\Big)
\,\n_{\p^*A}\cdot\big(B + \jump{\nabla\widetilde w}\big) \,\dH[d-1]\bigg|
\\&
\label{eq:reduction2}
\leq \bigg|\int_{\p\A} \prod_{i=1}^{d-1} (1 {-} h\kappa^{i}_{\p\A})
\Big((w {-} \widetilde{w})_h
- \dashint_{\p\A} (w {-} \widetilde w)_h \,\dH[d-1]\Big)
\n_{\p\A}\cdot \jump{\nabla\bar u {-} (\nabla \widetilde w)_h} \,\dH[d-1]\bigg|
\\&~~~~
\nonumber
+ \sum_{i=1}^{d-1} 
\bigg|\int_{\p\A}  \prod_{j=1,j\neq i}^{d-1} (1 {-} h\kappa^{j}_{\p\A})
(\ta^{i}_{\p\A}\cdot\nabla_{\p\A})h
\\&~~~~~~~~~~~~~~~~
\nonumber
\times\Big((w {-} \widetilde{w})_h - \dashint_{\p\A} (w {-} \widetilde w)_h \,\dH[d-1]\Big)
\ta^{i}_{\p\A}
\cdot \jump{\nabla\bar u {-} (\nabla \widetilde w)_h} \,\dH[d-1]\bigg|
\\&
\nonumber
=: \widetilde{Err}_1 + \widetilde{Err}_2. 
\end{align}

We proceed with an estimate for~$\widetilde{Err}_2$. To this end, we first record
the following elementary identities:
\begin{align}
\label{eq:aux_reduction2}
\n_{\p\A} \cdot \jump{(\nabla \widetilde w)_h} 
&= \n_{\p\A} \cdot \jump{\nabla (\widetilde w_h)} 
&& \text{on } \p\A,
\\
\label{eq:aux_reduction3}
\ta^{i}_{\p\A} \cdot \jump{(\nabla \widetilde w)_h}
&= \Big(\frac{1}{1 {-} h\kappa^{i}_{\p\A}}\,\ta^{i}_{\p\A} 
- \frac{(\ta^{i}_{\p\A}\cdot\nabla_{\p\A})h}{1 {-} h\kappa^{i}_{\p\A}}\,\n_{\p\A}\Big)
\cdot \jump{\nabla (\widetilde w_h)}
&& \text{on } \p\A.
\end{align}
Indeed, the claims~\eqref{eq:aux_reduction2}--\eqref{eq:aux_reduction3}
are direct consequences of the chain rule in the form of
$(\nabla \widetilde w)_h = \nabla (\widetilde w_h)(\nabla\Psi^h \circ (\Psi^h)^{-1})$
and the formulas~\eqref{eq:grad_diffeo} and~\eqref{eq:inv_diffeo_short}.
Hence, we obtain from an application of the Cauchy--Schwarz inequality,
Young's inequality, the Sobolev--Poincar\'{e} trace inequality~\eqref{eq:Sobolev_Poincare_trace},
and the estimates~\eqref{eq:energyEstimate2}, 
\eqref{eq:Schauder_estimate1} and~\eqref{eq:coercivity_Erel}
\begin{align}
\label{eq:reduction3}
\widetilde{Err}_2 
\leq 
\frac{\delta}{8} \int_{\Omega} |\nabla (w {-} \widetilde w)|^2 \,dx
+ \frac{\widetilde{C}}{\delta}  E[A,\mu|\A].
\end{align}

It remains to control~$\widetilde{Err}_1$. Recalling that
$|\kappa^{i}_{\p\A}| \leq 1/\ell$, it follows analogously to~\eqref{eq:reduction3},
relying this time on~\eqref{eq:coercivity_Ebulk} instead of~\eqref{eq:coercivity_Erel}
and~\eqref{eq:aux_reduction2} instead of~\eqref{eq:aux_reduction3}, that
\begin{equation}
\begin{aligned}
\label{eq:reduction4}
\widetilde{Err}_1 
&\leq \bigg|\int_{\p\A} \Big((w {-} \widetilde{w})_h
- \dashint_{\p\A} (w {-} \widetilde w)_h \,\dH[d-1]\Big)
\n_{\p\A}\cdot \jump{\nabla(\bar u {-} \widetilde w_h)} \,\dH[d-1]\bigg|
\\&~~~
+ \frac{\delta}{8} \int_{\Omega} |\nabla (w {-} \widetilde w)|^2 \,dx
+ \frac{\widetilde{C}}{\delta}  E[A,\mu|\A].
\end{aligned}
\end{equation}
Recalling the decompositions $(w {-} \widetilde{w})_h = v^{(1)}_h + v^{(2)}_h$ 
and $\bar u {-} \widetilde w_h = \bar{v}^{(1)}_h + \bar{v}^{(2)}_h$,
where we in addition note that $v^{(2)}_h = \bar{v}^{(2)}_h = 0$ on~$\p\A$ by virtue of
the definitions of~$v^{(1)}_h$ and~$\bar{v}^{(1)}_h$, we further estimate
by triangle inequality and the definition~\eqref{eq:reduction_argument} of~$Err$
\begin{align}
\nonumber
&\bigg|\int_{\p\A} \Big((w {-} \widetilde{w})_h
- \dashint_{\p\A} (w {-} \widetilde w)_h \,\dH[d-1]\Big)
\n_{\p\A}\cdot \jump{\nabla(\bar u {-} \widetilde w_h)} \,\dH[d-1]\bigg|
\\&~
\label{eq:reduction5}
\leq Err
+ \bigg|\int_{\p\A} \Big(v^{(1)}_h - \dashint_{\p\A} v^{(1)}_h \,\dH[d-1]\Big)
\n_{\p\A}\cdot \jump{\nabla\bar{v}^{(2)}_h} \,\dH[d-1]\bigg|
\\&~ 
\nonumber
=: Err + \widetilde{Err}'_1. 
\end{align}
In order to estimate the remaining term~$\widetilde{Err}'_1$,
we record that
\begin{align}
\label{eq:transformedPDE4}
\Delta \bar{v}^{(2)}_h &= 
\nabla\cdot\big((a^h {-} \mathrm{Id})\nabla \widetilde{w}_h\big)
&& \text{in } \Omega \setminus \p\A,
\\
\mathrm{tr}_{\p\A} \bar{v}^{(2)}_h &= 0
&& \text{on } \p\A,
\\
(\n_{\p\Omega}\cdot\nabla)\bar{v}^{(2)}_h &= 0 
&& \text{on } \p\Omega.
\end{align}
Hence, testing in a first step the PDE satisfied by~$\bar{v}^{(2)}_h$
with the test function $v^{(1)}_h - \dashint_{\p\A} v^{(1)}_h \,\dH[d-1]$, 
and in a second step the PDE satisfied by~$v^{(1)}_h$ with
the test function~$\bar{v}^{(2)}_h$, we get
\begin{align}
\nonumber
&\int_{\p\A} \Big(v^{(1)}_h - \dashint_{\p\A} v^{(1)}_h \,\dH[d-1]\Big)
\n_{\p\A}\cdot \jump{\nabla\bar{v}^{(2)}_h} \,\dH[d-1]
\\&~ \nonumber
= \int_{\Omega} \nabla v^{(1)}_h \cdot \nabla\bar{v}^{(2)}_h \,dx
- \int_{\Omega} \nabla v^{(1)}_h  \cdot \big((a^h {-} \mathrm{Id})\nabla \widetilde{w}_h\big) \,dx
\\&~~~~ \nonumber
+ \int_{\p\A} \Big(v^{(1)}_h - \dashint_{\p\A} v^{(1)}_h \,\dH[d-1]\Big)
\n_{\p\A}\cdot (a^h{-}\mathrm{Id})\jump{\nabla\widetilde{w}_h} \,\dH[d-1]
\\&~ \label{eq:reduction6}
= - \int_{\Omega} \nabla v^{(1)}_h  \cdot \big((a^h {-} \mathrm{Id})\nabla \widetilde{w}_h\big) \,dx
\\&~~~~ \nonumber
+ \int_{\p\A} \Big((w {-} \widetilde{w})_h - \dashint_{\p\A} (w {-} \widetilde{w})_h \,\dH[d-1]\Big)
\n_{\p\A}\cdot (a^h{-}\mathrm{Id})\jump{\nabla\widetilde{w}_h} \,\dH[d-1].
\end{align}
In particular, based on the same arguments leading to, e.g., \eqref{eq:reduction3},
the definition of~$a^h$, the formulas~\eqref{eq:grad_diffeo} and~\eqref{eq:div_diffeo},
as well as the estimates from Lemma~\ref{lem:energy_estimates}, we deduce
from~\eqref{eq:reduction6} that
\begin{align}
\label{eq:reduction7}
\widetilde{Err}'_1 
&\leq \frac{\delta}{8} \int_{\Omega} |\nabla (w {-} \widetilde w)|^2 \,dx
+ \frac{\widetilde{C}}{\delta} E[A,\mu|\A].
\end{align}

In summary, comparing our estimates~\eqref{eq:reduction1}--\eqref{eq:reduction5}
and \eqref{eq:reduction7} with the claim~\eqref{eq:estimate_R_dissip},
we see that we indeed reduced matters to an estimate of~\eqref{eq:reduction_argument}.

\textit{Step 2: Estimate for~\eqref{eq:reduction_argument}.}
Note that~\eqref{eq:div_xi} and~\eqref{eq:inv_diffeo_short} 
directly entail
\begin{align*}
H_{\p\A} + (\nabla\cdot\xi)\circ (\Psi^h)^{-1}
= -  \sum_{i=1}^{d-1} \frac{\big(\kappa^i_{\p\A}\big)^2}{1 {-} h \kappa^i_{\p\A}} h
 =: f_{h}
\quad\text{on } \p\A.
\end{align*}
Furthermore, the very definition of the Hilbert space
structure on~$H^{1/2}_{MS}(\p\A)$ yields
\begin{align*}
Err = \bigg|\Big\langle 
v^{(1)}_h - \dashint_{\p\A} v^{(1)}_h \,\dH[d-1], 
f_h {-} \dashint_{\p\A} f_h \,\dH[d-1] 
\Big\rangle_{H^{1/2}_{MS}(\p\A)}\bigg|.
\end{align*}
Hence, from the Cauchy--Schwarz inequality, the definition of the norm on~$H^{1/2}_{MS}(\p\A)$, i.e.\
$\|v^{(1)}_h - \dashint_{\p\A} v^{(1)}_h \,\dH[d-1]\|^2_{H^{1/2}_{MS}(\p\A)}
= \int_{\Omega} |\nabla v^{(1)}_h|^2 \,dx$,
the estimates~\eqref{eq:energyEstimate1}--\eqref{eq:energyEstimate2}
we infer
\begin{align}
\label{eq:reduction8}
Err &\leq \frac{\delta}{16} \int_{\Omega} |\nabla (w {-} \widetilde w)|^2 \,dx
+ \frac{\widetilde{C}}{\delta}\bigg\|f_h {-} \dashint_{\p\A} f_h\bigg\|^2_{H^{1/2}_{MS}(\p\A)}.
\end{align}
Denoting by $\mathrm{L}_{\p\A}$ the Weingarten tensor of~$\p\A$ 
(see, e.g., \cite[Equation~(2.7), page~47]{Pruess2016}), we note that in invariant notation
$f_h = \mathrm{trace}(\mathrm{L}^2_{\p\A}(\mathrm{Id} {-} h\mathrm{L}_{\p\A})^{-1}) h$,
so that together with the interpolation estimate~\eqref{eq:interpolation_estimate}
and the estimates~\eqref{eq:coercivity_Erel}--\eqref{eq:coercivity_Ebulk}
we get as an upgrade of~\eqref{eq:reduction8}
\begin{align}
\label{eq:reduction9}
Err &\leq \frac{\delta}{16} \int_{\Omega} |\nabla (w {-} \widetilde w)|^2 \,dx
+ \frac{\widetilde{C}}{\delta} E[A,\mu|\A]
+ \frac{\widetilde{C}}{\delta}\sqrt{E[A,\mu|\A]}\|\mathrm{Hess}_{\p\A}h\|_{L^2(\p\A)}.
\end{align}
Furthermore, by $L^2 $ Cald\'{e}ron--Zygmund theory ($\p\A$ is closed)
\begin{align}
\|\mathrm{Hess}_{\p\A}h\|_{L^2(\p\A)} \leq \widetilde{C}'\|\Delta_{\p\A} h\|_{L^2(\p\A)}.
\end{align} 
Since by~\eqref{eq:div_xi}, \eqref{eq:meanCurvatureInterfaceWeakSolution},
\eqref{eq:coercivity_Erel}--\eqref{eq:coercivity_Ebulk} and~\eqref{eq:transformedPDE1} 
\begin{align}
\|\Delta_{\p\A} h\|_{L^2(\p\A)} \leq \widetilde{C}''
\big(\|(w {-} \widetilde{w})_h\|_{L^2(\p\A)} +
\sqrt{E[A,\mu|\A]}\big) + \delta'\widetilde{C}''\|\mathrm{Hess}_{\p\A}h\|_{L^2(\p\A)},
\end{align}
where $\delta' \in (0,1)$ will be specified in a moment,
and since by~\eqref{eq:bound_averages}, \eqref{eq:Sobolev_Poincare_trace}
and~\eqref{eq:energyEstimate2}
\begin{align}
\|(w {-} \widetilde{w})_h\|_{L^2(\p\A)} 
\leq \widetilde{C}'''\big(\|\nabla(w {-} \widetilde{w})\|_{L^2(\Omega)}+ \sqrt{E[A,\mu|\A]}\big),
\end{align}
it follows from choosing~$\delta' \ll_{\widetilde{C}',\widetilde{C}''} 1$
and the previous four displays that
\begin{align}
\label{eq:reduction10}
Err &\leq \frac{\delta}{8} \int_{\Omega} |\nabla (w {-} \widetilde w)|^2 \,dx
+ \frac{\widetilde{C}''''}{\delta} E[A,\mu|\A].
\end{align}
Together with the estimates from Step~1, this concludes the
proof of~\eqref{eq:estimate_R_dissip}.
\end{proof}

\begin{proof}[Proof of Lemma~\ref{lem:geometric_identities0}]
The identity~\eqref{eq:xi} simply follows from~\eqref{eq:min_assumption3}
and $|\nabla s_{\p\A}| = 1$. The identities~\eqref{eq:grad_xi} and~\eqref{eq:div_xi}
in turn follow from~\eqref{eq:NablaProjection},
\begin{align}
\label{eq:DeltaNormal}
\nabla_{\p\A} \n_{\p\A} = - \sum_{i=1}^{d-1} \kappa_i \ta^{i}_{\p\A} \otimes \ta^{i}_{\p\A},
\end{align}
and the chain rule in the form of $\nabla \xi = (\nabla_{\p\A}\n_{\p\A} \circ P_{\p\A})\nabla P_{\p\A}$.
Since~\eqref{eq:CoareaFormulaTubularNeighborhoodDiffeo} is simply a consequence
of the coarea~formula applied to the tubular neighborhood diffeomorphism
\begin{align}
\label{eq:tub_nbhd_diffeo2}
\Phi\colon B_{\ell/2}(\p\A) \to \p\A \times (-\ell/2,\ell/2),
\quad x \mapsto (P_{\p\A}(x),s_{\p\A}(x)),
\end{align}
it only remains to prove~\eqref{eq:NablaProjection}. 

To this end, we first note that
\begin{align}
\label{eq:projection}
P_{\p\A}(x) = x - s_{\p\A}(x)\bar{\n}_{\p\A}(x),
\quad x \in B_\frac{\ell}{2}(\p\A),
\end{align}
and therefore
\begin{align}
\label{eq:grad_projection}
\nabla P_{\p\A} = \big(\mathrm{Id} - \bar{\n}_{\p\A} \otimes \bar{\n}_{\p\A}\big)
- s_{\p\A}\nabla \bar{\n}_{\p\A}
\quad \text{in } B_\frac{\ell}{2}(\p\A).
\end{align}
In particular, $\nabla P_{\p\A}$ is symmetric (recall that 
$\nabla \bar{\n}_{\p\A} = \mathrm{Hess} s_{\p\A}$), so that
by an application of the chain rule
(again in the form of $\nabla \bar{\n}_{\p\A} = (\nabla_{\p\A}\n_{\p\A} \circ P_{\p\A})\nabla P_{\p\A}$)
it follows that $(\bar{\ta}^{i}_{\p\A} \cdot \nabla) P_{\p\A} = \frac{1}{1 {-} s_{\p\A}\bar{\kappa}_i} \bar{\ta}^{i}_{\p\A}$.
Since furthermore $(\nabla P_{\p\A})^\mathsf{T}\bar{\n}_{\p\A} = (\bar{\n}_{\p\A} \cdot\nabla)P_{\p\A} = 0$,
we indeed obtain~\eqref{eq:NablaProjection}.
\end{proof}

\begin{proof}[Proof of Lemma~\ref{lem:geometric_identities}]
Thanks to Lemma~\ref{lem:geometric_identities0},
the formulas~\eqref{eq:grad_diffeo}--\eqref{eq:div_diffeo} 
are just a straightforward computation starting
from the definition~\eqref{eq:def_diffeo}.
The formulas~\eqref{eq:diffeo_short} and~\eqref{eq:inv_diffeo_short}
in turn directly follow from $\supp\bar\zeta \subset [-1/2,1/2]$, the definition~\eqref{eq:def_plateau_cutoff},
and~\eqref{eq:perturbative_regime1_again}. 
Furthermore, \eqref{eq:areaFormulaSurfaceIntegral}
is just the area formula applied to the map 
\begin{align}
\big.(\Psi^h)^{-1}\big|_{\p\A}\colon \p\A \to \p^*A.
\end{align}
Next, we claim that (in the sense of weak derivatives)
\begin{align}
\label{eq:GibbsThomsonHeightFunction}
\mathrm{tr}_{\p^*A}(w) \n_{\p^*A} = \nabla_{\p^*A} 
\cdot (\mathrm{Id} - \n_{\p^*A} \otimes \n_{\p^*A}).
\end{align}
Indeed, this is true thanks to the Gibbs--Thomson law~\eqref{WeakFormGibbsThomson}.
Since the right hand side of~\eqref{eq:GibbsThomsonHeightFunction}
is precisely the mean curvature vector of~$\p^*A$, the 
claims~\eqref{eq:normalInterfaceWeakSolution} 
and~\eqref{eq:meanCurvatureInterfaceWeakSolution} can therefore
be directly inferred from~\cite[Section~2.2]{Pruess2016}.
\end{proof}

\begin{proof}[Proof of Lemma~\ref{lem:error_control_perturbative_regime}]
The estimates~\eqref{eq:comparable_length}
simply follow from the area formula~\eqref{eq:areaFormulaSurfaceIntegral}
and~\eqref{eq:perturbative_regime1_again}.
For the proof of~\eqref{eq:coercivity_Erel},
note that~\eqref{eq:xi} and~\eqref{eq:normalInterfaceWeakSolution} imply
\begin{align}
(\n_{\p^*A} \cdot \xi) \cdot (\Psi^{h})^{-1} = \frac{1}
{\sqrt{1 {+} \sum_{i=1}^{d-1} \big(\frac{(\ta^{i}_{\p\A}\cdot\nabla_{\p\A})h}{1 {-} h\kappa^{i}_{\p\A}}\big)^2}},
\end{align}
which in turn entails~\eqref{eq:coercivity_Erel} due to~\eqref{eq:areaFormulaSurfaceIntegral}
and~\eqref{eq:perturbative_regime1_again}.
Finally, by the coarea formula~\eqref{eq:CoareaFormulaTubularNeighborhoodDiffeo}
as well as the assumptions~\eqref{eq:min_assumption4} and~\eqref{eq:perturbative_regime1_again}
\begin{align}
E_{vol}[A|\A] &= 
\int_{\p\A} \int_0^h  \frac{s}{\ell^2}
\prod_{i=1}^{d-1} (1 {-} s\kappa^i_{\p\A}) \,dsd\mathcal{H}^{d-1},
\end{align}
so that~\eqref{eq:coercivity_Ebulk} is now a direct consequence of~\eqref{eq:perturbative_regime1_again}.
\end{proof}

\begin{proof}[Proof of Lemma~\ref{lem:transformed_chemical_potentials}]
These assertions are immediate consequences of the well-known 
transformation formulas for PDEs in distributional form.
\end{proof}

\begin{proof}[Proof of Lemma~\ref{lem:bound_averages}]
We proceed in three steps.

\textit{Step 1: Proof of~\eqref{eq:bound_averages}, Part I.}
Adding zero, recalling~\eqref{eq:auxChemPotential2}, 
and exploiting the Gibbs--Thomson law~\eqref{WeakFormGibbsThomson}
with admissible test function~$\xi$, we obtain
\begin{align}
\nonumber
\int_{\p^* A \cap \Omega} w - \widetilde w \,\dH[d-1]
&= \int_{\p^* A \cap \Omega} w + \nabla\cdot\xi \,\dH[d-1] 
\\&
\nonumber
= \int_{\p^* A \cap \Omega} w \n_{\p^*A}\cdot\xi \,\dH[d-1]
+ \int_{\p^* A \cap \Omega} w (1 - \n_{\p^*A}\cdot\xi) \,\dH[d-1]
\\&~~~
\nonumber
+ \int_{\p^* A \cap \Omega} (\mathrm{Id} {-} \n_{\p^*A}\otimes \n_{\p^*A}):\nabla\xi \,\dH[d-1]
\\&~~~
\nonumber
+ \int_{\p^* A \cap \Omega} \n_{\p^*A}\otimes \n_{\p^*A}:\nabla\xi \,\dH[d-1]
\\& 
\label{eq:aux_averages1}
= \int_{\p^* A \cap \Omega} w (1 - \n_{\p^*A}\cdot\xi) \,\dH[d-1]
\\&~~~
\nonumber
+ \int_{\p^* A \cap \Omega} \n_{\p^*A}\otimes \n_{\p^*A}:\nabla\xi \,\dH[d-1].
\end{align}
Based on~\eqref{eq:xi} and~\eqref{eq:grad_xi}, we also deduce that along~$\p^*A$
\begin{align}
\nonumber
\n_{\p^*A}\otimes \n_{\p^*A}:\nabla\xi
&= - \sum_{i=1}^{d-1} \frac{\bar{\kappa}^{i}_{\p\A}}{1 {-} \bar{h}\bar{\kappa}^{i}_{\p\A}} 
(\n_{\p^*A}\cdot\bar{\ta}^{i}_{\p\A})^2
\\&
\label{eq:aux_averages2}
= (\nabla\cdot\xi) (1 - \n_{\p^*A}\cdot\xi)
\\&~~~ \nonumber
+ (\nabla\cdot\xi) \n_{\p^*A}\cdot\xi (1 - \n_{\p^*A}\cdot\xi)
\\&~~~ \nonumber
+ \sum_{i=1}^{d-1} \frac{\bar{\kappa}^{i}_{\p\A}}{1 {-} \bar{h}\bar{\kappa}^{i}_{\p\A}} 
\sum_{j=1,j\neq i}^{d-1} (\n_{\p^*A}\cdot\bar{\ta}^{j}_{\p\A})^2.
\end{align}
Hence, inserting~\eqref{eq:aux_averages2} back into~\eqref{eq:aux_averages1}
and adding two times another zero yields (recalling in the process also~\eqref{eq:auxChemPotential2}
and~\eqref{eq:div_xi})
\begin{align}
\nonumber
\int_{\p^* A \cap \Omega} w - \widetilde w \,\dH[d-1]
&= \int_{\p^* A \cap \Omega} (w - \widetilde w) (1 - \n_{\p^*A}\cdot\xi) \,\dH[d-1]
\\&~~~
\nonumber
+ \int_{\p^* A \cap \Omega} 
(\nabla\cdot\xi) \n_{\p^*A}\cdot\xi (1 - \n_{\p^*A}\cdot\xi) \,\dH[d-1]
\\&~~~
\nonumber
+ \int_{\p^* A \cap \Omega} \sum_{i=1}^{d-1} \frac{\bar{\kappa}^{i}_{\p\A}}{1 {-} \bar{h}\bar{\kappa}^{i}_{\p\A}} 
\sum_{j=1,j\neq i}^{d-1} (\n_{\p^*A}\cdot\bar{\ta}^{j}_{\p\A})^2 \,\dH[d-1]
\\&
\nonumber
= \bigg(\,\dashint_{\p^* A \cap \Omega} w - \widetilde w \,\dH[d-1]\bigg) E_{rel}[A,\mu|\A]
\\&~~~
\nonumber
+ \int_{\p^* A \cap \Omega} \Big((w - \widetilde w)
- \dashint_{\p^* A \cap \Omega} (w - \widetilde w) \,\dH[d-1]\Big) 
(1 - \n_{\p^*A}\cdot\xi) \,\dH[d-1]
\\&~~~
\nonumber
+ \int_{\p^* A \cap \Omega} 
(\nabla\cdot\xi) \n_{\p^*A}\cdot\xi (1 - \n_{\p^*A}\cdot\xi) \,\dH[d-1]
\\&~~~
\nonumber
+ \int_{\p^* A \cap \Omega} \sum_{i=1}^{d-1} \frac{\bar{\kappa}^{i}_{\p\A}}{1 {-} \bar{h}\bar{\kappa}^{i}_{\p\A}} 
\sum_{j=1,j\neq i}^{d-1} (\n_{\p^*A}\cdot\bar{\ta}^{j}_{\p\A})^2 \,\dH[d-1]
\\& 
\label{eq:aux_averages3}
=: I + II + III + IV.
\end{align}

By virtue of~\eqref{eq:comparable_length}, \eqref{eq:small_error}
and $\sum_{j=1,j\neq i}^{d-1} (\n_{\p^*A}\cdot\bar{\ta}^{j}_{\p\A})^2
\leq 2(1 - \n_{\p^*A}\cdot\xi)$
we then obtain for $M \gg 1$
\begin{align}
\label{eq:aux_averages4}
|I + III| &\leq \frac{1}{4} \bigg|\int_{\p^* A \cap \Omega} w - \widetilde w \,\dH[d-1]\bigg|
+ \frac{\widetilde{C}}{\ell} E_{rel}[A,\mu|\A],
\end{align}
so that we obtain the following upgrade of~\eqref{eq:aux_averages3}
\begin{align}
\label{eq:aux_averages5}
\bigg|\int_{\p^* A \cap \Omega} w - \widetilde w \,\dH[d-1]\bigg|
\leq \widetilde{C}|II| + \frac{\widetilde{C}}{\ell} E_{rel}[A,\mu|\A].
\end{align}

It remains to bound the contribution from~$II$, cf.\ \eqref{eq:aux_averages3}.
Because of $(1-n_{\p^*A}\cdot\xi)^2 \leq 4(1-n_{\p^*A}\cdot\xi)$ we get
from an application of Cauchy--Schwarz inequality
\begin{align*}
|II| &\leq 2\bigg(\int_{\p^*A\cap\Omega} \Big| (w{-}\widetilde w) - 
\dashint_{\p^*A\cap\Omega} (w {-}\widetilde{w}) \,\dH[d-1]
\Big|^2\,\dH[d-1]\bigg)^\frac{1}{2} E_{rel}^\frac{1}{2}[A,\mu|\A],
\end{align*}
so that by adding zero we obtain
\begin{align}
\nonumber
&|II| 
\\&
\nonumber
\leq 2\bigg(\int_{\p^*A\cap\Omega} \Big| (w{-}\widetilde w) - 
\dashint_{\p\A} (w {-}\widetilde{w})_h \,\dH[d-1]
\Big|^2\,\dH[d-1]\bigg)^\frac{1}{2}E_{rel}^\frac{1}{2}[A,\mu|\A]
\\&~~~
\nonumber
+ 2\sqrt{\mathcal{H}^{d-1}(\p^*A\cap\Omega)}
\bigg|\,\dashint_{\p^*A\cap\Omega} (w {-}\widetilde{w}) \,\dH[d-1]
- \dashint_{\p\A} (w {-}\widetilde{w})_h \,\dH[d-1] \bigg| E_{rel}^\frac{1}{2}[A,\mu|\A]
\\&
\nonumber
\leq 2\bigg(\int_{\p^*A\cap\Omega} \Big| (w{-}\widetilde w) - 
\dashint_{\p\A} (w {-}\widetilde{w})_h \,\dH[d-1]
\Big|^2\,\dH[d-1]\bigg)^\frac{1}{2} E_{rel}^\frac{1}{2}[A,\mu|\A]
\\&~~~
\nonumber
+ 2\frac{\sqrt{\mathcal{H}^{d-1}(\p^*A\cap\Omega)}}{\mathcal{H}^{d-1}(\p\A)}
\bigg|\int_{\p^*A\cap\Omega} (w {-} \widetilde{w}) \,\dH[d-1]
- \int_{\p\A} (w {-} \widetilde{w})_h \,\dH[d-1]\bigg| E_{rel}^\frac{1}{2}[A,\mu|\A]
\\&~~~
\nonumber
+ 2\sqrt{\mathcal{H}^{d-1}(\p^*A\cap\Omega)} \bigg|
\frac{1}{{\mathcal{H}^{d-1}(\p^*A\cap\Omega)}}-\frac{1}{{\mathcal{H}^{d-1}(\p\A)}}\bigg|
E_{rel}^\frac{1}{2}[A,\mu|\A]
\bigg|\int_{\p^*A\cap\Omega}  (w{-}\widetilde w) \,\dH[d-1]\bigg| 
\\&
\label{eq:aux_averages6}
=: II' + II'' + II'''.
\end{align}

We estimate term by term. For the first term, we simply make use
of the area formula~\eqref{eq:areaFormulaSurfaceIntegral}
(abbreviating in the following the associated area factor by~$\mathcal{J}^{d-1}_h$), 
the Sobolev--Poincar\'{e} trace inequality~\eqref{eq:Sobolev_Poincare_trace}
as well as the estimate~\eqref{eq:energyEstimate2} to deduce
\begin{align}
\nonumber
II' &\leq
\widetilde{C} \bigg(\int_{\p\A} \Big| (w{-}\widetilde w)_h - 
\dashint_{\p\A} (w {-}\widetilde{w})_h \,\dH[d-1]
\Big|^2\,\dH[d-1]\bigg)^\frac{1}{2} E_{rel}^\frac{1}{2}[A,\mu|\A]
\\&
\label{eq:aux_averages7}
\leq \widetilde{C} \sqrt{\mathcal{H}^{d-1}(\p\A)} C_{PS}(\A) E_{rel}^\frac{1}{2}[A,\mu|\A] 
\bigg(\int_{\Omega} |\nabla(w{-}\widetilde{w})|^2\,dx\bigg)^\frac{1}{2}.
\end{align}
For an estimate of the second term, we first recognize
through an application of the area formula~\eqref{eq:areaFormulaSurfaceIntegral}
(abbreviating in the following the associated area factor by~$\mathcal{J}^{d-1}_h$) 
and adding zero that
\begin{align}
\nonumber
&\bigg|\int_{\p^*A\cap\Omega} (w {-} \widetilde{w}) \,\dH[d-1]
- \int_{\p\A} (w {-} \widetilde{w})_h \,\dH[d-1]\bigg|
\\&
\nonumber
= \bigg|\int_{\p\A} \big(\mathcal{J}^{d-1}_h - 1\big) 
(w {-} \widetilde{w})_h \,\dH[d-1] \bigg|
\\&
\label{eq:aux_averages8}
\leq \bigg|\int_{\p\A} \big(\mathcal{J}^{d-1}_h - 1\big) 
\Big((w {-} \widetilde{w})_h - \dashint_{\p\A} (w {-} \widetilde{w})_h\,\dH[d-1]\Big) \,\dH[d-1] \bigg|
\\&~~~
\nonumber
+ \frac{1}{\mathcal{H}^{d-1}(\p\A)} \bigg(\int_{\p\A} 
\Big|\mathcal{J}^{d-1}_h {-} 1\Big|\,\dH[d-1]\bigg)
\bigg|\int_{\p^*A\cap\Omega} (w {-} \widetilde{w}) \,\dH[d-1]
- \int_{\p\A} (w {-} \widetilde{w})_h \,\dH[d-1]\bigg|
\\&~~~
\nonumber
+ \frac{1}{\mathcal{H}^{d-1}(\p\A)} \bigg(\int_{\p\A} \Big|\mathcal{J}^{d-1}_h - 1\Big|\,\dH[d-1]\bigg)
\bigg|\int_{\p^*A\cap\Omega} (w {-} \widetilde{w}) \,\dH[d-1] \bigg|.
\end{align}
Furthermore, thanks to the estimates~\eqref{eq:coercivity_Erel}--\eqref{eq:coercivity_Ebulk},
a Taylor expansion, the assumption~\eqref{eq:perturbative_regime1_again}, and Jensen's inequality
\begin{align}
\label{eq:aux_averages9}
\int_{\p\A} \Big|\mathcal{J}^{d-1}_h - 1\Big|\,\dH[d-1]
\leq \widetilde{C} \sqrt{\mathcal{H}^{d-1}(\p\A)} \sqrt{E[A,\mu|\A]}.
\end{align}
Similarly,
\begin{align}
\label{eq:aux_averages10}
\bigg(\int_{\p\A} \Big|\mathcal{J}^{d-1}_h - 1\Big|^2\,\dH[d-1]\bigg)^\frac{1}{2}
\leq \widetilde{C} \sqrt{E[A,\mu|\A]}.
\end{align}
Hence, inserting the estimates~\eqref{eq:aux_averages9} and~\eqref{eq:aux_averages10}
back into~\eqref{eq:aux_averages8} (the second after an application of Cauchy--Schwarz inequality
in the first right hand side term of~\eqref{eq:aux_averages8})
and recalling the argument for~\eqref{eq:aux_averages7},
we may upgrade~\eqref{eq:aux_averages8} to
\begin{align}
\nonumber
&\bigg|\int_{\p^*A\cap\Omega} (w {-} \widetilde{w}) \,\dH[d-1]
- \int_{\p\A} (w {-} \widetilde{w})_h \,\dH[d-1]\bigg|
\\&
\label{eq:aux_averages11}
\leq \widetilde{C} \sqrt{\mathcal{H}^{d-1}(\p\A)} C_{PS}(\A) E_{rel}^\frac{1}{2}[A,\mu|\A] 
\bigg(\int_{\Omega} |\nabla(w{-}\widetilde{w})|^2\,dx\bigg)^\frac{1}{2}
\\&~~~
\nonumber
+ \widetilde{C} \frac{\sqrt{E[A,\mu|\A]}}{\sqrt{\mathcal{H}^{d-1}(\p\A)}}
\bigg|\int_{\p^*A\cap\Omega} (w {-} \widetilde{w}) \,\dH[d-1]
- \int_{\p\A} (w {-} \widetilde{w})_h \,\dH[d-1]\bigg|
\\&~~~
\nonumber
+ \widetilde{C} \frac{\sqrt{E[A,\mu|\A]}}{\sqrt{\mathcal{H}^{d-1}(\p\A)}}
\bigg|\int_{\p^*A\cap\Omega} (w {-} \widetilde{w}) \,\dH[d-1] \bigg|.
\end{align}
Overall, it follows now from~\eqref{eq:aux_averages11}, \eqref{eq:comparable_length},
and assumption~\eqref{eq:small_error} that for $M \gg 1$
\begin{align}
\label{eq:aux_averages12}
II'' &\leq \widetilde{C} C_{PS}(\A) E_{rel}[A,\mu|\A] 
\bigg(\int_{\Omega} |\nabla(w{-}\widetilde{w})|^2\,dx\bigg)^\frac{1}{2}
\\&~~~
\nonumber
+ \frac{1}{4} \bigg|\int_{\p^*A\cap\Omega} (w {-} \widetilde{w}) \,\dH[d-1] \bigg|.
\end{align}
It follows also immediately from \eqref{eq:comparable_length},
and assumption~\eqref{eq:small_error} that for $M \gg 1$
\begin{align}
\label{eq:aux_averages13}
II''' &\leq \frac{1}{4} \bigg|\int_{\p^*A\cap\Omega} (w {-} \widetilde{w}) \,\dH[d-1] \bigg|.
\end{align}

Hence, plugging the estimates~\eqref{eq:aux_averages7}, \eqref{eq:aux_averages12} and~\eqref{eq:aux_averages13}
back into~\eqref{eq:aux_averages6}, we obtain as an upgrade of~\eqref{eq:aux_averages6}
\begin{align}
\label{eq:aux_averages14}
|II| &\leq 
\widetilde{C} \sqrt{\mathcal{H}^{d-1}(\p\A)} C_{PS}(\A) E_{rel}^\frac{1}{2}[A,\mu|\A] 
\bigg(\int_{\Omega} |\nabla(w{-}\widetilde{w})|^2\,dx\bigg)^\frac{1}{2}
\\&~~~
\nonumber
+ \widetilde{C} C_{PS}(\A) E_{rel}[A,\mu|\A] 
\bigg(\int_{\Omega} |\nabla(w{-}\widetilde{w})|^2\,dx\bigg)^\frac{1}{2}
\\&~~~
\nonumber
+ \frac{1}{2} \bigg|\int_{\p^*A\cap\Omega} (w {-} \widetilde{w}) \,\dH[d-1] \bigg|.
\end{align}
In view of~\eqref{eq:aux_averages5} and assumption~\eqref{eq:small_error}, 
it therefore remains to choose $M \gg 1$ to conclude with~\eqref{eq:bound_averages}
in terms of the average $\dashint_{\p^*A\cap\Omega} w {-} \widetilde w \,\dH[d-1]$.

\textit{Step 2: Proof of~\eqref{eq:bound_averages}, Part II.}
Decomposing
\begin{align*}
\int_{\p\A} (w - \widetilde{w})_h \,\dH[d-1] 
&= \int_{\p^*A\cap\Omega} (w {-} \widetilde{w}) \,\dH[d-1]
\\&~~~
+ \int_{\p\A} \bigg( (w {-} \widetilde{w}) - \dashint_{\p\A} (w {-} \widetilde{w}) \,\dH[d-1]\bigg)
\big(1 - \mathcal{J}^{d-1}_h\big) \,\dH[d-1]
\\&~~~
+ \bigg(\,\dashint_{\p\A} (w {-} \widetilde{w}) \,\dH[d-1] \bigg) 
\int_{\p\A} \big(1 - \mathcal{J}^{d-1}_h\big) \,\dH[d-1],
\end{align*}
we may infer the asserted estimate~\eqref{eq:bound_averages}
for $\dashint_{\p\A} (w - \widetilde{w})_h \,\dH[d-1]$ from
the corresponding estimate for $\dashint_{\p^*A\cap\Omega} (w {-} \widetilde w) \,\dH[d-1]$
from Step~1 and the arguments used to derive~\eqref{eq:aux_averages14}. 

\textit{Step 3: Proof of~\eqref{eq:average_B}.}
First, we simply recognize that by means of~\eqref{eq:auxChemPotential1}--\eqref{eq:auxChemPotential3}
\begin{align*}
\int_{\p^*A\cap\Omega} \n_{\p^*A}\cdot\jump{\nabla\widetilde{w}} \,\dH[d-1] = 0.
\end{align*}
Second, by assumption~\eqref{eq:min_assumption9}, the 
area formula~\eqref{eq:CoareaFormulaTubularNeighborhoodDiffeo}
(abbreviating in the following the associated coarea factor by~$\mathcal{C}^{d-1}_h$), 
and the splitting $\n_{\p^*\A} = (\n_{\p^*\A}\cdot\bar{\n}_{\p\A})\bar{\n}_{\p\A}
+ (\mathrm{Id} {-} \bar{\n}_{\p\A} \otimes \bar{\n}_{\p\A}) \n_{\p^*\A}$, we obtain
\begin{align*}
\int_{\p^*A\cap \Omega} \n_{\p^*A}\cdot B\,\dH[d-1]
&= \int_{\p\A} \bar{\n}_{\p\A} \cdot \jump{\nabla\bar{u}}
\mathcal{C}^{d-1}_h \,\dH[d-1]
\\&
= \int_{\p\A} \bar{\n}_{\p\A} \cdot \jump{\nabla\bar{u}}
\big(\mathcal{C}^{d-1}_h - 1\big) \,\dH[d-1].
\end{align*}
Hence, 
\begin{align*}
\bigg| \int_{\p^*A\cap \Omega} \n_{\p^*A}\cdot B\,\dH[d-1] \bigg|
\leq \widetilde{C} \big\|\n_{\p\A}\cdot\jump{\nabla\bar{u}}\big\|_{L^\infty(\p\A)} 
\sqrt{\mathcal{H}^{d-1}(\p\A)} \sqrt{E_{vol}[A|\A]}, 
\end{align*}
so that the claim~\eqref{eq:average_B} follows from the previous three displays.
\end{proof}

\begin{proof}[Proof of Lemma~\ref{lem:regularity_auxiliary_chemical_potential}]
By a change of variables~$\Psi^h$, it follows from~\eqref{eq:auxChemPotential1}--\eqref{eq:auxChemPotential3} 
and~\eqref{eq:div_xi}
\begin{align}
\label{eq:auxTransformedChemPotential1}
- \nabla\cdot (a^h\nabla \widetilde w_h) &= 0 
&& \text{in } \Omega \setminus \p\A,
\\ \label{eq:auxTransformedChemPotential2}
\mathrm{tr}_{\p\A}\widetilde w_h &= -(\nabla\cdot\xi)\circ(\Psi^h)^{-1}
&& \text{on } \p\A \cap \Omega,
\\ \label{eq:auxTransformedChemPotential3}
(\n_{\p\Omega} \cdot \nabla) \widetilde w_h &= 0
&& \text{on } \p\Omega.
\end{align} 
Hence, the regularity estimate~\eqref{eq:Schauder_estimate1} is
just a simple consequence of standard Schauder theory applied to the two regular open sets
$\Omega\setminus\overline{\A}$ and $\A$ (see, e.g., \cite[Theorem~5.21]{Giaquinta2012}). 
The corresponding estimates
depend on the respective domain, the ellipticity constant of~$a^h$, the $C^{0,1/2}$ H\"{o}lder norm
of~$a^h$ on the closure of the respective domain, and finally the $C^{1,1/2}$ H\"{o}lder norm of the 
Dirichlet data $-(\nabla\cdot\xi)\circ(\Psi^h)^{-1}$.
We therefore only need to ensure that this data is bounded by a constant of required form~$C=C(\A,\Lambda)$.
However, recalling the formulas~\eqref{eq:grad_diffeo}, \eqref{eq:div_diffeo}
and $a^h = \frac{1}{|\det \Psi^h|}(\nabla\Psi^h)^\mathsf{T}\nabla\Psi^h$,
this claim follows since $\|h\|_{C^{1,1/2}(\p\A)} \leq C(\A,\Lambda)$
due to~\eqref{eq:regularity_graph}.
\end{proof}

\begin{proof}[Proof of Lemma~\ref{lem:energy_estimates}]
The estimate~\eqref{eq:energyEstimate2} follows from the chain rule
and assumption~\eqref{eq:perturbative_regime1_again}. The estimate~\eqref{eq:energyEstimate1}
in turn follows from the decomposition~$(w-\widetilde{w})_h=v^{(1)}_h+v^{(2)}_h$
and the energy estimate for~$v^{(2)}_h$ satisfying the PDE
\begin{align}
\label{eq:transformedPDE100}
\Delta v^{(2)}_h &= 
\nabla\cdot\big((\mathrm{Id}-a^h)\nabla (w - \widetilde{w})_h\big)
&& \text{in } \Omega \setminus \p\A,
\end{align}
with boundary conditions
\begin{align}
\mathrm{tr}_{\p\A} v^{(2)}_h &= 0
&& \text{on } \p\A,
\\
(\n_{\p\Omega}\cdot\nabla)v^{(2)}_h &= 0 
&& \text{on } \p\Omega.
\end{align}

For a proof of~\eqref{eq:auxEstimate}, we first note that by definition of~$a^h$
it holds $\supp (\mathrm{Id} - a^h) \subset B_{\ell/2}(\p\A)$. Hence, by means of the
coarea formula~\eqref{eq:CoareaFormulaTubularNeighborhoodDiffeo}, we deduce~\eqref{eq:auxEstimate}
from Cauchy--Schwarz inequality, \eqref{eq:Schauder_estimate1},
$|\mathrm{Id} - a^h| \leq \widetilde{C} (|\nabla_{\p\A}h| 
+ \max_{i=1,\ldots,d{-}1}|\kappa^{i}_{\p\A} h|)$ in $B_{\ell/2}(\p\A)$,
and the estimates~\eqref{eq:coercivity_Erel}--\eqref{eq:coercivity_Ebulk}.
\end{proof}

\begin{proof}[Proof of Lemma~\ref{lem:interpolation_estimate}]
We argue by contradiction that there exists $C=C(\A,\Omega)$
such that for all $f \in H^2(\p\A)$ with $\dashint_{\p\A} f \,\dH[d-1] = 0$
it holds
\begin{align}
\label{eq:auxInterpolation}
\|f\|_{H^{1/2}_{MS}(\p\A)} \leq C\|f\|_{[H^1(\p\A),H^2(\p\A)]_{\frac{1}{2}}},
\end{align}
where $[H^1,H^2(\p\A)]_{\frac{1}{2}} = H^{\frac{3}{2}}(\p\A)$ denotes the (complex) interpolation
space with parameter $\theta=\frac{1}{2}$ between $H^1(\p\A)$ and $H^2(\p\A)$. 
Of course, \eqref{eq:auxInterpolation} implies the claim of Lemma~\ref{lem:interpolation_estimate}.

For each $k \in \mathbb{N}$,
assume that we may find a map $f_k \in H^2(\p\A)$
satisfying $\dashint_{\p\A} f_k \,\dH[d-1] = 0$
such that
\begin{align}
\label{eq:assump_contradiction}
\|\nabla u_k\|_{L^2(\Omega)}
= \|f_k\|_{H^{1/2}_{MS}(\p\A)} = 1 > k\|f_k\|_{[H^1(\p\A),H^2(\p\A)]_{\frac{1}{2}}},
\end{align} 
where $u_k \in H^1(\Omega)$ denotes the associated chemical potential,
i.e., 
\begin{align*}
\Delta u_k &= 0 &&\text{in } \Omega\setminus\p\A,
\\
u_k &= f_k &&\text{on } \p\A,
\\
(\n_{\p\Omega}\cdot\nabla)u_k &= 0 &&\text{on } \p\Omega. 
\end{align*}
We deduce from the elliptic regularity estimate
\begin{align*}
\|u_k\|_{H^2(\Omega\setminus\p\A)}
\lesssim_{\A,\Omega} \|u_k\|_{H^1(\Omega)} + \|f_k\|_{[H^1(\p\A),H^2(\p\A)]_{\frac{1}{2}}}
\end{align*}
and Poincar\'{e} inequality (recall that $\dashint_{\p\A} u_k \,\dH[d-1]=0$
by construction) that
\begin{align*}
\sup_{k \in \mathbb{N}} \|u_k\|_{H^2(\Omega\setminus\p\A)} \lesssim_{\A,\Omega} 1.
\end{align*}
Hence, modulo taking a subsequence, 
$(u_k)_k$ weakly converges in $H^2(\Omega\setminus\p\A)$
to some $u \in H^2(\Omega\setminus\p\A) \cap H^1(\Omega)$ such that 
$u_k \to u$ strongly in $H^1(\Omega)$.
However, by~\eqref{eq:assump_contradiction},
$f_k \to 0$ strongly in $[H^{1}(\p\A),H^2(\p\A)]_{\frac{1}{2}}$, so that 
$u$ satisfies
\begin{align*}
\Delta u &= 0 &&\text{in } \Omega\setminus\p\A,
\\
u &= 0 &&\text{on } \p\A,
\\
(\n_{\p\Omega}\cdot\nabla)u &= 0 &&\text{on } \p\Omega. 
\end{align*}
Hence, $u \equiv 0$ in contradiction to $\lim_{k\to\infty} 
\|\nabla u_k\|_{L^2(\Omega)} = 1 = \|\nabla u\|_{L^2(\Omega)}$.
\end{proof}

\begin{proof}[Proof of Proposition~\ref{prop:stability_estimate_nonlocal_terms}, 
							Part II: Estimates~\eqref{eq:estimate_A_dissip1} and~\eqref{eq:estimate_A_dissip2}]
The argument is naturally a very close variant
of the argument in favor of~\eqref{eq:estimate_R_dissip}.
We leave details to the interested reader.
\end{proof}

\section{Proof of Proposition~\ref{prop:perturbative_graph_regime}: Reduction to perturbative graph setting}

\subsection{Strategy for the proof of Proposition~\ref{prop:perturbative_graph_regime}}
The proof of Proposition~\ref{prop:perturbative_graph_regime} is divided into
several steps which are collected and explained here. The corresponding proofs are 
presented in the next subsection.

First, let us fix the setting for the whole section.
Fix $T' \in (0,T_*)$, $C \in (1,\infty)$, and let 
$\Lambda = \Lambda(\mathscr{A},A(0),T') \in (0,\infty)$ the constant 
from Lemma~\ref{lem:stability_estimate_bad_times}. 
We aim to find $M=M(\Omega,A(0),\mathscr{A},C,\Lambda,T') \in (1,\infty)$ such that
for a.e.\ $t \in \Tgood(\Lambda,M) \cap (0,T')$, i.e.,
\begin{align}
\label{def:goodTimesRepeated1}
\int_{\Omega} \frac{1}{2}|\nabla w(\cdot,t)|^2 \,dx &\leq \Lambda,
\\
\label{def:goodTimesRepeated2}
E[A,\mu|\A](t) &\leq \frac{1}{M} \ell(t)^{d-1},
\end{align}
the conclusions of Proposition~\ref{prop:perturbative_graph_regime} hold true,
see~\eqref{eq:regularity_graph}--\eqref{eq:perturbative_regime1}. For what
follows, we restrict ourselves to the subset of full measure in $(0,T')$
such that all the a.e.\ properties of Definition~\ref{DefinitionWeakSolution}
(in particular, the ones of \cite[Definition~3]{Hensel2022}) are satisfied,
and fix $t \in (0,T')$ being an element of this subset. It will be convenient 
to trivially extend the oriented varifold~$\mu_t$ to an element of~$\mathrm{M}(\Rd[d]{\times}\mathbb{S}^{d-1})$
as well as the vector field~$\xi(\cdot,t)$ to a map $\Rd[d]\to\Rd[d]$
by defining 
\begin{align}
\label{eq:extensionVarifold}
|\mu_t|_{\mathbb{S}^{d-1}} &:= 0 \text{ and } \xi(\cdot,t) := 0
&&\text{in } \Rd[d]\setminus\overline{\Omega}. 
\end{align}
For readability, we suppress from now on the dependence of all quantities on~$t$ in the notation.

Heuristically, the idea for our proof of Proposition~\ref{prop:perturbative_graph_regime} 
is based on the competition between two effects:
\begin{itemize}[leftmargin=0.7cm]
\item
On one side, the assumptions~\eqref{def:goodTimesRepeated1}--\eqref{def:goodTimesRepeated2}
will allow us to prove that around every point $x_0 \in \supp\mu$, the rectifiable set $\supp\mu$
can be represented as a graph within a ball $B_{\rho}(x_0)$
(parametrized over a suitable affine subspace), where crucially the scale $\rho$
can be chosen independently of a given point $x_0 \in \supp\mu$
(more precisely, $\rho \ll \ell$ uniformly on $[0,T']$).
Naturally, the important building block here is Allard's regularity theory,
and the assumptions~\eqref{def:goodTimesRepeated1}--\eqref{def:goodTimesRepeated2}
(in particular, the tuning of the constant~$M$)
are used to show that one is in the corresponding setting.
\item
On the other side, based on the 
coercivity
of the overall error~$E[A,\mu|\A]$,
any unwanted feature of~$\supp\mu$ contradicting the 
assertions of Proposition~\ref{prop:perturbative_graph_regime}
(e.g., non-graphical components of~$\p^*A$ as screened from~$\p\A$ in normal direction,
or in our case also intersections of~$\p^*A$ with~$\p\Omega$)
can be shown to be of 
arbitrarily small mass, contradicting
a lower bound on the mass which one may derive from the local graph property
at scale~$\rho$.
\end{itemize}

For the rigorous argument, we ensure in a first step
that our assumption~\eqref{def:goodTimesRepeated1} implies
a curvature bound. 

\begin{lemma}[Curvature estimate up to the boundary]
\label{lem:curvatureControl}
There exists
$\vec{H}\colon\Rd[d]\to\Rd[d]
\in L^1_{loc}(\Rd[d];d|\mu|_{\mathbb{S}^{d-1}})$,
$\supp\vec{H}\subset\overline{\Omega}$,
being the generalized mean curvature
vector of~$\mu$ with respect to tangential variations, i.e., it holds
\begin{align}
\label{eq:genMeanCurvatureTangential}
\int_{\Rd[d]{\times}\mathbb{S}^{d-1}} (\mathrm{Id} {-} \pvec\otimes \pvec):\nabla B \,d\mu(\cdot,\pvec)
= -\int_{\Rd[d]} \vec{H} \cdot B \,d|\mu|_{\mathbb{S}^{d-1}}
\end{align}
for all $B \in C^1(\overline{\Omega};\Rd[d])$ with $n_{\p\Omega}\cdot B \equiv 0$
along $\p\Omega$.
Furthermore, for each 
\begin{align*}
p \in
\begin{cases}
[1,\infty) &\text{if } d = 2, \\
[1,4] &\text{if } d = 3,
\end{cases}
\end{align*}
and $\Lambda \in (0,\infty)$, there exists 
\begin{align}
\label{def:curvatureBound}
C_{\Lambda}:=C(\Omega,p,A(0),T,\Lambda) \in (0,\infty)
\end{align}
such that the assumption~\eqref{def:goodTimesRepeated1} implies
\begin{align}
\label{eq:integrability_curvarture}
\|\vec{H}\|_{L^p(\Rd[d];|\mu|_{\mathbb{S}^{d-1}})} 
\leq \|w\|_{L^p(\overline{\Omega};|\mu|_{\mathbb{S}^{d-1}})} 
\leq C_{\Lambda}. 
\end{align}
\end{lemma}

Based on this, our argument then relies on 
Allard's regularity theory for integer rectifiable varifolds with free boundary
(cf.\ \cite{Simon1983} and~\cite{Grueter1986}).
This theory is written down in the language of general varifolds
and not oriented ones. However, for the oriented 
varifold~$\mu \in \mathrm{M}(\Rd[d]{\times}\mathbb{S}^{d-1})$
one may always 
canonically associate a general
varifold $\widehat{\mu}\in\mathrm{M}(\Rd[d]{\times}\mathbf{G}(d,d{-}1))$,
where~$\mathbf{G}(d,d{-}1)$ denotes the space of $(d{-}1)$-dimensional linear subspaces of~$\Rd[d]$,
by means of
\begin{align}
\label{eq:genVarifold}
\int_{\Rd[d]{\times}\mathbb{G}(d,d{-}1)} \varphi(x,\widehat{T}) \,d\widehat{\mu}(x,\widehat{T})
&:= \int_{\Rd[d]{\times}\mathbb{S}^{d-1}} \varphi(x,\widehat{T}_\pvec) \,d\mu(x,\pvec),
\\ \notag
\varphi &\in C_{cpt}(\Rd[d]{\times}\mathbf{G}(d,d{-}1)),
\end{align}
where for $\pvec \in \mathbb{S}^{d-1}$ we denote by $\widehat{T}_\pvec \in \mathbf{G}(d,d{-}1)$ the
$1$-dimensional linear subspace of~$\Rd[d]$ having~$s$ as its unit normal vector.
Note that the mass measure~$|\widehat\mu|_{\mathbf{G}(d,d{-}1)}\in \mathrm{M}(\Rd[d])$ of~$\widehat{\mu}$ 
is simply~$|\mu|_{\mathbb{S}^{d-1}}$ and that the map~$\vec{H}$ from the previous lemma
is the generalized mean curvature vector of~$\widehat{\mu}$ in the usual sense
(with respect to tangential variations). Hence, for what follows we will be cavalier
about the distinction between~$\mu$ and~$\widehat{\mu}$.

\begin{theorem}[Allard's regularity theory---interior and boundary case]
\label{theo:allard_reg}
Fix data $\rho \in (0,\ell)$, $x_0 \in \supp\mu$,
and $\widehat{T} \in \mathbf{G}(d,d{-}1)$.

There exist constants $\varepsilon_{reg},\gamma_{reg} \in (0,1)$ and
$C_{reg} \in (1,\infty)$ such that:
\begin{itemize}[leftmargin=0.7cm]
\item[i)] If $x_0 \in \supp\mu \cap \Omega$
such that~\eqref{eq:genMeanCurvatureTangential} holds for all 
$B \in C^1_{cpt}(B_{2\rho}(x_0);\Rd[d])$ and
\begin{align}
\label{eq:assumptionsIntRegularity}
\begin{cases}
\frac{|\mu|_{\mathbb{S}^{d-1}}(B_\rho(x_0))}{\omega_{d-1}\rho^{d-1}} 
\leq 1 {+} \varepsilon_{reg}, & \\[1ex]
E_*^{\circ}[x_0,\rho,\widehat{T}] := 
\max\Big\{E_{tilt}[x_0,\rho,\widehat{T}],
\frac{\rho^{2(1{-}\frac{d{-}1}{4})}}{\varepsilon_{reg}}
\big(\int_{B_\rho(x_0)}|\vec{H}|^4\,d|\mu|_{\mathbb{S}^{d-1}}\big)^\frac{1}{2}
\Big\} 
\leq \varepsilon_{reg}, &
\end{cases}
\end{align}
where $E_{tilt}[x_0,\rho,\widehat{T}]$ is the usual tilt excess
relative to the integer rectifiable varifold~$\widehat{\mu}$,
then there exists a $C^{1,1-\frac{d{-}1}{4}}$ function
\begin{align}
\label{eq:graphVarifold}
u\colon (x_0 {+} \widehat{T}) \cap B_{\gamma_{reg}\rho}(x_0) \to \mathbb{R}
\end{align}
with the following properties: $u(x_0)=0$,
\begin{equation}
\begin{aligned}
\label{eq:graphRepVarifold}
&\supp\mu \cap B_{\gamma_{reg}\rho}(x_0)
\\&~~~
= \big\{y{+}u(y)\vec{n}_{\widehat{T}}(y) \colon y \in (x_0 {+} \widehat{T}) \cap B_{\gamma_{reg}\rho}(x_0)\big\}
\cap B_{\gamma_{reg}\rho}(x_0),
\end{aligned}
\end{equation}
where $\vec{n}_{\widehat{T}}$ is a normal vector of the affine subspace $x_0{+}\widehat{T}$, and
\begin{equation}
\begin{aligned}
\label{eq:estimatesGraphVarifold}
&\big\|(\rho^{-1} u,\nabla_{x_0{+}\widehat{T}} u)\big\|_{L^\infty\big((x_0 {+} \widehat{T}) \cap B_{\gamma_{reg}\rho}(x_0)\big)}
\\&~~~
+
\rho^{1-\frac{d{-}1}{4} 
\big[\nabla_{x_0{+}\widehat{T}} u\big]_{C^{0,1-\frac{d{-}1}{4}}\big((x_0 {+} \widehat{T}) \cap B_{\gamma_{reg}\rho}(x_0)\big)}
}\\&~~~~~~
\leq C_{reg}\bigg\{E_{tilt}^{\frac{1}{2}}[x_0,\rho,\widehat{T}]
+ 
\rho^{1-\frac{d{-}1}{4}}
\Big(\int_{B_\rho(x_0)}|\vec{H}|^4\,d|\mu|_{\mathbb{S}^{d-1}}\Big)^\frac{1}{4}
\bigg\}.
\end{aligned}
\end{equation}
\item[ii)] If $x_0 \in \supp\mu \cap \partial\Omega$ such that
\begin{align}
\label{eq:assumptionsBoundaryRegularity}
\begin{cases}
\frac{|\mu|_{\mathbb{S}^{d-1}}(B_\rho(x_0)) + |\mu|_{\mathbb{S}^{d-1}}(\widetilde{B}_\rho(x_0))}
{\omega_{d-1}\rho^{d-1}} 
\leq 1 {+} \varepsilon_{reg}, & \\[1ex]
E_*[x_0,\rho,\mathrm{Tan}_{x_0}^\perp\p\Omega] :=
\max\Big\{E_*^{\circ}[x_0,\rho,\mathrm{Tan}_{x_0}^\perp\p\Omega],
\frac{\rho}{\ell}\Big\} \leq \varepsilon_{reg}, &
\end{cases}
\end{align}
where $\widetilde{B}_\rho(x_0) :=
\{x \in \Rd[d]\colon |\widetilde{x} - x_0| < \rho \text{ where } \widetilde{x} 
:= x - 2s_{\p\Omega}(x)\n_{\p\Omega}(P_{\p\Omega}(x))\}$, 
then there exists a $C^{1,\frac{1}{2}}$ function
\begin{align}
\label{eq:graphVarifold2}
u\colon (x_0 {+} \widehat{T}) \cap B_{\gamma_{reg}\rho}(x_0) \to \mathbb{R}
\end{align}
with the following properties: $u(x_0)=0$, 
\begin{equation}
\begin{aligned}
\label{eq:graphRepVarifold2}
&\supp\mu \cap B_{\gamma_{reg}\rho}(x_0)
\\&~~~
= \big\{y{+}u(y)\vec{n}_{\widehat{T}}(y) \colon y \in (x_0 {+} \widehat{T}) \cap B_{\gamma_{reg}\rho}(x_0)\big\}
\cap B_{\gamma_{reg}\rho}(x_0) \cap \overline{\Omega},
\end{aligned}
\end{equation}
and
\begin{align}
\label{eq:contactAngle}
\n_{\p\Omega}(x_0) \in \mathrm{Tan}_{x_0}(\supp\mu)
\end{align}
as well as
\begin{equation}
\begin{aligned}
\label{eq:estimatesGraphVarifold2}
&\big\|(\rho^{-1} u,\nabla_{x_0{+}\widehat{T}} u)\big\|_{L^\infty\big((x_0 {+} \widehat{T}) \cap B_{\gamma_{reg}\rho}(x_0) \cap \overline{\Omega}\big)}
\\&~~~
\leq C_{reg}\varepsilon_{reg}^\frac{1}{2(d{-}1) + 3},
\\[1ex]
&\rho^{\frac{1}{2}} \big[\nabla_{x_0{+}\widehat{T}} u\big]_{C^{0,\frac{1}{2}}\big((x_0 {+} \widehat{T}) 
\cap B_{\gamma_{reg}\rho}(x_0) \cap \overline{\Omega}\big)}
\\&~~~
\leq C_{reg}\bigg\{E_{tilt}^{\frac{1}{2}}[x_0,\rho,\widehat{T}]
+ \rho^{1-\frac{d{-}1}{4}}
\Big(\int_{B_\rho(x_0)}|\vec{H}|^4\,d|\mu|_{\mathbb{S}^{d-1}}\Big)^\frac{1}{4}
+ \Big(\frac{\rho}{\ell}\Big)^\frac{1}{2}
\bigg\}.
\end{aligned}
\end{equation}
\end{itemize}
\end{theorem}

In a next step, we aim to define a scale
$\rho_{reg} \in (0,\ell)$, uniformly on $[0,T']$, such that one may
afterward choose $M \gg 1$, again uniformly on $[0,T']$, so that the
assumptions~\eqref{def:goodTimesRepeated1}--\eqref{def:goodTimesRepeated2}
imply, among other things, that for all $x_0 \in \supp |\mu|_{\mathbb{S}^{d-1}}$,
the rectifiable set $\supp |\mu|_{\mathbb{S}^{d-1}}$ can be locally represented around~$x_0$
as a graph on scale~$\gamma_{reg}\rho_{reg}$ (with respect to a suitably chosen subspace)
in the precise sense of Theorem~\ref{theo:allard_reg}.

By locally uniform regularity of~$\mathscr{A}$, one may choose 
\begin{align}
\label{def:epsilon}
\varepsilon = \varepsilon(\A)
\in \bigg(0,\min\Big\{\varepsilon_{reg},\frac{1}{8}\frac{1}{C_{reg}}\frac{1}{16 C}\Big\}\bigg)
\quad\text{and}\quad 
\widetilde{C} = \widetilde{C}(\A) \in (1,\infty)
\end{align}
uniformly on $[0,T']$ 
such that for all
\begin{align}
\label{def:rho}
\rho \in (0,\rho_{reg}],
\quad\text{where }
\rho_{reg} := \widetilde{C}^{-1}\widetilde{\rho}
\text{ and } \widetilde{\rho} := \frac{\varepsilon^2}{C_{\Lambda}^2}\ell,
\end{align}
at least the following properties are satisfied
($P_{\mathrm{Tan}_{y}\p\A}$ denotes the orthogonal 
projection onto the tangent space $\mathrm{Tan}_{y}\p\A$
at $y \in \p\A$):
\begin{itemize}[leftmargin=0.5cm]
\item it holds
					 \begin{align}
					 \label{eq:flatness3}
					 x_0 \in \{|\xi| \leq 1/2\}
					 \quad\Longrightarrow\quad
					 B_\rho(x_0) \subset \{|\xi| \leq 3/4\},
					 \end{align}
\item for all $x_0 \in \{|\xi| > \frac{1}{2}\}$ and all $x \in B_\rho(x_0)$ it holds
					\begin{align}
					\label{eq:flatness4}
					\big|\xi(x) - \n_{\p\A}\big(P_{\p\A}(x_0)\big)\big| \leq 
					\frac{1}{1+\varepsilon/2}\frac{\varepsilon^2}{4},
					\end{align}
\item for all $x_0 \in \{|\xi| > \frac{1}{2}\}$ and all 
			$y \in B_{\widetilde{C}^{-1}\rho}(y_0) \cap \p\A$,
			where $y_0 := P_{\p\A}(x_0)$, it holds
					\begin{equation}
					\begin{aligned}
					\label{eq:flatness5}
					&P_{x_0{+}\mathrm{Tan}_{y_0}\p\A}\big(B_{\rho}(x_0)
					\cap\{y{+}s\n_{\p\A}(y)\colon s \in (-\ell,\ell)\}\big)
					\\&~~~~~~~~~~~~~~~~~~~~~~~~~~~~~
					\subset (x_0 {+} \mathrm{Tan}_{y_0}\p\A) \cap B_{(\cos\alpha_{reg})\rho}(x_0)
					\end{aligned}
					\end{equation}
					and
					\begin{align}
					\label{eq:flatness6}
					\text{the opening angle between } \n_{\p\A}(y) \text{ and } \n_{\p\A}(y_0)
					\text{ is at most } \frac{1}{2}\Big(\frac{\pi}{2} - \alpha_{reg}\Big),
					\end{align}
					where $\alpha_{reg} \in (0,\frac{\pi}{2})$ is defined by
					$3C_{reg}\varepsilon =: \tan\alpha_{reg}$,
\item for all $y_0,y\in\p\A$ it holds
					\begin{align}
					\label{eq:flatness1}
					\big|P_{\mathrm{Tan}_{y_0}\p\A}(y{-}y_0)\big| \leq \rho
					\quad\Longrightarrow\quad
					|(y{-}y_0) \cdot \n_{\p\A}(y_0)| \leq \frac{1}{4} \Big(\frac{\ell}{16 C}\Big).
					\end{align}
\end{itemize}
Condition~\eqref{eq:flatness3} will be applied to show that $\supp\mu$
cannot be located too far away from~$\p\A$. Condition~\eqref{eq:flatness4}
will be of importance whenever one has to switch between the usual geometric measure
theory tilt excess~$E_{tilt}$ and our relative entropy~$E_{rel}$, as~$E_{tilt}$ is working locally with
a fixed tangent space whereas~$E_{rel}$ encodes tilt excess at each point based on the
varying vector field~$\xi$. We will rely on conditions~\eqref{eq:flatness4} and~\eqref{eq:flatness5}
for the proof of the graph property~\eqref{eq:graph_representation}; more precisely,
if $x_0 \in \supp\mu$ is a point sufficiently close to~$\p\A$ where $\supp\mu$
admits the graph representation~\eqref{eq:graphRepVarifold}--\eqref{eq:estimatesGraphVarifold} 
with respect to~$\widehat{T} = \mathrm{Tan}_{P_{\p\A}(x_0)}\p\A$
locally on scale $\rho=\gamma_{reg}\rho_{reg}$,
then~\eqref{eq:flatness4} and~\eqref{eq:flatness5} ensure that $\supp\mu$
can also be written as a graph in the sense of~\eqref{eq:graph_representation}
locally on the smaller scale $\widetilde{C}^{-1}\gamma_{reg}\rho_{reg}$.
Finally, condition~\eqref{eq:flatness1} is just a local height estimate for
the interface~$\p\A$ if one locally screens it from~$y_0 {+} \mathrm{Tan}_{y_0}(\p\A)$,
and it will be needed in our proof for the estimate~\eqref{eq:perturbative_regime1}.

With these choices in place,
we first show that~$\rho_{reg}$ indeed satisfies the aforementioned goal
of representing $\supp |\mu|_{\mathbb{S}^{d-1}}$ locally on scale~$\gamma_{reg}\rho_{reg}$ 
around any of its points as a graph.

\begin{lemma}[Applicability of Allard regularity theory]
\label{lem:localGraphProperty}
There exists 
$M_0 \gg_{\A,C_\Lambda} 1$, uniformly on $[0,T']$, 
such that for all $M \geq M_0$ the 
assumptions~\emph{\eqref{def:goodTimesRepeated1}--\eqref{def:goodTimesRepeated2}}
imply that
the hypotheses~\eqref{eq:assumptionsIntRegularity} hold true 
at scale $\rho = \rho_{reg}$ for all interior points
$x_0 \in \supp\mu \cap \Omega$ and
\begin{align}
\label{eq:choiceSubspaceInt}
\begin{cases}
\text{any } \widehat{T} \in \mathbf{G}(d,d{-}1) & \text{if } x_0 \in \{|\xi| \leq 1/2\}, \\
\widehat{T} = \mathrm{Tan}_{P_{\p\A}(x_0)}\p\A &\text{if } x_0 \in \{|\xi| > 1/2\},
\end{cases}
\end{align}
as well as that the hypotheses~\eqref{eq:assumptionsBoundaryRegularity}
hold true at scale $\rho = \rho_{reg}$ for all boundary points $x_0 \in \supp\mu \cap \p\Omega$
(and $\widehat{T} \in \mathbf{G}(d,d{-}1)$ fixed by $\mathrm{Tan}_{x_0}^\perp\p\Omega$).
Furthermore, in both cases
\begin{align}
\label{eq:smallnessGraph1}
E_{tilt}^\frac{1}{2}[x_0,\rho_{reg},\widehat{T}] &\leq \varepsilon,
\\
\label{eq:smallnessGraph2}
\rho_{reg}^{1-\frac{d{-}1}{4}}
\Big(\int_{B_{\rho_{reg}}(x_0)}|\vec{H}|^4\,d|\mu|_{\mathbb{S}^{d-1}}\Big)^\frac{1}{4}
&\leq \varepsilon
\end{align}
for all $x_0 \in \supp\mu$ with associated data $(\rho_{reg},\widehat{T})$ as above.
\end{lemma}

By means of the previous result, we then ensure that the mass measure
of the varifold~$\mu$ reduces to the mass measure of the reduced boundary of the finite perimeter set~$A$
and that the latter is supported sufficiently close to the interface~$\p\A$
of the strong solution (in particular, of positive distance to the physical boundary~$\p\Omega$).

\begin{lemma}[Reduction to mass measure with no boundary intersection]
\label{lem:structureVarifold}
One may choose 
$M_0 \gg_{\A,C_\Lambda} 1$ uniformly on $[0,T']$
such that for all $M \geq M_0$ the 
assumptions~\emph{\eqref{def:goodTimesRepeated1}--\eqref{def:goodTimesRepeated2}} imply
\begin{align}
\label{eq:closeToStrongSolution}
\partial^*A &\subset \{|\xi| > 1/2\} \subset B_{\ell}(\p\A),
\\
\label{eq:repMassMeasureVarifold}
|\mu|_{\mathbb{S}^{d-1}} &= \mathcal{H}^{d-1} \llcorner \partial^*A
\end{align}
and
\begin{equation}
\begin{aligned}
\label{eq:genMeanCurvatureFullSpace}
\int_{\p^*A} (\mathrm{Id} {-} \n_{\p^*A}\otimes \n_{\p^*A}):\nabla B \,d\mathcal{H}^{d-1}
&= \int_{\Rd[d]{\times}\mathbb{S}^{d-1}} (\mathrm{Id} {-} s\otimes s):\nabla B \,d\mu
\\&= -\int_{\Rd[d]} \vec{H} \cdot B \,d|\mu|_{\mathbb{S}^{d-1}}
\end{aligned}
\end{equation}
for all $B \in C^1_{cpt}(\Rd[d];\Rd[d])$.
\end{lemma} 

In the next step, we leverage on the previous results to show that
the interface of the weak solution~$\p^*A$ is in fact ``sufficiently rich''
in order to allow for a graph
representation relative to the interface of the strong solution~$\p\A$.
Note that this cannot follow just from requiring smallness of the
relative entropy~\eqref{eq:rel_entropy}---which, by careful inspection
of the proofs, is in fact up to now the only required ingredient
concerning the smallness of the overall error---but also requires smallness
of the volume error~\eqref{eq:bulk_error}: indeed, the relative entropy 
provides no error control in the regime of vanishing interfacial
mass $|\mu|_{\mathbb{S}^{d-1}}(\Rd[d]) \downarrow 0$.

\begin{lemma}[Graph representation]
\label{lem:graphCandidate}
One may choose 
$M_0 \gg_{\A,C_\Lambda} 1$ uniformly on $[0,T']$
such that for all $M \geq M_0$ the 
assumptions~\emph{\eqref{def:goodTimesRepeated1}--\eqref{def:goodTimesRepeated2}} imply that $\p^*A$ 
can be represented as a graph over $\p\A$
in the sense of~\eqref{eq:graph_representation}
with regularity~\eqref{eq:regularity_graph}.
\end{lemma}

We finally conclude by showing that the already established graph
representation satisfies all the required properties.

\begin{lemma}[Height function estimates]
\label{lem:graphRep}
One may choose 
$M_0 \gg_{\A,C_\Lambda} 1$ uniformly on $[0,T']$
such that for all $M \geq M_0$ the 
assumptions~\emph{\eqref{def:goodTimesRepeated1}--\eqref{def:goodTimesRepeated2}} imply that
the height function $h\colon\p\A\to[-\ell,\ell]$ 
from Lemma~\ref{lem:graphCandidate} representing~$\p^*A$ as a graph over~$\p\A$
satisfies the estimates~\eqref{eq:perturbative_regime1}.
\end{lemma}

\subsection{Proofs} We continue with the proofs of the several intermediate results
from the previous subsection, eventually culminating into a proof of
Proposition~\ref{prop:perturbative_graph_regime}, which itself is the last missing
ingredient for the proof of our main result, Theorem~\ref{theo:mainResult}.

\begin{proof}[Proof of Lemma~\ref{lem:curvatureControl}]
Denote by $H_{|\mu|_{\mathbb{S}^{d-1}} \llcorner \Omega}\colon 
\supp(|\mu|_{\mathbb{S}^{d-1}} \llcorner \Omega) \to \mathbb{R}^d$
the map from \cite[Definition~3, item~iii)]{Hensel2022} and define
\begin{align}
\label{eq:definitionGenMeanCurvatureTangential}
\vec{H} := 
\begin{cases}
H_{|\mu|_{\mathbb{S}^{d-1}} \llcorner \Omega} & \text{in } \Omega,
\\
0 & \text{else}.
\end{cases}
\end{align}
The identity~\eqref{eq:genMeanCurvatureTangential} then directly
follows from~\eqref{eq:definitionGenMeanCurvatureTangential} and
\cite[(17h)]{Hensel2022}, whereas the estimate~\eqref{eq:integrability_curvarture}
is a consequence of \cite[Proposition~5, (26), Definition~3 item iv)]{Hensel2022}
and~\eqref{eq:L2ControlChemPotential}.
\end{proof}

\begin{proof}[Proof of Theorem~\ref{theo:allard_reg}]
We naturally distinguish between two cases.

\textit{Case 1:} 
For the case of interior points $x_0 \in \supp\mu\cap\Omega$,
let $\bar\varepsilon,\bar\gamma\in (0,1)$ and $\bar c\in(1,\infty)$ be the constants
from~\cite[23.1 Theorem]{Simon1983}. It follows from~\cite[23.2 Remark, item~(2)]{Simon1983}
and our hypotheses that for suitably small~$\varepsilon_{reg} \in (0,\bar\varepsilon)$ 
\begin{align}
\label{eq:massBoundSmallScales}
\frac{|\mu|_{\mathbb{S}^{d-1}}(B_\sigma(x))}
{\omega_{d-1}\sigma^{d-1}} \leq \frac{3}{2}
\end{align}
for all $\sigma \in (0,\frac{\rho}{2})$ and all $x \in \supp\mu \cap B_{\varepsilon_{reg}\rho}(x_0)$.
In particular, the $(d{-}1)$-dimensional density of~$\mu$, denoted by $\Theta_{d-1}(\mu,\cdot)$, satisfies  
\begin{align}
\label{eq:unitDensity}
\Theta_{d-1}(\mu,x) &= 1
&&\text{for all } x \in \supp\mu \cap B_{\varepsilon_{reg}\rho}(x_0).
\end{align}
Furthermore, for 
small enough~$\varepsilon_{reg} \in (0,\bar\varepsilon)$, we deduce 
from~\cite[Proof of 23.1, inequality~(12)]{Simon1983} that
\begin{align}
\label{eq:smallnessExcess}
E_*^{\circ}[x_0,(\varepsilon_{reg}\rho)/2,\widehat{T}] \leq \bar\varepsilon.
\end{align}
Hence, the hypotheses of~\cite[23.1 Theorem]{Simon1983} are fulfilled
for the choices $U = B_{\varepsilon_{reg}\rho}(x_0)$ and 
$\rho = (\varepsilon_{reg}\rho)/2$, so that the remaining claims
follow from the conclusions of~\cite[23.1 Theorem]{Simon1983} 
for $\gamma_{reg}:=(\varepsilon_{reg}\bar\gamma)/2$ and $C_{reg}:=\bar c$.

\textit{Case 2:} For the case of boundary points $x_0 \in \supp\mu\cap\p\Omega$,
one may argue analogously, using~\cite[inequality between~(40) and~(41)]{Grueter1986},
\cite[inequality~(51)]{Grueter1986}
and \cite[4.9 Theorem]{Grueter1986}
as substitutes for~\cite[23.2 Remark, item~(2)]{Simon1983}, 
\cite[Proof of 23.1, inequality~(12)]{Simon1983}
and \cite[23.1 Theorem]{Simon1983}, respectively. 
Note that the analogue of~\eqref{eq:massBoundSmallScales2} is given by
\begin{align}
\label{eq:massBoundSmallScales2}
\frac{|\mu|_{\mathbb{S}^{d-1}}(B_\sigma(x))+
|\mu|_{\mathbb{S}^{d-1}}(\widetilde{B}_\sigma(x))}
{\omega_{d-1}\sigma^{d-1}} \leq \frac{3}{2}
\end{align}
for all $\sigma \in (0,\frac{\rho}{2})$ and all $x \in \supp\mu \cap B_{\varepsilon_{reg}\rho}(x_0)$.
In particular,
\begin{equation}
\begin{aligned}
\label{eq:unitDensity2}
\Theta_{d-1}(\mu,x) &= 1
&&\text{for all } x \in \supp\mu \cap B_{\varepsilon_{reg}\rho}(x_0) \cap \Omega,
\\
\Theta_{d-1}(\mu,x) &\leq \frac{3}{4}
&&\text{for all } x \in \supp\mu \cap B_{\varepsilon_{reg}\rho}(x_0) \cap \p\Omega;
\end{aligned}
\end{equation}
a fact which we state for future reference (cf.\ \cite[3.2 Corollary]{Grueter1986}).
Note finally that the second statement of~\eqref{eq:unitDensity2} implies
$|\mu|_{\mathbb{S}^{d-1}}(B_{\varepsilon_{reg}\rho}(x_0) \cap \p\Omega) = 0$. 
\end{proof}

\begin{proof}[Proof of Lemma~\ref{lem:localGraphProperty}]
Using the bound~\eqref{eq:integrability_curvarture}
and the definitions~\eqref{def:epsilon}--\eqref{def:rho},
we get
\begin{align}
\label{eq:hypothesisCurvature}
\frac{\rho_{reg}^{2(1{-}\frac{d{-}1}{4})}}{\varepsilon_{reg}}
\Big(\int_{B_{\rho_{reg}}(x_0)}|\vec{H}|^4\,d|\mu|_{\mathbb{S}^{d-1}}\Big)^\frac{1}{2}
\leq \frac{\rho_{reg}}{\varepsilon_{reg}} C_\Lambda^2
\leq \frac{\varepsilon^2}{\varepsilon_{reg}} \leq \varepsilon_{reg}.
\end{align}
The penultimate inequality settles~\eqref{eq:smallnessGraph2}
whereas the last is precisely the curvature estimate
required by~\eqref{eq:assumptionsIntRegularity}
and~\eqref{eq:assumptionsBoundaryRegularity}. Note also
that by definitions~\eqref{def:epsilon}--\eqref{def:rho}
it holds $\rho_{reg}/\ell \leq \varepsilon_{reg}$.
It thus remains to derive the asserted estimates for the
mass ratios and the tilt excess.
To this end, we distinguish between the three natural cases.

\textit{Case 1: $x_0 \in \supp\mu \cap \Omega \cap \{|\xi| \leq 1/2\}$.}
We start estimating by a union bound and 
recalling~\eqref{def:multiplicity}--\eqref{eq:valuesMultiplicity}
\begin{equation}
\label{eq:unionBoundMass}
\begin{aligned}
|\mu|_{\mathbb{S}^{d-1}}\big(B_{\rho_{reg}}(x_0)\big)
&\leq |\mu|_{\mathbb{S}^{d-1}}\big(B_{\rho_{reg}}(x_0) \cap \p\Omega\big)
\\&~~~
+ |\mu|_{\mathbb{S}^{d-1}}\big(B_{\rho_{reg}}(x_0) \cap \Omega \cap \{\varrho \leq 1/3\}\big) 
\\&~~~
+ |\mu|_{\mathbb{S}^{d-1}}\big(B_{\rho_{reg}}(x_0) \cap \Omega \cap \{\varrho = 1\}\big).
\end{aligned}
\end{equation}
Recalling also~\eqref{eq:repRelEntropy2} and
from property~\eqref{eq:flatness3} that $B_{\rho_{reg}}(x_0) \subset \{|\xi| \leq \frac{3}{4}\}$, 
we get
\begin{equation}
\label{eq:localGraphProperty}
\begin{aligned}
&|\mu|_{\mathbb{S}^{d-1}}\big(B_{\rho_{reg}}(x_0) \cap \p\Omega\big)
+ |\mu|_{\mathbb{S}^{d-1}}\big(B_{\rho_{reg}}(x_0) \cap \Omega \cap \{\varrho \leq 1/3\}\big)
\\&~~~
\leq |\mu|_{\mathbb{S}^{d-1}}(\p\Omega) + 
\frac{3}{2} \int_{\Omega \cap \{\varrho \leq \frac{1}{3}\}} 1 - \varrho \,d|\mu|_{\mathbb{S}^{d-1}}
\leq \frac{3}{2} E_{rel}[A,\mu|\A]
\end{aligned}
\end{equation}
and
\begin{equation}
\label{eq:localGraphProperty2}
\begin{aligned}
&|\mu|_{\mathbb{S}^{d-1}}\big(B_{\rho_{reg}}(x_0) \cap \Omega \cap \{\varrho = 1\}\big)
\\&~~~
\leq \int_{\p^*A\cap\Omega\cap\{|\xi|\leq 3/4\}} 1 \,d\mathcal{H}^{d-1}
\\&~~~
\leq 4 \int_{\p^*A\cap\Omega} (1 - \n_{\p^*A}\cdot\xi)\,d\mathcal{H}^{d-1}
\leq 4 E_{rel}[A,\mu|\A].
\end{aligned}
\end{equation}
Hence, the previous three displays together with assumption~\eqref{def:goodTimesRepeated2} imply
\begin{align}
\label{eq:localGraphProperty3}
\frac{|\mu|_{\mathbb{S}^{d-1}}(B_{\rho_{reg}}(x_0))}{\omega_{d-1}\rho_{reg}^{d-1}}
\leq \frac{11}{2} \frac{\widetilde{C}^{d-1}C_\Lambda^{d-1}}
{\omega_{d-1}\varepsilon^{2(d-1)}} \frac{1}{M}
\stackrel{M \geq M_0 \gg_{\A,C_\Lambda} 1}{\leq} \varepsilon^2.
\end{align}
Let now $\widehat{T} \in \mathbf{G}(d,d{-}1)$ arbitrary but fixed.
We may then analogously ensure by choosing $M \geq M_0 \gg_{\A,C_\Lambda} 1$ that
\begin{align}
\label{eq:localGraphProperty4}
E_{tilt}[x_0,\rho_{reg},\widehat{T}] \leq \varepsilon^2
\end{align} 
because of the simple observation that
$E_{tilt}[x_0,\rho_{reg},\widehat{T}] \leq C_{tilt}
\frac{|\mu|_{\mathbb{S}^{d-1}}(B_{\rho_{reg}}(x_0))}{\omega_{d-1}\rho_{reg}^{d-1}}$
for some absolute constant $C_{tilt}\in(1,\infty)$.

\textit{Case 2: $x_0 \in \supp\mu \cap \p\Omega$.}
Due to \cite[Lemma~4.2]{Kagaya2017}, $\rho_{reg} \leq \ell/8$,
\eqref{eq:no_contact_points} and~\eqref{eq:min_assumption6},
we know that
\begin{equation}
\begin{aligned}
\label{eq:localGraphProperty5}
\widetilde{B}_{\rho_{reg}}(x_0) \subset B_{5\rho_{reg}}(x_0)
\subset B_{\ell}(\p\Omega) \subset \{|\xi| = 0\}.
\end{aligned}
\end{equation}
Hence, arguing essentially analogous to the previous case shows
\begin{align}
\label{eq:localGraphProperty6}
|\mu|_{\mathbb{S}^{d-1}}(B_{\rho_{reg}}(x_0)) 
+ |\mu|_{\mathbb{S}^{d-1}}(\widetilde{B}_{\rho_{reg}}(x_0))
\leq 3 E_{rel}[A,\mu|\A],
\end{align}
so that a suitable choice of $M \geq M_0 \gg_{\A,C_\Lambda} 1$
allows to guarantee
\begin{align}
\label{eq:localGraphProperty7}
\frac{|\mu|_{\mathbb{S}^{d-1}}(B_{\rho_{reg}}(x_0)) 
+ |\mu|_{\mathbb{S}^{d-1}}(\widetilde{B}_{\rho_{reg}}(x_0))}
{\omega_{d-1}\rho_{reg}^{d-1}}
\leq \varepsilon^2
\end{align}
as well as
\begin{align}
\label{eq:localGraphProperty8}
E_{tilt}[x_0,\rho_{reg},\mathrm{Tan}_{x_0}^\perp\p\Omega] \leq \varepsilon^2.
\end{align}

\textit{Case 3: $x_0 \in \supp\mu \cap \Omega \cap \{|\xi| > 1/2\}$.}
By~\eqref{eq:unionBoundMass}, \eqref{eq:localGraphProperty} 
and~\eqref{def:multiplicity} it holds
\begin{align}
\label{eq:localGraphProperty9}
|\mu|_{\mathbb{S}^{d-1}}\big(B_{\rho_{reg}}(x_0)\big)
\leq \frac{3}{2} E_{rel}[A,\mu|\A] + \int_{\p^*A \cap B_{\rho_{reg}}(x_0)} 1 \,d\mathcal{H}^{d-1}.
\end{align}
Note that indeed $\p^*A \cap B_{\rho_{reg}}(x_0) = \p^*A \cap B_{\rho_{reg}}(x_0) \cap \Omega$
since $x_0 \in \Omega \cap \{|\xi|>1/2\}$ implies $B_{\rho_{reg}}(x_0) \subset B_{\ell}(\p\A) \subset \Omega$
due to~\eqref{eq:min_assumption6}, $\rho_{reg} \leq \ell/8$ and~\eqref{eq:no_contact_points}.

To estimate the second term on the right hand side of~\eqref{eq:localGraphProperty9},
first introduce some additional notation. Define $T_{x_0} := P_{\p\A}(x_0) + \mathrm{Tan}_{P_{\p\A}(x_0)}\p\A$
and let $P_{x_0}$ the nearest point projection onto~$T_{x_0}$. For every $x \in \Rd[d]$,
we further denote by $A_{P_{x_0}(x)}$ the one-dimensional slice
$A \cap \{P_{x_0}(x) + s\n_{\p\A}(P_{\p\A}(x_0)) \colon |s| \leq \ell\}$.
Finally, define the sets
\begin{align}
\label{eq:localGraphProperty10}
S^{(1)}_{x_0} &:= \p^*A \cap B_{\rho_{reg}}(x_0) 
\cap \{x \in \Rd[d] \colon \mathcal{H}^0(\p^* A_{P_{x_0}(x)}) > 1\},
\\
\label{eq:localGraphProperty11}
S^{(2)}_{x_0} &:= \bigg\{\big(\p^*A \cap B_{\rho_{reg}}(x_0)\big) \setminus S^{(1)}_{x_0} 
\colon \n_{\p^*A}(x)\cdot\xi(x) \geq \frac{1 {+} \varepsilon/4}{1 {+} \varepsilon/2}\bigg\},
\\
\label{eq:localGraphProperty12}
S^{(3)}_{x_0} &:= \bigg\{\big(\p^*A \cap B_{\rho_{reg}}(x_0)\big) \setminus S^{(1)}_{x_0} 
\colon \n_{\p^*A}(x)\cdot\xi(x) < \frac{1 {+} \varepsilon/4}{1 {+} \varepsilon/2}\bigg\}.
\end{align}
In particular,
\begin{align}
\label{eq:localGraphProperty13}
\p^*A \cap B_{\rho_{reg}}(x_0) = S^{(1)}_{x_0} \cup S^{(2)}_{x_0} \cup S^{(3)}_{x_0}
\end{align}
and, by definition of~$S^{(1)}_{x_0}$ and~$S^{(2)}_{x_0}$
as well as by the flatness property~\eqref{eq:flatness4},
\begin{align}
\label{eq:localGraphProperty14}
x \in S^{(2)}_{x_0}
\quad \Longrightarrow \quad
\mathcal{H}^0(\p^* A_{P_{x_0}(x)}) = 1
\text{ and } \big|\n_{\p^*A}(x)\cdot\n_{\p\A}(P_{\p\A}(x_0))\big|
\geq \frac{1}{1 {+} \varepsilon/2}.
\end{align}

Now, we estimate term by term in the decomposition~\eqref{eq:localGraphProperty13}.
First, a slicing argument ensures
\begin{align}
\label{eq:localGraphProperty15}
\mathcal{H}^{d-1}(S^{(1)}_{x_0}) \leq C_{rel} E_{rel}[A,\mu|\A]
\end{align}
for some universal constant $C_{rel} \in (1,\infty)$.
Second, one immediately deduces that
\begin{align}
\label{eq:localGraphProperty16}
\mathcal{H}^{d-1}(S^{(3)}_{x_0}) \leq \frac{4(1 {+} \varepsilon/2)}{\varepsilon} 
\int_{\p^*A \cap \Omega} (1 - \n_{\p^*A}\cdot\xi) \,d\mathcal{H}^{d-1}
\leq \frac{4(1 {+} \varepsilon/2)}{\varepsilon} E_{rel}[A,\mu|\A].
\end{align}
Third, by coarea formula and~\eqref{eq:localGraphProperty14}
\begin{align}
\label{eq:localGraphProperty17}
\mathcal{H}^{d-1}(S^{(2)}_{x_0}) \leq (1 {+} \varepsilon/2) \mathcal{H}^{d-1}\big(B_{\rho_{reg}}(P_{\p\A}(x_0))\big)
= (1 {+} \varepsilon/2)\omega_{d-1}\rho_{reg}^{d-1}.
\end{align}

In summary, we may now infer from~\eqref{eq:localGraphProperty9},
\eqref{eq:localGraphProperty13} and~\eqref{eq:localGraphProperty15}--\eqref{eq:localGraphProperty17}
that
\begin{align}
\label{eq:localGraphProperty18}
|\mu|_{\mathbb{S}^{d-1}}\big(B_{\rho_{reg}}(x_0)\big)
\leq (1 {+} \varepsilon/2)\omega_{d-1}\rho_{reg}^{d-1}
+ \Big(\frac{3}{2} + C_{rel} + \frac{4(1 {+} \varepsilon/2)}{\varepsilon}\Big) E_{rel}[A,\mu|\A]
\end{align}
for some universal constant $C_{rel} \in (1,\infty)$. Hence, choosing 
$M \geq M_0 \gg_{\A,C_\Lambda} 1$
in a suitable manner entails by assumption~\eqref{def:goodTimesRepeated2}
\begin{align}
\label{eq:localGraphProperty19}
|\mu|_{\mathbb{S}^{d-1}}\big(B_{\rho_{reg}}(x_0)\big) \leq (1 {+} \varepsilon)
\omega_{d-1}\rho_{reg}^{d-1}. 
\end{align}

Having established the desired estimate on the mass ratio, we now turn to the asserted
bound for the tilt excess and claim that, upon suitably choosing~$M_0$,
\begin{align}
\label{eq:localGraphProperty20}
E_{tilt}[x_0,\rho_{reg},T_{x_0}] \leq \varepsilon^2. 
\end{align}
The argument in favor of~\eqref{eq:localGraphProperty20} is very close
to the one producing~\eqref{eq:localGraphProperty19}. 

First, employing the decomposition underlying~\eqref{eq:unionBoundMass},
we obtain as the analogue of~\eqref{eq:localGraphProperty9}
\begin{equation}
\begin{aligned}
\label{eq:localGraphProperty21}
&E_{tilt}[x_0,\rho_{reg},T_{x_0}] 
\\
&\leq
C_{tilt} \frac{E_{rel}[A,\mu|\A]}{\rho_{reg}^{d-1}}
+ \frac{1}{\rho_{reg}^{d-1}} \int_{\p^*A \cap B_{\rho_{reg}}(x_0)} 
\big(1 - \n_{\p^*A}(x)\cdot\n_{\p\A}\big(P_{\p\A}(x_0)\big)\big) \,d\mathcal{H}^{d-1}(x),
\end{aligned}
\end{equation}
where $C_{tilt} \in (1,\infty)$ is a universal constant.
For an estimate of the second term on the right hand side of the previous display,
we again make use of the decomposition~\eqref{eq:localGraphProperty10}--\eqref{eq:localGraphProperty13}.
Hence, due to the estimates~\eqref{eq:localGraphProperty15} and~\eqref{eq:localGraphProperty16},
\begin{equation}
\begin{aligned}
\label{eq:localGraphProperty22}
&\frac{1}{\rho_{reg}^{d-1}} \int_{\p^*A \cap B_{\rho_{reg}}(x_0)} 
\big(1 - \n_{\p^*A}(x)\cdot\n_{\p\A}\big(P_{\p\A}(x_0)\big)\big) \,d\mathcal{H}^{d-1}(x)
\\&~~~
\leq \frac{2}{\rho_{reg}^{d-1}}
\big(\mathcal{H}^{d-1}(S^{(1)}_{x_0}) + \mathcal{H}^{d-1}(S^{(3)}_{x_0})\big)
\\&~~~~~~
+ \frac{1}{\rho_{reg}^{d-1}} \int_{S^{(2)}_{x_0}} 
\big(1 - \n_{\p^*A}(x)\cdot\n_{\p\A}\big(P_{\p\A}(x_0)\big)\big) \,d\mathcal{H}^{d-1}(x)
\\&~~~
\leq 2\Big(C'_{tilt} + \frac{4(1 {+} \varepsilon/2)}{\varepsilon}\Big)
\frac{E_{rel}[A,\mu|\A]}{\rho_{reg}^{d-1}}
\\&~~~~~~
+ \frac{1}{\rho_{reg}^{d-1}} \int_{S^{(2)}_{x_0}} 
\big(1 - \n_{\p^*A}(x)\cdot\n_{\p\A}\big(P_{\p\A}(x_0)\big)\big) \,d\mathcal{H}^{d-1}(x)
\end{aligned}
\end{equation}
for some universal constant $C'_{tilt} \in (1,\infty)$. Furthermore,
adding zero and estimating based on~\eqref{eq:localGraphProperty17}
and~\eqref{eq:flatness4} yields
\begin{equation}
\begin{aligned}
\label{eq:localGraphProperty23}
&\frac{1}{\rho_{reg}^{d-1}} \int_{S^{(2)}_{x_0}} 
\big(1 - \n_{\p^*A}(x)\cdot\n_{\p\A}\big(P_{\p\A}(x_0)\big)\big) \,d\mathcal{H}^{d-1}(x)
\\&~~~
\leq \frac{E_{rel}[A,\mu|\A]}{\rho_{reg}}
+ \frac{1}{\rho_{reg}^{d-1}} \int_{S^{(2)}_{x_0}} 
\big|\xi(x) - \n_{\p\A}\big(P_{\p\A}(x_0)\big)\big| \,d\mathcal{H}^{d-1}(x)
\\&~~~
\leq \frac{E_{rel}[A,\mu|\A]}{\rho_{reg}^{d-1}}
+ \frac{\varepsilon^2}{4}(1 {+} \varepsilon/2)\omega_{d-1}.
\end{aligned}
\end{equation}
In summary, we obtain from the previous three displays that
\begin{align}
\label{eq:localGraphProperty24}
E_{tilt}[x_0,\rho_{reg},T_{x_0}] 
\leq 2\Big(C''_{tilt} + \frac{4(1 {+} \varepsilon/2)}{\varepsilon}\Big)
\frac{E_{rel}[A,\mu|\A]}{\rho_{reg}^{d-1}} + \frac{3}{4}\varepsilon^2
\end{align}
for some universal constant $C''_{tilt} \in (1,\infty)$.
Hence, we may infer~\eqref{eq:localGraphProperty20} from~\eqref{eq:localGraphProperty24}
after suitably selecting $M \geq M_0 \gg_{\A,C_\Lambda} 1$.

\textit{Conclusion:} Since $\varepsilon \leq \varepsilon_{reg}$
by~\eqref{def:epsilon}, the tilt excess estimates
required by~\eqref{eq:smallnessGraph1},
\eqref{eq:assumptionsIntRegularity}
and~\eqref{eq:assumptionsBoundaryRegularity}
as well as the mass ratio estimates required 
by~\eqref{eq:assumptionsIntRegularity}
and~\eqref{eq:assumptionsBoundaryRegularity}
follow from~\eqref{eq:localGraphProperty3}--\eqref{eq:localGraphProperty4},
\eqref{eq:localGraphProperty7}--\eqref{eq:localGraphProperty8}
and~\eqref{eq:localGraphProperty19}--\eqref{eq:localGraphProperty20},
respectively. This eventually concludes the proof.
\end{proof}

\begin{proof}[Proof of Lemma~\ref{lem:structureVarifold}]
The proof is split into two steps.

\textit{Step 1:} 
We first claim that
\begin{align}
\label{eq:noBoundaryIntersection}
\supp\mu \cap \p\Omega = \emptyset.
\end{align}
(Note that the weaker property $|\mu|_{\mathbb{S}^{d-1}}(\p\Omega)=0$ is immediate from
Lemma~\ref{lem:localGraphProperty} and the remark at the end of
the proof of Theorem~\ref{theo:allard_reg}).

For a proof of~\eqref{eq:noBoundaryIntersection}, we argue
by contradiction and assume that there exists $x_0 \in \supp\mu \cap \p\Omega$.
Thanks to Lemma~\ref{lem:localGraphProperty}, the conclusions of Theorem~\ref{theo:allard_reg}
item~ii) apply, i.e., $\supp\mu \cap B_{\gamma_{reg}\rho_{reg}}(x_0)$ admits
a graph representation with associated height function~$u$ in the precise sense
of~\eqref{eq:graphRepVarifold2}--\eqref{eq:estimatesGraphVarifold2}. 
Due to the first estimate of~\eqref{eq:estimatesGraphVarifold2},
we infer that for $\alpha_{reg} \in (0,\frac{\pi}{2})$ defined by
$C_{reg}\varepsilon_{reg}^{\frac{1}{2(d{-}1)+3}} =: \tan(\alpha_{reg})$ it holds
$P_{x_0 {+} \mathrm{Tan}_{x_0}^\perp\p\Omega}(\supp\mu \cap B_{\gamma_{reg}\rho_{reg}}(x_0))
\supset \big((x_0 {+} \mathrm{Tan}_{x_0}^\perp\p\Omega) \cap B_{\gamma_{reg}\rho_{reg}\cos(\alpha_{reg})}(x_0)\big)$,
so that in particular
\begin{align}
\label{eq:contradictionLowerBound}
|\mu|_{\mathbb{S}^{d-1}}\big(B_{\gamma_{reg}\rho_{reg}}(x_0)\big)
\geq \omega_{d-1} \big(\gamma_{reg}\rho_{reg}\cos(\alpha_{reg})\big)^{d-1}.
\end{align}
By~\eqref{eq:min_assumption6}
and~\eqref{eq:no_contact_points}, we also know that 
$|\mu|_{\mathbb{S}^{d-1}}(\supp\mu \cap B_{\gamma_{reg}\rho_{reg}}(x_0))
\leq E_{rel}[A,\mu|\A]$, so that assumption~\eqref{def:goodTimesRepeated2} and
a suitable choice $M \gg_{\A,C_\Lambda} 1$ imply 
\begin{align}
\label{eq:contradictionUpperBound}
|\mu|_{\mathbb{S}^{d-1}}\big(B_{\gamma_{reg}\rho_{reg}}(x_0)\big)
\leq \frac{1}{2} \omega_{d-1} \big(\gamma_{reg}\rho_{reg}\cos(\alpha_{reg})\big)^{d-1},
\end{align}
thus contradicting~\eqref{eq:contradictionLowerBound}.

\textit{Step~2:} By compactness of~$\supp\mu$ and~$\p\Omega$ it now follows from
the first step that $\dist(\supp\mu,\p\Omega) > 0$. Hence, 
the second identity of~\eqref{eq:genMeanCurvatureFullSpace}
is immediate from~\eqref{eq:genMeanCurvatureTangential}, and 
based on that, one obtains~\eqref{eq:repMassMeasureVarifold} simply
from~\eqref{eq:unitDensity} and~\cite[Definition~3, item~(ii)]{Hensel2022}.
Since the canonically associated general varifold~$\widehat{\mu}$ of~$\mu$
is $(d{-}1)$-integer rectifiable, the first identity of~\eqref{eq:genMeanCurvatureFullSpace}
follows now in turn from~\eqref{eq:repMassMeasureVarifold} and 
$|\widehat{\mu}|_{\mathbf{G}(d,d{-}1)}=|\mu|_{\mathbb{S}^{d-1}}$.

It remains to verify the first inclusion of~\eqref{eq:closeToStrongSolution}.
To this end, one may exploit a contradiction argument
being	 essentially analogous to the one conducted in Step~1 of this proof. Indeed,
one works with the conclusions of Theorem~\ref{theo:allard_reg} item~i)
instead of the ones from item~ii), which are applicable due to the already
established validity of~\eqref{eq:genMeanCurvatureTangential} and Lemma~\ref{lem:localGraphProperty},
and exploits in addition property~\eqref{eq:flatness3} for $\rho = \gamma_{reg}\rho_{reg}$ 
to ensure that $|\mu|_{\mathbb{S}^{d-1}}(B_{\gamma_{reg}\rho_{reg}}(x_0)) \leq 4 E_{rel}[\mu,A|\A]$
for any $x_0 \in \{|\xi| \leq 1/2\}$.
\end{proof}

\begin{proof}[Proof of Lemma~\ref{lem:graphCandidate}]
Denote by $(\mathscr{J}_k)_{k=1,\ldots,K}$ the finitely
many connected components of~$\p\A$.
We now argue in four steps.

\textit{Step 1:} We claim that for each $k \in \{1,\ldots,K\}$
it holds
\begin{align}
\label{eq:nontrivialInterface}
\mathcal{H}^{d-1}\big(\p^*A \cap B_{\ell}(\mathscr{J}_k)\big) > 0.
\end{align}
If this is not the case, fix $\bar{k} \in \{1,\ldots,K\}$ with
$\mathcal{H}^{d-1}(\p^*A \cap B_{\ell}(\mathscr{J}_{\bar{k}})) = 0$. Let $\A_{\bar{k}} \subset \A$
such that $\p\A_{\bar{k}}=\mathscr{J}_{\bar{k}}$, so that by~\eqref{eq:closeToStrongSolution}
and~\eqref{eq:min_assumption6}
it holds $$\mathcal{H}^{d-1}(\p^*A \cap \{\dist(\cdot,\A_{\bar{k}})<\ell)\}) = 0.$$ Hence,
by relative isoperimetric inequality $\mathrm{vol}(A \cap \{\dist(\cdot,\A_{\bar{k}})<\ell)\}) = 0$
which in turn implies
\begin{align}
\|\chi_A {-} \chi_{\A}\|_{L^1}^2 \geq \mathrm{vol}(\A_{\bar{k}})^2.
\end{align}
However, it also holds 
\begin{align}
\label{eq:coercivityBulkError}
\|\chi_A {-} \chi_{\A}\|_{L^1}^2 \leq C_{vol}\ell^{d{+}1}E_{vol}[A|\A]
\end{align}
for some universal constant $C_{vol} \in (1,\infty)$. Hence, assumption~\eqref{def:goodTimesRepeated2}
together with the previous two displays guarantees $\min\{\mathrm{vol}(\A_k)^2\colon k=1,\ldots,K\}
\leq \frac{C_{vol}}{M}\ell^{2d}$, so that choosing $M \gg_{\A} 1$ sufficiently large 
yields a contradiction.

\textit{Step 2:} 
Fix $k \in \{1,\ldots,K\}$. We claim that for all $y_0 \in \mathscr{J}_k$
there exists a local height function 
\begin{align}
\label{eq:regularityLocalHeightFunction}
h_{y_0}\colon \mathscr{J}_k \cap B_{\widetilde{C}^{-1}\gamma_{reg}\rho_{reg}}(y_0)
\to (-\ell,\ell) \text{ of class } C^{1,1-\frac{d{-}1}{4}} \cap H^2 
\end{align}
such that
\begin{equation}
\begin{aligned}
\label{eq:step2aux1graph}
&\Big\{y {+} h_{y_0}(y)\n_{\p\A}(y) \colon y \in B_{\widetilde{C}^{-1}\gamma_{reg}\rho_{reg}}(y_0) \cap \mathscr{J}_k\Big\}
\subset\p^*A,
\end{aligned}
\end{equation}
where $\widetilde{C}$ is the constant from~\eqref{def:epsilon}. In particular, for all $y_0 \in \mathscr{J}_k$
\begin{align}
\label{eq:step2aux2graph}
\#\big\{s \in (-\ell,\ell) \colon y_0 {+} s \n_{\p\A}(y_0) \in \p^*A \big\} \geq 1.
\end{align}

For a proof of~\eqref{eq:regularityLocalHeightFunction}--\eqref{eq:step2aux1graph}, we first choose
$\widetilde{x} \in \p^*A \cap B_{\ell}(\mathscr{J}_k)$,
where the latter set is indeed non-empty due to the first step.
By~\eqref{eq:genMeanCurvatureFullSpace} and Lemma~\ref{lem:localGraphProperty}, we may
apply Theorem~\ref{theo:allard_reg} item~i) meaning there exists a $C^{1,1-\frac{d{-}1}{4}}$ function
$u\colon (\widetilde{x} {+} \mathrm{Tan}_{\widetilde{y}}\p\A) 
\cap B_{\gamma_{reg}\rho_{reg}}(\widetilde{x}) 
\to \mathbb{R}$, where $\widetilde{y} := P_{\p\A}(\widetilde{x})$,
such that $\p^*A$ can be represented
within $B_{\gamma_{reg}\rho_{reg}}(\widetilde{x})$ by means of~$u$ in the sense of~\eqref{eq:graphRepVarifold}
and that, thanks to~\eqref{eq:estimatesGraphVarifold} and~\eqref{eq:smallnessGraph1}--\eqref{eq:smallnessGraph2}, 
it holds 
\begin{align}
\label{eq:lipschitzCone}
\sup |\nabla_{\widetilde{x} {+} \mathrm{Tan}_{\widetilde{y}}\p\A} u| \leq 3C_{reg}\varepsilon =: \tan\alpha_{reg}.
\end{align}
In particular,
\begin{align}
\label{eq:domainParametrization}
P_{\widetilde{x} {+} \mathrm{Tan}_{\widetilde{y}}\p\A}\big(\p^*A \cap B_{\gamma_{reg}\rho_{reg}}(\widetilde{x})\big)
\supset \big((\widetilde{x} {+} \mathrm{Tan}_{\widetilde{y}}\p\A) 
\cap B_{(\cos\alpha_{reg})\gamma_{reg}\rho_{reg}}(\widetilde{x})\big).
\end{align}

In preparation for the proof of the regularity claim of~\eqref{eq:regularityLocalHeightFunction},
we first show that
\begin{align}
\label{eq:regularityU}
u \in (C^{1,1-\frac{d{-}1}{4}} \cap H^2)
\big((\widetilde{x} {+} \mathrm{Tan}_{\widetilde{y}}\p\A) 
\cap B_{\frac{1}{4}(\cos\alpha_{reg})\gamma_{reg}\rho_{reg}}(\widetilde{x})\big),
\end{align}
where only the asserted second-order Sobolev regularity requires further attention.
In spirit of Subsection~\ref{subsec:estimate_dissipative_terms}, let us introduce
a change of variables
\begin{equation}
\begin{aligned}
\Psi^u\colon \underbrace{\overbrace{(\widetilde{x} {+} \mathrm{Tan}_{\widetilde{y}}\p\A)}^{\mathrm{T}_{\widetilde{x}}} 
\cap B_{(\cos\alpha_{reg})\gamma_{reg}\rho_{reg}}(\widetilde{x})}_{=: \mathcal{U}_{\widetilde{x}}}
&\to \p^*A \cap B_{\gamma_{reg}\rho_{reg}}(\widetilde{x})
\\
x &\mapsto x + u(x) \n_{\p\A}(\widetilde{y}).
\end{aligned}
\end{equation}
Defining $a\colon (-1,\infty) \to \mathbb{R},\,r\mapsto \frac{1}{\sqrt{1+r}}$,
we will derive the claim $u \in H^2(\frac{1}{4}\mathcal{U}_{\widetilde{x}})$
from two things: (i) within~$\frac{1}{2}\mathcal{U}_{\widetilde{x}}$, $u$ is a weak solution of
\begin{align}
\label{eq:pdeU}
\nabla_{\mathrm{T}_{\widetilde{x}}} 
\cdot \big( a(|\nabla_{\mathrm{T}_{\widetilde{x}}} u|^2)
\nabla_{\mathrm{T}_{\widetilde{x}}} u\big)
= \big(\mathrm{H} \cdot \n_{\p^*A}\big) \circ \Psi^u 
\in L^2\Big(\frac{1}{2}\mathcal{U}_{\widetilde{x}}\Big),
\end{align}
and~$(ii)$ a standard difference quotient argument
applied to~\eqref{eq:pdeU}, where the non-linearity in~\eqref{eq:pdeU} turns out to represent
no obstruction due to an absorption argument facilitated by the estimate~\eqref{eq:lipschitzCone}
and the smallness of~$\varepsilon$, see~\eqref{def:epsilon}. 

In order to recover~\eqref{eq:pdeU},
fix $\varsigma \in C^\infty_{cpt}(\mathrm{T}_{\widetilde{x}})$
with $\supp \varsigma \subset \frac{1}{2}\mathcal{U}_{\widetilde{x}}$,
fix $\overline{\upsilon} \in C^\infty(\mathbb{R};[0,1])$ and~$c_{\alpha_{reg}} \in (0,1)$
with $\supp\overline{\upsilon}  \subset [-(1{+}c_{\alpha_{reg}}),(1{+}c_{\alpha_{reg}})]$
and $\overline{\upsilon} \equiv 1$ on $[-1,1]$ such that the cutoff $(\varsigma\circ P_{\mathrm{T}_{\widetilde{x}}})\upsilon$, where
$\upsilon(x):=\overline{\upsilon}(\frac{\dist(x,\mathrm{T}_{\widetilde{x}})}{\frac{1}{2}(\sin\alpha_{reg})\gamma_{reg}\rho_{reg}})$,
is supported within the ball~$B_{\gamma_{reg}\rho_{reg}}(\widetilde{x})$.
Denoting by $P_{\mathrm{T}_{\widetilde{x}}}$ the nearest-point
projection onto~$\mathrm{T}_{\widetilde{x}}$,
the idea now is to test~\eqref{eq:genMeanCurvatureFullSpace}
with $B_{\varsigma} = (\varsigma\circ P_{\mathrm{T}_{\widetilde{x}}})\upsilon\n_{\p\A}(\widetilde{y})
\in C^\infty_{cpt}(B_{\gamma_{reg}\rho_{reg}}(\widetilde{x}))$
and subsequently represent integrals over~$\p^*A$ as integrals over~$\mathcal{U}_{\widetilde{x}}$.

As $d\mathcal{H}^{d-1} \llcorner (\p^*A \cap B_{\gamma_{reg}\rho_{reg}}(\widetilde{x}))
= \frac{1}{(\n_{\p^*A}\circ\Psi^{u})\cdot\n_{\p\A}(\widetilde{y})} 
d\mathcal{H}^{d-1} \llcorner P_{\mathrm{T}_{\widetilde{x}}}(\p^*A \cap B_{\gamma_{reg}\rho_{reg}}(\widetilde{x}))$
and $\mathrm{H}=(\mathrm{H}\cdot\n_{\p^*A})\n_{\p^*A}$, the right hand side of~\eqref{eq:genMeanCurvatureFullSpace}
for the test function~$B_{\varsigma}$ is given by
\begin{align}
\label{eq:RHSaux}
- \int_{\p^*A} \vec{H} \cdot B_{\varsigma} \,d\mathcal{H}^{d-1}
= - \int_{\frac{1}{2}\mathcal{U}_{\widetilde{x}}} 
\varsigma \big(\mathrm{H} \cdot \n_{\p^*A}\big) \circ \Psi^u \,d\mathcal{H}^{d-1}.
\end{align}
By the properties of the cut-off functions, we further compute on $\p^*A \cap B_{\gamma_{reg}\rho_{reg}}(\widetilde{x})$
\begin{align}
\nabla B = \n_{\p\A}(\widetilde{y}) \otimes (\nabla_{\mathrm{T}_{\widetilde{x}}}\varsigma) \circ P_{\mathrm{T}_{\widetilde{x}}}.
\end{align}
Moreover, throughout $\mathcal{U}_{\widetilde{x}}$ it holds
\begin{align}
\n_{\p^*A}\circ\Psi^u = \frac{1}{\sqrt{1 + |\nabla_{\mathrm{T}_{\widetilde{x}}} u|^2}}
\big(\n_{\p\A}(\widetilde{y}) - \nabla_{\mathrm{T}_{\widetilde{x}}} u\big).
\end{align}
Due to the previous two displays and the already mentioned
representation of the coarea factor, the left hand side of~\eqref{eq:genMeanCurvatureFullSpace}
for the test function~$B_{\varsigma}$ is given by
\begin{align}
\label{eq:LHSaux}
\int_{\p^*A} (\mathrm{Id} {-} \n_{\p^*A}\otimes \n_{\p^*A}):\nabla B_{\varsigma} \,d\mathcal{H}^{d-1}
= \int_{\frac{1}{2}\mathcal{U}_{\widetilde{x}}} 
a(|\nabla_{\mathrm{T}_{\widetilde{x}}} u|^2)
\nabla_{\mathrm{T}_{\widetilde{x}}} u \cdot \nabla_{\mathrm{T}_{\widetilde{x}}} \varsigma\,d\mathcal{H}^{d-1}.
\end{align}
In summary, \eqref{eq:RHSaux} and~\eqref{eq:LHSaux} imply the claim~\eqref{eq:pdeU}.
Performing now a difference quotient argument for the PDE~\eqref{eq:pdeU}
(see, e.g., \cite[Section~4.3]{Giaquinta2012}), one observes that in this scheme the term
for which the difference quotient operates on the coefficient~$a(|\nabla_{\mathrm{T}_{\widetilde{x}}} u|^2)$
is a perturbative one with respect to the $L^2$-norm of the difference quotient of~$\nabla_{\mathrm{T}_{\widetilde{x}}} u$
thanks to~\eqref{eq:lipschitzCone} and the smallness of~$\varepsilon$ (the latter in the form of~\eqref{def:epsilon}).
In other words, one obtains~\eqref{eq:regularityU} by standard arguments.

Based on the properties~\eqref{eq:lipschitzCone}--\eqref{eq:regularityU}
of the auxiliary local height function~$u$,
we now show that~\eqref{eq:regularityLocalHeightFunction}--\eqref{eq:step2aux1graph} holds
for $\widetilde{y} = P_{\p\A}(\widetilde{x})$.
Consider to this end an
arbitrary but fixed $y \in B_{\widetilde{C}^{-1}\gamma_{reg}\rho_{reg}}(\widetilde{y})$.
Because of~\eqref{eq:domainParametrization} as well as the flatness condition~\eqref{eq:flatness5} applied for
$x_0=\widetilde{x}$, $y_0=\widetilde{y}$ and $\rho=\gamma_{reg}\rho_{reg}$,
we may infer that the slice $\{y{+}s\n_{\p\A}(y)\colon s \in (-\ell,\ell)\}$
intersects~$\p^*A$ at least once in the region where the latter is locally
represented by~$u$. By~\eqref{eq:lipschitzCone}, we also know that, for each point~$x$ on the graph
induced by~$u$, the cone with axis~$\mathrm{Tan}_{\widetilde{y}}\p\A$, opening angle~$2\alpha_{reg}$
and apex sitting at~$x$ contains the whole graph induced by~$u$. However, thanks
to the flatness condition~\eqref{eq:flatness6}, this in turn simply means that
the slice $\{y{+}s\n_{\p\A}(y)\colon s \in (-\ell,\ell)\}$
intersects~$\p^*A$ exactly once in the region where the latter is locally
represented by~$u$, so that the local parametrization of~$\p^*A$
in the sense of~\eqref{eq:step2aux1graph} follows. 

Furthermore, we note that~$h_{\widetilde{y}}$ inherits the regularity
of~$u$, see~\eqref{eq:regularityU}, by means of the following argument: the map 
$\iota\colon (\widetilde{x} {+} \mathrm{Tan}_{\widetilde{y}}\p\A)  
\cap B_{(\cos\alpha_{reg})\gamma_{reg}\rho_{reg}}(\widetilde{x}) \to \mathscr{J}_k$
defined by $x \mapsto P_{\p\A}(x {+} u(x)\n_{\p\A}(\widetilde{y}))$ induces
a chart for~$\mathscr{J}_k$, so that the claim follows from the identity 
$h(\iota(x)) = s_{\p\A}(x {+} u(x)\n_{\p\A}(\widetilde{y}))$, 
$x \in \mathcal{U}_{\widetilde{x}}$, and smoothness of the signed distance function~$s_{\p\A}$.
In summary, we proved the claim of Step~2 for $y_0 = \widetilde{y}$.

Fix now $y \in \mathscr{J}_k \setminus\{\widetilde{y}\}$. As~$\mathscr{J}_k$ is connected and compact, 
the Hopf--Rinow theorem ensures that 
we may find a geodesic~$\gamma_{\widetilde{y} \to y}$ connecting the two distinct points
$\widetilde{y},\, y \in \mathscr{J}_k$. We then equipartition~$\gamma_{\widetilde{y} \to y}$
into~$N$ geodesic segments~$\gamma_{\widetilde{y}_{n-1} \to \widetilde{y}_{n}}$, $n\in\{1,\ldots,N\}$, 
where $\widetilde{y}_0 := \widetilde{y}$, $\widetilde{y}_N := y$ and $N \in \mathbb{N}$
is sufficiently large such that $\widetilde{y}_n \in B_{\widetilde{C}^{-1}\gamma_{reg}\rho_{reg}}(\widetilde{y}_{n-1}) \cap
\mathscr{J}_k$ for all $n \in \{1,\ldots,N\}$. Because of the latter condition, 
we may first transfer the conclusion of the previous argument from the starting point $\widetilde{y}_0 = \widetilde{y}$ to 
the next point~$\widetilde{y}_1$, and thus also, after finitely many iterations, to the endpoint $\widetilde{y}_N = y$.
This concludes the proof of~\eqref{eq:step2aux1graph}.
 
\textit{Step~3:} Fix again $k \in \{1,\ldots,K\}$. We now show that for all $y_0 \in \mathscr{J}_k$ it holds
\begin{align}
\label{eq:step3aux1graph}
\#\big\{s \in (-\ell,\ell) \colon y_0 {+} s \n_{\p\A}(y_0) \in \p^*A \big\} = 1.
\end{align}

We assume by contradiction that there exists $y_0 \in \mathscr{J}_k$ such that~\eqref{eq:step3aux1graph}
fails, i.e., by~\eqref{eq:step2aux2graph}, that
\begin{align}
\label{eq:step3aux2graph}
\#\big\{s \in (-\ell,\ell) \colon y_0 {+} s \n_{\p\A}(y_0) \in \p^*A \big\} > 1.
\end{align}
Hence, we may choose $x_0 \in (\p^*A \cap B_{\ell}(\mathscr{J}_k)) 
\setminus \{\widetilde{x}_0\}$, where we also defined $\widetilde{x}_0 := y_0 + h_{y_0}(y_0) \n_{\p\A}(y_0)$.
Around both points~$x_0$ and~$\widetilde{x}_0$, the local graph property of~$\p^*A$
in the precise sense of Theorem~\ref{theo:allard_reg} item~i) holds true
(with respect to the same tangent space). In particular,
$\p^*A \cap B_{\gamma_{reg}\rho_{reg}}(x_0)$ is a subset of the points within~$\p^*A \cap B_{\rho_{reg}}(x_0)$
where~$\p^*A$ is hit more than once by slicing in normal direction relative to~$y_0 {+} \mathrm{Tan}_{y_0}\p\A$
(or more precisely, relying on the notation introduced right before~\eqref{eq:localGraphProperty10},
$\p^*A \cap B_{\gamma_{reg}\rho_{reg}}(x_0) \subset S_{x_0}^{(1)}$). Hence, by a slicing argument,
\begin{align}
\label{eq:step3aux3graph}
\mathcal{H}^{d-1}(\p^*A \cap B_{\gamma_{reg}\rho_{reg}}(x_0))
\leq C_{rel} E_{rel}[A,\mu|\A] 
\end{align}
for some universal constant $C_{rel} \in (1,\infty)$. Based on a
suitably large choice of $M \gg_{\A,C_\Lambda} 1$, one may now
obtain a contradiction analogously to Step~1 of the proof of Lemma~\ref{lem:structureVarifold}
(i.e., by the analogues of~\eqref{eq:contradictionLowerBound}
and~\eqref{eq:contradictionUpperBound}).

\textit{Step~4:} Fix again $k \in \{1,\ldots,K\}$. By compactness of~$\mathscr{J}_k$,
we may select finitely many points $(\widetilde{y}_n)_{n=1,\ldots,N}$
along~$\mathscr{J}_k$ such that $\mathscr{J}_k
\subset \bigcup_{n=1}^N B_{\widetilde{C}^{-1}\gamma_{reg}\rho_{reg}}(\widetilde{y}_n) \cap \mathscr{J}_k$.
It then follows from the previous two steps that the local graph representations at 
$\widetilde{y}_1,\ldots,\widetilde{y}_N$ from Step~2 can be stitched together
in a consistent manner by a subordinate partition of unity. In other words,
there exists $J_k \subset \p^*A$ with the property that~$J_k$ can be obtained
globally as a graph over~$\mathscr{J}_k$ in the sense of~\eqref{eq:graph_representation}
with regularity~\eqref{eq:regularity_graph}.

We may now conclude because $\p^*A = \bigcup_{k=1}^{K} J_k$ due to Step~3.
\end{proof}

\begin{proof}[Proof of Lemma~\ref{lem:graphRep}]
We proceed in two steps.

\textit{Step 1: Estimate for~$\sup|h|$.}
Assume that there exists $x_0 \in \p^*A$ such that
\begin{align}
\label{eq:lowerBoundHeight}
\big|h(P_{\p\A}(x_0))\big| > \frac{\ell}{16 C}.
\end{align}
By~\eqref{eq:genMeanCurvatureFullSpace} and Lemma~\ref{lem:localGraphProperty}, we may
again apply Theorem~\ref{theo:allard_reg} item~i) to write $\supp\mu=\p^*A$ within 
$B_{\gamma_{reg}\rho_{reg}}(x_0)$ as a graph over $\mathrm{T}_{x_0} := x_0 + \mathrm{Tan}_{P_{\p\A}(x_0)}\p\A$
with height function~$u$, cf.\ \eqref{eq:graphRepVarifold}. Because of~\eqref{eq:estimatesGraphVarifold},
\eqref{eq:smallnessGraph1}--\eqref{eq:smallnessGraph2}, 
and~\eqref{def:epsilon}--\eqref{def:rho}
\begin{align}
\label{eq:estimateGraph1}
\sup |u| &\leq \rho 2 C_{reg} \varepsilon \leq \frac{1}{4} \frac{\ell}{16 C},
\\
\label{eq:estimateGraph2}
\sup |\nabla_{\mathrm{T}_{x_0}} u| &\leq 2 C_{reg} \varepsilon =: \tan\alpha,\,\alpha \in (0,\pi/2).
\end{align}
Together with the previous two estimates, the flatness property~\eqref{eq:flatness1}
implies that
\begin{align}
\|\chi_{A} {-} \chi_{\A}\|_{L^1} \geq \omega_{d-1}
\big(\gamma_{reg}\rho_{reg}\cos\alpha\big)^{d-1}
\frac{1}{2} \frac{\ell}{16 C}.
\end{align}
In combination with the coercivity estimate~\eqref{eq:coercivityBulkError}
and assumption~\eqref{def:goodTimesRepeated2}, we may therefore choose $M \gg_{\A,C_\Lambda} 1$
to reach a contradiction.

\textit{Step 2: Estimate for~$\sup|\nabla_{\p\A} h|$.} 
Consider $y_0 \in \p\A$ and define $x_0 := y_0 + h(y_0) \n_{\p\A}(y_0) \in \p^*A$.
Denote by~$\ta_1$ and~$\ta_2$ a choice of orthonormal principal curvature directions
at~$y_0$ with corresponding principal curvatures~$\kappa_1$ and~$\kappa_2$,
and let $\widetilde{\ta}_1$ and~$\widetilde{\ta}_2$ be the unique
unit length tangent vectors of~$\p^*A$ at~$x_0$ such that 
$\ta_i \cdot \widetilde{\ta}_i > 0$
and $\widetilde{\ta}_i \in \{\alpha \ta_i + \beta \n_{\p\A}(y_0)\colon \alpha,\beta \in \mathbb{R}\}$  for both $i=1,2$.
Fix now $i \in \{1,2\}$.

Locally around~$y_0$, we may choose a curve 
$\gamma_i\colon (-\delta,\delta) \to \p\A$
such that $\gamma_i(0) = y_0$ and $(\gamma_i)'(0) = \eta_i$. 
This also induces a curve in~$\p^*A$ by
\begin{align}
&\gamma_i^h\colon (-\delta,\delta) \to \p^*A, 
&& \theta \mapsto \gamma_i(\theta) + h(\gamma_i(\theta)) \n_{\p\A}(\gamma_i(\theta)).
\end{align}
A straightforward computation yields
\begin{align}
(\gamma_i^h)'(0) = \big(1 - h(y_0)\kappa_i\big) \ta_i + (\ta_i \cdot\nabla_{\p\A}) h(y_0) \n_{\p\A}(y_0)
\end{align}
and
\begin{align}
\frac{(\gamma_i^h)'(0)}{|(\gamma_i^h)'(0)|} \cdot \ta_i = 
\frac{1}{\sqrt{1 + \big(\frac{(\ta_i \cdot\nabla_{\p\A}) h(y_0)}{1 - h(y_0)\kappa_i}\big)^2}}.
\end{align}
In particular, 
\begin{align}
\frac{(\gamma_i^h)'(0)}{|(\gamma_i^h)'(0)|} = \widetilde{\ta}_i.
\end{align}

Instead of using the height function~$h$, we may also represent~$\p^*A$
locally around~$x_0$ by means of a height function~$u$ in the sense of~\eqref{eq:graphRepVarifold}.
Defining $\gamma_{x_0,i}(\theta) := x_0 + \theta\ta_i$, we then obtain a second induced curve in~$\p^*A$ by
\begin{align}
&\gamma_i^u\colon (-\delta,\delta) \to \p^*A, 
&& \theta \mapsto \gamma_{x_0,i}(\theta) + u(\gamma_{x_0,i}(\theta)) \n_{\p\A}(y_0),
\end{align}
for which one may compute
\begin{align}
(\gamma_i^u)'(0) = \ta_i + (\ta_i \cdot\nabla_{\mathrm{T}_{x_0}}) u(x_0) \n_{\p\A}(y_0)
\end{align}
and
\begin{align}
\frac{(\gamma_i^u)'(0)}{|(\gamma_i^u)'(0)|} \cdot \ta_i = 
\frac{1}{\sqrt{1 + \big((\ta_i \cdot\nabla_{\mathrm{T}_{x_0}}) u(x_0)\big)^2}},
\end{align}
so that again
\begin{align}
\frac{(\gamma_i^u)'(0)}{|(\gamma_i^u)'(0)|} = \widetilde{\ta}_i.
\end{align}

In total, we obtain from these two different ways of representing~$\widetilde{\ta}_i$ that
\begin{align}
\big|(\ta_i \cdot\nabla_{\p\A}) h(y_0)\big|
= \big|1 - h(y_0)\kappa_i\big| \big|(\ta_i \cdot\nabla_{\mathrm{T}_{x_0}}) u(x_0)\big|.
\end{align}
By the previous step, $|h(y_0)| \leq \frac{1}{16}\ell$, 
so that the estimate $|\kappa_i| \leq \frac{1}{\ell}$
together with the previous identity, \eqref{eq:estimateGraph2} 
and~\eqref{def:epsilon} entails
\begin{align}
\big|(\ta_i\cdot \nabla_{\p\A})h(y_0)\big| 
\leq \frac{17}{16} \big|(\ta_i\cdot\nabla_{\mathrm{T}_{x_0}}) u(x_0)\big| \leq \frac{1}{16 C}.
\end{align}
This concludes the proof.
\end{proof}

\begin{proof}[Proof of Proposition~\ref{prop:perturbative_graph_regime}]
Immediate from Lemma~\ref{lem:structureVarifold}, Lemma~\ref{lem:graphCandidate} and Lemma~\ref{lem:graphRep}.
\end{proof}

\section*{Acknowledgments}
Parts of the paper were developed during the visit of the authors to the Hausdorff Research Institute for Mathematics (HIM), University of Bonn, in the framework of the trimester program ``Evolution of Interfaces''. The support and the hospitality of HIM are gratefully acknowledged.
This project has also received funding from the European Union's Horizon 2020 research and 
innovation programme under the Marie Sk\l{}odowska-Curie Grant Agreement No.\ 665385 
\begin{tabular}{@{}c@{}}\includegraphics[width=3ex]{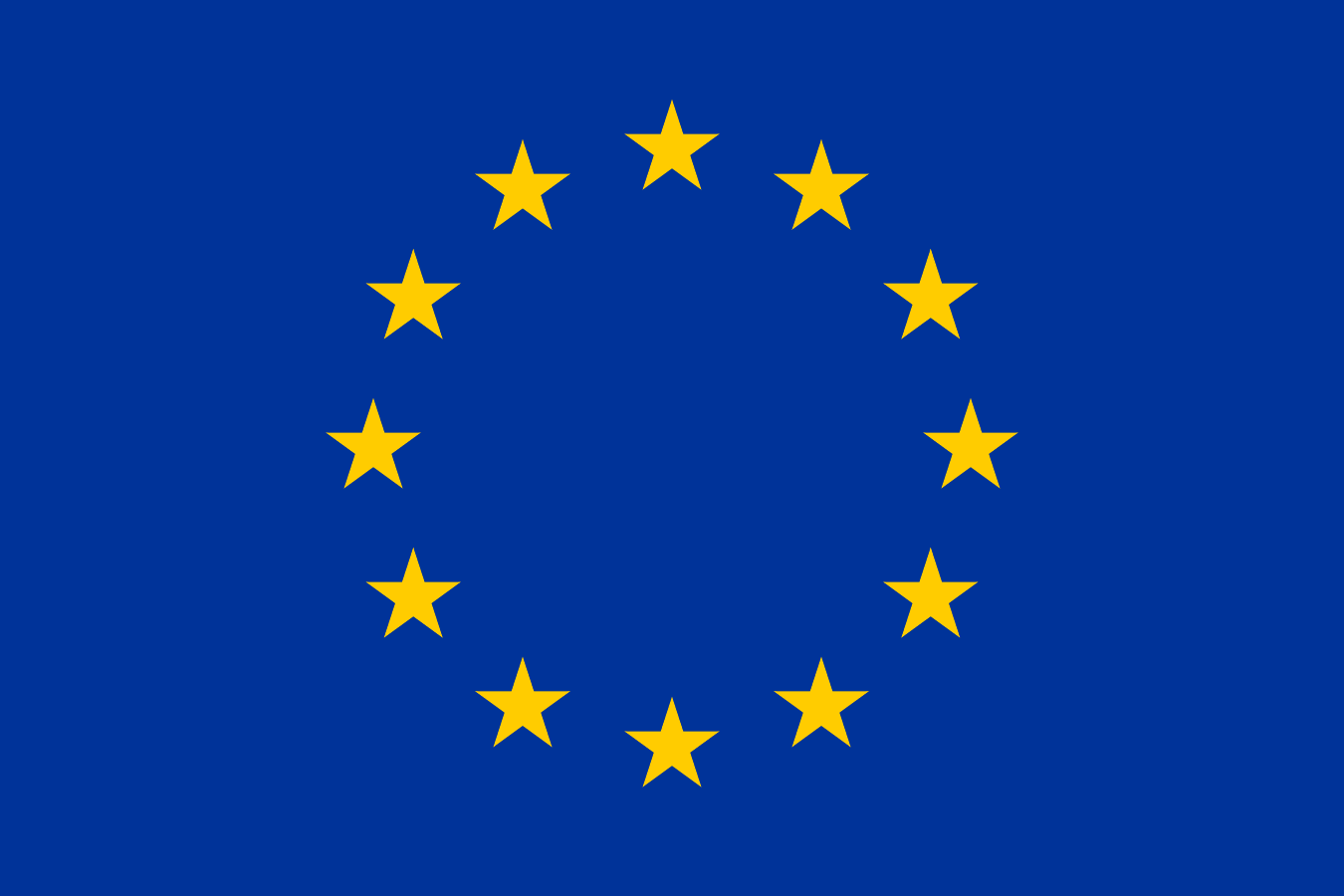}\end{tabular}, from the
European Research Council (ERC) under the European Union's Horizon 2020
research and innovation programme (grant agreement No 948819)
\smash{
\begin{tabular}{@{}c@{}}\includegraphics[width=6ex]{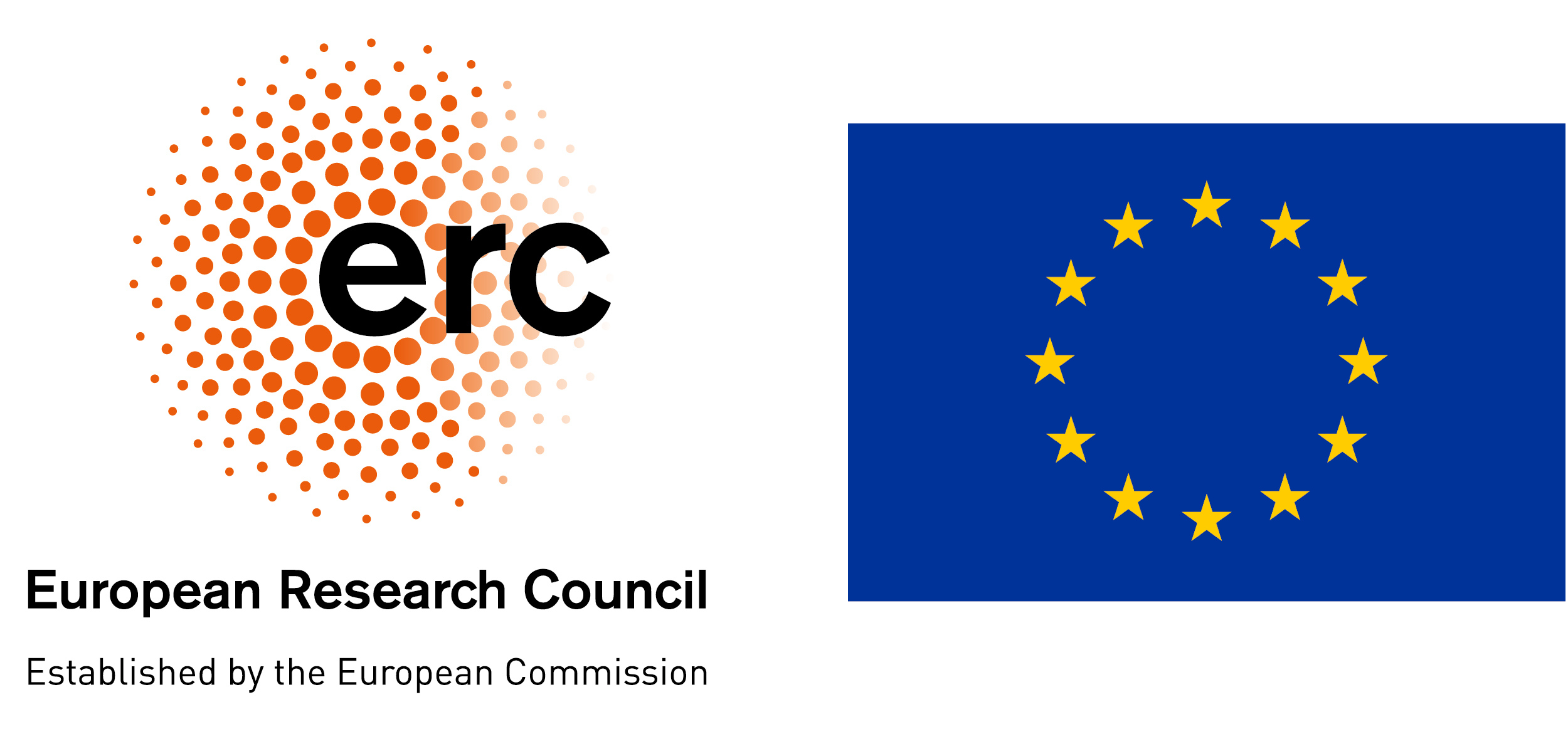}\end{tabular}
},
and from the Deutsche Forschungsgemeinschaft (DFG, German Research Foundation) 
under Germany's Excellence Strategy -- EXC-2047/1 -- 390685813
and EXC 2044--390685587, Mathematics M\"{u}nster: Dynamics--Geometry--Structure.

\bibliographystyle{abbrv}
\bibliography{mullins_sekerka_uniqueness}

\end{document}